\documentclass{amsart}
\usepackage{amssymb, enumerate}

\usepackage[dvipsnames]{xcolor}
\usepackage[all,pdf]{xy}

\usepackage{amssymb}
\usepackage{amsthm}
\usepackage{amsmath}
\usepackage{paralist}
\usepackage{graphics}
\usepackage{epsfig}
\usepackage[colorlinks=true]{hyperref}
\hypersetup{urlcolor=blue, citecolor=red}

\usepackage{graphicx}

\newcommand{\dvol} {\,d\operatorname{vol}}
\newcommand\Symb{{\operatorname{Symb}}}
\newcommand\ADN{{Douglis-Nirenberg}}

\newtheorem{theorem}{Theorem}[section]
\newtheorem{lemma}[theorem]{Lemma}
\newtheorem{corollary}[theorem]{Corollary}
\newtheorem{proposition}[theorem]{Proposition}

\theoremstyle{definition}
\newtheorem{definition}[theorem]{Definition}
\newtheorem{example}[theorem]{Example}
\newtheorem{notation}[theorem]{Notation}

\theoremstyle{remark}
\newtheorem{remark}[theorem]{Remark}

\usepackage[normalem]{ulem}
\renewcommand\sout{\bgroup\markoverwith
{\textcolor{red} {\rule[0.7ex]{3pt}{1.4pt}}}\ULon}

\newcommand\seq{\, = \,}
\newcommand\define{\mathrel{\ := \ }}
\newcommand\ede{\define}

\newcommand\dist{\operatorname{dist}}
\newcommand\Hom{\operatorname{Hom}}
\newcommand\End{\operatorname{End}}
\newcommand\supp{\operatorname{supp}}

\newcommand{\CC} {\mathbb C}

\newcommand{\NN} {\mathbb N}

\newcommand{\RR} {\mathbb R}

\newcommand{\ZZ} {\mathbb Z}

\newcommand{\maB} {\mathcal B}
\newcommand{\maC} {\mathcal C}
\newcommand{\maD} {\mathcal D}

\newcommand{\maF} {\mathcal F}

\newcommand{\maH} {\mathcal H}

\newcommand{\maR} {\mathcal R}
\newcommand{\maS} {\mathcal S}

\newcommand{\mfkL} {\mathfrak L}
\newcommand{\mfkM} {\mathfrak M}
\newcommand\manif{{\mfkM}}
\newcommand\link{{\mfkL}}

\newcommand{\CI}{{\mathcal C}^{\infty}}
\newcommand{\CIc}{{\mathcal C}^{\infty}_{\text{c}}}

\newcommand\pa{\partial}

\newcommand\inv{\operatorname{inv}}
\newcommand\iPS[1]{{\Psi_{\operatorname{inv}}^{#1}}}
\newcommand\iPSsus[2]{{\psi_{\operatorname{sus}}^{#1}(#2)}}
\newcommand\ePS[1]{{\Psi_{\operatorname{ess}}^{#1}}}
\newcommand\ePSsus[2]{{\Psi_{\operatorname{sus}}^{#1}(#2)}}

\newcommand\In{{\maR_\infty}}

\newcommand\m[1]{{$#1$}}
\newcommand\dm[1]{{$#1$}}

\numberwithin{equation}{section}

\begin{document}

\title[Layer potentials and pseudodifferential operators on manifolds]{Layer potentials and 
essentially translation invariant pseudodifferential operators on manifolds with cylindrical ends}


\author[M. Kohr]{Mirela Kohr}
\address{Faculty of Mathematics and Computer Science,
Babe\c{s}-Bolyai University, 1 M. Kog\u{a}l\-niceanu Str., 400084
Cluj-Napoca, Romania} \email{mkohr@math.ubbcluj.ro}

\author[V. Nistor]{Victor Nistor}
\address{Universit\'{e} de Lorraine, CNRS, IECL, F-57000 Metz, France}
\email{victor.nistor@univ-lorraine.fr}

\author[W.L. Wendland]{Wolfgang L. Wendland}
\address{Institut f\"ur Angewandte Analysis und Numerische Simulation,
Universit\"at Stuttgart, Pfaffenwaldring 57, 70569 Stuttgart,
Germany}
\email{wendland@mathematik.uni-stuttgart.de}

\thanks{M.K. has been partially supported by the
  Babe\c{s}-Bolyai University Grant AGC 33656/21.07.2022. V.N. has been partially supported by
  ANR OpART.}

\date\today


\date{\today}


\begin{abstract}
  Motivated by the study of layer potentials on manifolds with straight 
  conical or cylindrical ends, we introduce and study two classes (or calculi) 
  of pseudodifferential operators defined on manifolds with cylindrical ends: 
  the class of pseudodifferential operators that are ``translation invariant 
  at infinity'' and the class of ``essentially translation invariant operators.'' 
  These are ``minimal'' classes of pseudodifferential 
  operators containing the layer potential operators of interest. Both classes 
  are close to the $b$-calculus considered by Melrose and Schulze and to the 
  $c$-calculus considered by Melrose and Mazzeo-Melrose. Our calculi, however, 
  are different and, while some of their properties follow from those of the 
  $b$- or $c$-calculi, many of their properties do not. In particular, we prove 
  that the ``essentially translation invariant calculus'' is spectrally invariant, 
  a property not enjoyed by the ``translation invariant at infinity'' calculus or 
  the $b$-calculus. For our calculi, we provide easy, intuitive proofs of the usual 
  properties: stability for products and adjoints, mapping and boundedness properties 
  for operators acting between Sobolev spaces, regularity properties, existence of a 
  quantization map, topological properties of our algebras, and the Fredholm property. 
  Since our applications will be to the Stokes operator, we systematically work in the 
  setting of \ADN-elliptic operators. We also show that our calculi behave well
  with respect to restrictions to (suitable) submanifolds, which is crucial
  for our applications to layer potential operators.
\end{abstract}

\maketitle
\tableofcontents

\section{Introduction} \label{sec:1}

Boundary value problems play a very prominent role in the applications of
mathematics to other areas of science and engineering. The analysis of boundary value
problems is mainly based on \emph{two approaches.} The most commonly
used one relies on energy methods. The second one is based on layer potentials,
which are pseudodifferential operators if the boundary is nice enough.
While less elementary than the energy method approach, the layer potentials
approach to boundary value problems possesses several advantages.

This paper is the second paper in a series of papers in which \emph{we study the
method of layer potentials on manifolds with cylindrical ends.} To this end, we need to
develop two precise calculi of pseudodifferential operators that would contain
our differential operators and their associated layer potential operators, the 
``$\inv$-calculus'' and the ``essentially translation invariant at infinity calculus''. 
These two calculi were originally introduced in
\cite{Mitrea-Nistor}, with the same motivation of studying the layer potential
operators on manifolds with cylindrical ends. In this paper, we complete and
significantly extend the results of that paper, providing also an alternative
approach to several results in that papers. We thus provide a more systematical
account of the properties of these two calculi. We also more carefully indicate
the relations between our calculi and the closely related $b$- and $c$-calculi
\cite{MazzeoMelroseAsian, MelroseActa, MelroseAPS, SchulzeBook91, Schulze} and 
to the $c$-calculus of Melrose and Mazzeo-Melrose. We also work consistently in the
framework of \ADN-elliptic operators.

\subsection{The framework and some earlier results}
A manifold with \emph{straight cylindrical ends} is a Riemannian manifold $(\manif, g)$
isometrically diffeomorphic to one of the form
\begin{equation}\label{eq.def.can.form}
  \manif \seq \manif_0 \cup (\pa \manif_0 \times (-\infty, 0])\,,
\end{equation}
where $\manif_0$ is a smooth  manifold with boundary, the ``end'' $\pa \manif_0 \times (-\infty, 0]$
is endowed with a product metric, and we identify $\pa \manif_0$ with $\pa \manif_0 \times \{0\}$.
Most of the time $\manif_0$ will be compact, but, in order to be able to include $T^*\manif$
in our setting, we allow also the case $\manif_0$ non-compact.
See Definition \ref{def.ce} and Remark \ref{rem.def.ce} for details.
In our applications to layer potentials, $\manif$ will be, in fact, the \emph{boundary}
of a manifold with boundary and straight cylindrical ends $\Omega$, on which our boundary
value problem will be posed. Manifolds with straight cylindrical ends (for which our calculi were designed)
are important because they can be used to study manifolds
with \emph{straight conical} points via the Kondratiev transform $r = e^t$ ($r$ being the
distance to the conical points and $t$ the coordinate on the cylindrical end). Manifolds with
straight conical points are, of course, less general than manifolds with (plain) conical points,
but they appear often enough in applications to justify an explicit study, which yields, in
particular, stronger results than the ones available in the general case.

In view of our applications, we develop two pseudodifferential
calculi on $\manif$ that are specifically tailored to manifolds
with straight cylindrical ends. The two calculi are
\begin{equation*}
  \iPS{m}(\manif) \ \subset \ \ePS{m}(\manif)\,,
\end{equation*}
and are both closely related to the $b$-calculus of Melrose \cite{MelroseActa, MelroseAPS}
and Schulze \cite{SchulzeBook91, Schulze} and to the $c$-calculus of Melrose and
Mazzeo-Melrose \cite{MazzeoMelroseAsian}.
The first calculus is called the ``inv''-calculus and consists of pseudodifferential
operators that are translation invariant in a neighborhood of infinity on $\manif$. The
second calculus is called the ``essentially translation invariant''-calculus and
is a spectrally invariant enlargement of the ``inv''-calculus.

The motivation for considering these two calculi is the following. The ``inv''-calculus
is the smallest calculus that we could imagine and contains the operators we are interested
in on manifolds with straight cylindrical ends.
Saying that an operator belongs to this calculus is a strong statement, stronger than
saying that it belongs to the $b$-calculus (which is larger than the ``inv''-calculus).
In general, it is convenient in applications to use the smallest calculus that does the
job. This justifies considering the ``inv''-calculus.
The ``inv''-calculus is, however, not spectrally invariant, which then motivated us to consider
the slightly larger calculus  of ``essentially translation invariant'' operators,
which we prove to be spectrally invariant. Again, the resulting algebra is one of the smallest
one that is spectrally invariant and is adapted to the geometry, which again justifies its study.

As we have mentioned above, our calculi are closely related to the $b$ and $c$-calculi
\cite{MazzeoMelroseAsian, MelroseActa, MelroseAPS, SchulzeBook91, Schulze}, so we hope
that our results will help the reader better understand those calculi as well. See also
\cite{CiprianaThesis, GrieserBCalc, LauterSeiler, LeschBCalc, SchulzeWongBCalc}. Versions
of ``inv'' and ``essentially translation invariant'' calculi were introduced and briefly studied
in \cite{Mitrea-Nistor}. Here we provide a more general version of these calculi and a more
systematic account. For our applications, we were especially interested in the spectral
invariance property and the Fredholm property (Theorem \ref{thm.spectral.inv} and
\ref{thm.ADN.Fredholm}). There was a lot of work on spectral invariance properties
for pseudodifferential operators. A few more recent papers include \cite{HAbelsSpi, CoriascoToft,
DasgupaWongSpinv, ToftSpi, LMNsi, MazzeoMelroseAsian, SchroheFrechet}.
There was also a lot of work on Fredholm conditions for pseudodifferential
operators. A few more recent publications include \cite{CNQ, DLR, FanWongFredholm,
Kapanadze-Schulze, KottkeRochonFibered, MazzeoMelroseAsian, Melrose-Mendoza, SchroheSpInvFred,
SchroheFrechet, SchulzeBook91, Schulze}. See the aforementioned publications for further references.

\subsection{Main results.}
We first develop and study the general properties of two pseudodifferential calculi 
mentioned above on $\manif$, the ``inv'' and the ``essentially translation invariant''
calculi ($\iPS{m}(\manif)$ and  $\ePS{m}(\manif)$) using symbols in the classes 
$S_{1,0}^m$ \cite{HAbelsBook, Hormander3, RT-book2010, Taylor2, TrevesBookPsdo, Wong}.
We prove that these calculi are stable under products and under adjoints. We show that 
they induce bounded operators between suitable Sobolev spaces. We prove that they enjoy 
the usual symbolic properties, which leads to regularity results. We introduce and study 
natural topologies on these spaces of pseudodifferential operators. We provide a 
description of the distribution kernels of the classical, negative order pseudodifferential 
operators in our calculi, which is crucial for our applications to layer potentials. 
We also extend the concept of normal operator (which we
call ``limit operator'' for historical reasons) to our calculi and use it to study
the structure of our algebras. These results on the ``inv''-calculus can be found in Section
\ref{translation-inv-oper} and those about the ``essentially translation invariant''-calculus
can be found in Section \ref{ess-trans-inv}. In these two sections, we work mostly with
operators whose symbols belong to the H\"ormander classes $S_{1,0}^m$.


The results of these two sections are important for establishing our three main results
needed in applications: compatiblity with restrictions to submanifolds
(Theorem \ref{thm.kernel.restriction}), the spectral invariance (Theorem \ref{thm.spectral.inv})
and the Fredholm property (Theorem \ref{thm.ADN.Fredholm}). We prove all of the three
main results in the framework of \ADN-elliptic operators, which is
another novelty of our approach. This is an important feature since our main planned
application is to the Stokes operator. These three main results, however, are proved in
the setting of classical pseudodifferential operators. (The general case is significantly
different and including it would have made this paper significantly longer.) These
results are crucial in our applications to the study of layer potentials, in particular,
these results allow us to introduce the layer potential operators associated to an 
invertible classical pseudodifferential operator of low order.
The proof that our calculi behave well with respect to restrictions to (suitable)
submanifolds is the reason for our approach to the pseudodifferential calculi via
distribution kernels.

It can be proved (but it is non-trivial) that the $\inv$-calculus is a subset of the $b$-calculus
and that the essentially translation invariant calculus is a subset of the $c$-calculus. 
(Checking this is an excellent exercise for understanding the $b$-calculus and the differences 
between the $b$- and $c$-calculi). Some of the properties of our calculi would follow then from the
aforementioned inclusions. For our purposes, however, it was more convenient to include direct proofs. 
That makes our paper easier to follow by people with a limited knowledge of pseudodifferential operators
and, hopefully, will make pseudodifferential operators more accessible.
Our paper can also serve as an introduction to the $b$- and $c$-calculi. Most importantly, our
approach is useful for studying restrictions to submanifolds (Theorem \ref{thm.kernel.restriction}),
which is needed for the definition of the layer potential operators.

\subsection{Contents of the paper}
The paper is organized as follows. Section \ref{sec.background} contains definitions and
results on manifolds with \emph{straight} cylindrical ends and the corresponding Sobolev
spaces. The manifolds with straight cylindrical ends play a central role in our developments.
In Section \ref{sec.backgr.psdos}, we recall a few basic results on pseudodifferential
operators (the reader familiar with the topic can skip this section at a first
reading). In Section \ref{translation-inv-oper}, we recall and study the definition and the
properties of the calculus $\iPS{m}(\manif)$ of pseudodifferential operators that are invariant
in a neighborhood of infinity (the ``inv''-calculus), including a description of the
distribution kernel of the classical operators in this calculus. In Section \ref{ess-trans-inv},
we introduce and study the calculus $\ePS {\infty} (\manif)$ of essentially translation invariant
(in a neighborhood of infinity) pseudodifferential operators. This is a suitable enlargement
of $\iPS{m}(\manif)$ that is stable under inversion of elliptic operators of non-negative order.
This stability result is contained in Theorem \ref{thm.spectral.inv} in Section \ref{sec.SPInvFredholm},
where we also prove the restriction to submanifold property (Theorem \ref{thm.kernel.restriction})
and the Fredholm property for these operators (Theorem \ref{thm.ADN.Fredholm}). A novelty of our
approach is that we prove our three basic results (restriction to submanifolds,
the spectral invariance and the Fredholm property) in the framework
of \ADN-elliptic operators. The \ADN-ellipticity and related regularity results are discussed in
Section \ref{sec.six}.
We expect our methods to extend to other classes of non-compact manifolds.

We thank Cipriana Anghel, Nadine Gro\ss e, Massimo Lanza de Cristoforis, and Sergiu Moroianu
for useful discussions. We also thank Joerg Seiler and Elmar Schrohe for useful discussions and
for sending us their papers. The general introduction to pseudodifferential operators is based on
our lectures given at the ``Summer school in Geometric Analysis''  at \emph{Universit\'e Libre de
Bruxelles,} August-September 2023, school organized by Haydys Andriy and Joel Fine whom we
thank for their kind invitation.

\section{Background: manifolds with cylindrical ends and Sobolev spaces}
\label{sec.background}

We include in this background section several definitions and results
on manifolds with cylindrical ends and their Sobolev spaces. As explained in the
Introduction, we shall consider only the simplest class of manifolds
with cylindrical ends, the manifolds with ``straight cylindrical ends,''
because of its applications and because in this case we can obtain stronger results.
Let us add, nevertheless, that the main difference
between the various classes of manifolds with cylindrical ends is the choice of the
relevant algebra of pseudodifferential operators: choosing a smaller algebra leads to
stronger results on the layer potential operators, but if the algebra is too
small, it will not contain the relevant layer potential operators. In any case,
these algebras will be closely related to the $b$-calculus. There are many further
references on the $b$-calculus, see, for instance
\cite{GilLoya, GrieserBCalc, LauterSeiler, LeschBCalc, SchSch1}.
The general case of manifolds with cylindrical ends requires a different
pseudodifferential calculus and will be treated in a subsequent joint work with Cipriana
Anghel and Sergiu Moroianu, building also on the results on the $c$-calculus in
\cite{CiprianaThesis}.

\subsection{General notation: duals and connections}
We shall use the following conventions for the dual spaces that we use. First, $V'$
denotes the dual space of a real or complex topological vector space $V$ and, if
$T : V \to W$ is a linear map, then $T' : W' \to V'$ is its dual. Similarly,
if $V$ and $W$ are complex vector spaces endowed with inner products, then
$T^* : W \to V$ is the adjoint of $T : V \to W$. In case of real vector spaces
(still endowed with an inner product), instead of the adjoint we have the {\em transpose}
$T^\top : W \to V$. By $V^*$ we shall denote the complex conjugate of $V'$.
When there is no danger of confusion, we use the notation $1$ for the identity operator
of any given space.

Let $M$ be a smooth manifold.
The space of smooth sections of a smooth vector bundle $E \to M$ will be
denoted $\CI(M; E)$ and the space of smooth sections that in
addition have {\em compact} support will be denoted $\CIc(M; E)$.
From now on, $(\manif, g)$ will be a Riemannian manifold. We
can use the fixed metric on $\manif$ to trivialize the density bundle on $\manif$
and hence identify the dual of the space $\CIc (\manif)$ with
a space of distributions on $\manif$. As usual,
$T\manif \to \manif$ is the tangent bundle to $\manif$ and $T^*\manif \to \manif$ is the
cotangent bundle to $\manif$ (the dual of $T\manif$). On these bundles, as well as on their tensor
products, we shall consider the Levi-Civita connection $\nabla^{LC}$.
All vector bundles $E$ considered in this paper will be finite dimensional, smooth, Hermitian,
and will be endowed with metric compatible connections $\nabla^E$. (We shall write
simply $\nabla$ instead of $\nabla^E$ when there is no risk of confusion, that is most
of the time, since there will be typically only one connection on each vector bundle.)
A \emph{closed manifold} is a smooth, compact manifold without boundary.

\subsection{Manifolds with cylindrical ends}
Let $(\link, g_{\link})$ be a Riemannian manifold and $I \subset \RR$.
Unless stated otherwise, we shall endow $\link \times I$ with the product metric
\begin{equation}\label{eq.prod.metric}
  g_{\link \times I} \seq g_{\link} + (dt)^2\,,
\end{equation}
where $(dt)^2$ is the standard Euclidean metric on $I$.
Throughout this paper, $\manif$ will be a Riemannian manifold
(in fact, for the most part, a manifold with straight cylindrical ends).
The metric on $\manif$ will be denoted simply $g$ and $\dist_g$
will be the distance function on $\manif$ defined by $g$.
From now on, $\link$ will denote a closed manifold (i.e. smooth,
compact, without boundary) endowed with a metric.

\begin{definition} \label{def.ce}
  We say that $\manif$ has \emph{straight cylindrical ends} if it is geodesically
  complete and there exists a Riemannian
  manifold $\link \neq \emptyset$ and an isometry $\psi_{\manif} : \link
  \times (-\infty, 0) \to \manif$ whose complement is a  manifold
  with boundary $\manif_0$. (On $\link \times (-\infty, 0)$ we consider the product metric, as usual.)
  If $\manif_0$ is compact, we shall say that $\manif$ has compact ends.
\end{definition}

\begin{figure}[ht]
  \centering
  \includegraphics[width=0.5\textwidth]{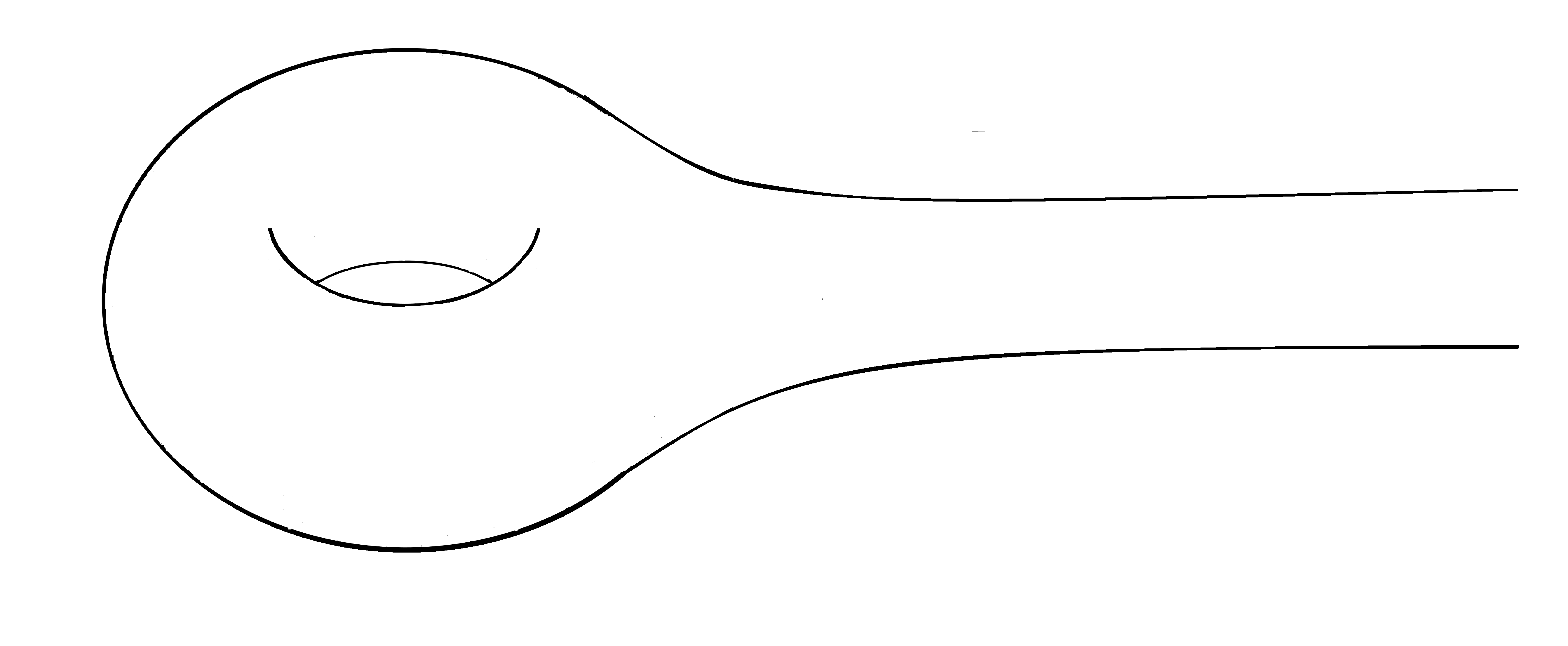}
  \caption{A picture of a manifold with straight cylindrical ends.}
\end{figure}

A manifold with cylindrical ends is thus a non-compact, Riemannian manifold without boundary.
If it has compact ends, it is, moreover, with bounded geometry. (This can be seen, for instance, from
\cite{AGN1, GrosseSchneider}.) We let $r_{\manif}$ denote its injectivity radius.

\begin{remark} \label{rem.def.ce}
  We use the notations of Definition \ref{def.ce}. By setting $\manif_0$ to be the
  complement of $\psi_{\manif}(\link \times (-\infty, 0))$, we see
  that an equivalent way of formulating Definition \ref{def.ce} is that the Riemannian
  manifold $\manif$ contains a smooth
  submanifold $\manif_0\subset \manif$ with boundary $\emptyset \neq \partial \manif_0$ for which there
  exists a diffeomorphism
  $$\phi : \partial \manif_0 \times (-\infty, 0] \to \manif \smallsetminus
  \stackrel{\circ}\manif_0$$
  such that $\phi(x,0) = x$ for $x \in \partial \manif_0$ and the induced
  metric on $\partial \manif_0 \times (-\infty, 0]$ is the product of the metric on $\partial \manif_0$
  and the standard metric on $(-\infty, 0]$. We shall use this notation (as in
  \cite{Mitrea-Nistor}) and assume from now on, for simplicity, that $\phi$ is the identity
  map, which recovers the Equation \eqref{eq.def.can.form} of the Introduction.
  Formula \eqref{eq.def.can.form} provides a {\it standard decomposition} of $\manif$.
\end{remark}

\begin{remark}
When studying $\manif$, it will be necessary to consider also $\pa \manif_0 \times \RR$, which
is also a manifold with straight cylindrical ends.
\end{remark}

From now on, $\manif$ will always be a manifold with straight cylindrical ends
($\manif = \manif_0\cup \big(\partial \manif_0\times (-\infty ,0]\big)$) \emph{or} a closed manifold.
(This is needed for symmetry, because the results that hold for manifolds with cylindrical ends
hold trivially also for compact manifolds. In fact, if in the definition of a manifold
with cylindrical ends one allows $\link = \pa \manif_0 = \emptyset$ and $\manif_0$
is compact, the resulting manifold $\manif$ is compact without boundary).

\subsection{Compatible vector bundles and differential operators}
We now introduce the vector bundles and the differential operators that we will use.

\begin{definition}\label{def.comp.vb}
  Let $\manif$ be a manifold with straight cylindrical ends (as in Remark
  \ref{rem.def.ce}, whose notation we continue to use) and let
  $E \to \manif$ be a Hermitian vector bundle.
  We shall say that $E$ is \emph{compatible with the straight cylindrical ends structure
  on $\manif$} (\emph{compatible}, for short)
  if there exists a constant $R_E < 0$ such that $E\vert_{\pa \manif_0 \times (-\infty, R_E]}$
  is isometric to the pull-back to $\pa \manif_0 \times (-\infty, R_E]$ of a
  Hermitian vector bundle $E_0 \to \pa \manif_0$ with connection $\nabla^{E_0}$
  and that $E$ has a metric preserving connection which, on $\pa \manif_0 \times (-\infty, R_E]$
  coincides with the pull-back connection of $\nabla^{E_0}$. (In the text, we
  shall often use the same notation for $E$ and $E_0$.)
\end{definition}

From now on, $E \to \manif$ will always denote a vector bundle that is compatible
with the straight cylindrical ends structure on $\manif$. The following remark
justifies our interest in compatible vector bundles.

\begin{remark}\label{rem.Bochner}
A compatible vector bundle $E \to \manif$ as in Definition \ref{def.comp.vb} above will thus
also have a product structure on $\pa \manif_0 \times (-\infty, R_E]$. Hence $\nabla^{E} = \nabla^{E_0}
+ dt \otimes \pa_t$ on $\pa \manif_0 \times (-\infty, R_E)$.
Similarly, the Bochner Laplacian of $E$ satisfies
\begin{equation*}
  \Delta_{E} \ede (\nabla^{E})^{*}\nabla^{E} \seq (\nabla^{E_0})^{*}\nabla^{E_0} - \pa_t^2\,,
\end{equation*}
where $\nabla^{E_0}$ is the connection on $E_0 \to \pa \manif_0$.
\end{remark}

We now introduce an important class of smooth sections.

\begin{definition}\label{def.CIinv}
  Let $E \to \manif$ be as in Remark \ref{rem.Bochner} and let $u \in \CI(\manif; E)$.
  We shall say that $u$ is \emph{translation invariant in a neighborhood of infinity}
  (or \emph{translation invariant at infinity}) if
  there exists $R_u \le R_E$ such that
  $u(x, t)$ is independent of $t$ for $t \le R_u$. We let $\maR_{\infty}u(x) := u(x, t)$
  for $-t$ large.
  We let $\CI_{\inv}(\manif; E)$ denote the set of
  sections of $E$ that are translation invariant in a neighborhood of infinity.
  We identify $\CI(\pa \manif_0; E) \simeq \CI(\pa \manif_0 \times \RR; E)^{\RR}$
  and call the induced map
  \begin{equation*} 
    \begin{gathered}
    \maR_{\infty} : \CI_{\inv}(\manif; E) \to \CI(\pa \manif_0 \times \RR; E)^{\RR}
    \simeq \CI(\pa \manif_0; E)\\
    u(x, t) \, =:\, \maR_{\infty} u(x) \in E \quad \mbox{for } t \le R_u\,,
    \end{gathered}
  \end{equation*}
  the \emph{limit at infinity} map.
  We let $\CI_{\inv}(\manif) := \CI_{\inv}(\manif; \CC)$.
\end{definition}

We now introduce similarly the differential operators we are interested in.

\begin{notation}\label{not.diff.ops}
  Let $F \to \manif$ be a second vector bundle on $\manif$ compatible with its straight
  cylindrical ends at infinity. We shall denote by $\operatorname{Diff}_{\inv}^m(\manif; E, F)$
  the set of differential operators $P$ acting
  on sections of $E$ such that there exist $a_j \in \CI_{\inv}(\manif;
  \Hom(T^{* \otimes j} \manif \otimes E, F))$, $j = 0, \ldots, m$, with the property that
  \begin{equation*}
    Pu  \seq \sum_{j=0}^m a_j (\nabla^{E})^j u\,.
  \end{equation*}
  In particular,
  \begin{equation*}
    \nabla^{LC} \in \operatorname{Diff}^1_{\inv}(\manif; T\manif, T^*\manif \otimes T\manif) \,.
  \end{equation*}
  We let
  \begin{equation*}
    \maR_{\infty} (P) \ede \sum_{j=0}^m \maR_{\infty} (a_j) (\nabla^{E})^j
  \end{equation*}
  be the associated operator on $\pa\manif_0 \times \RR$ obtained by ``freezing
  the coefficients'' of $P$ near infinity. It will be called the \emph{limit at infinity}
  operator associated to $P$.
\end{notation}

We shall need the following simple result.

\begin{lemma} \label{lemma.easy.inv}
  Let $E, F, G \to \manif$ be compatible vector bundles on our manifold with straight
  cylindrical ends $\manif$. We use the notation introduced in this subsection.
  \begin{enumerate}
    \item The limit at infinity map $\maR_{\infty} : \CI_{\inv}(\manif; E) \to
    \CI(\pa \manif_0 \times \RR; E)^{\RR}$ of Defintion \ref{def.CIinv} is surjective and
    \begin{equation*}
      \ker \maR_\infty \seq \CIc(\manif; E)\,.
    \end{equation*}
  \item $\CI_{\inv}(\manif; \Hom(E, F)) = \CI(\manif; \Hom(E, F))
  \cap \operatorname{Diff}_{\inv}^m(\manif; E, F)$.

  \item $\operatorname{Diff}_{\inv}^m(\manif; F, G) \operatorname{Diff}_{\inv}^{m'}(\manif; E, F)
  \subset \operatorname{Diff}_{\inv}^{m+m'}(\manif; E, G)$.

  \item $\operatorname{Diff}_{\inv}^m(\manif; E, F) \CI_{\inv}(\manif; E)
  \subset \CI_{\inv}(\manif; F)$.

  \item \label{item.easy.inv5}
  $\maR_{\infty} : \operatorname{Diff}_{\inv}^m(\manif; E, F) \to
  \operatorname{Diff}_{\inv}^m(\pa \manif_0 \times \RR; E, F)^{\RR}$ is well-defined,
  surjective, and multiplicative.
  \end{enumerate}
\end{lemma}

\begin{proof}
  Everything follows immediately from the definitions. (For (3) see also \cite{KohrNistor1}.) Let us just add that
  for \eqref{item.easy.inv5} we first prove that that $\maR_{\infty} : \CI_{\inv}(\manif; \Hom(E, F)) \to
  \CI_{\inv}(\pa \manif_0 \times \RR; \Hom(E, F))^{\RR}$ is well-defined,
  surjective, and multiplicative and that $\maR_{\infty}(\nabla^{LC}) = \nabla^{LC}.$
\end{proof}

We have the following simple consequence (recall that $\nabla$ is the generic
notation for our connections).

\begin{corollary}\label{cor.lemma.ue}
  Let $u \in \CI_{\inv}(\manif; E)$ (see Definition \ref{def.CIinv}), then $u$ is
  bounded and, for all $k \ge 1$,
  $\nabla^k u \in \CI_{\inv}(\manif; E \otimes T^{*\otimes k} \manif)$.
  Let $a \in \CI_{\inv}(\manif; \End(E))$ with $a \neq 0$ everywhere. Then $a$
  is bounded and $a^{-1} \in \CI_{\inv}(\manif; \End(E))$. In particular, $\maR_{\infty}(a)
  \neq 0$ everywhere.
\end{corollary}

\begin{proof}
  This follows from Definition \ref{def.CIinv}, Lemma \ref{lemma.easy.inv}.
  We have that $\maR_{\infty}(a)$ is a restriction of $a$, so it doesn't vanish either.
\end{proof}

Although we shall not really use it, a useful picture
when $\manif$ has compact ends is that one can ``compactify''
$\manif$ to $\overline{\manif}$
by  including the ``boundary at infinity'' $\pa \manif_0 \times \{-\infty\}$.  All
sections that are translation invariant at infinity will extend to sections
on this compactification. In particular, the
boundedness of $u$ in Corollary \ref{cor.lemma.ue} follows from the fact that it
extends to the compactification $\overline{\manif}$.

\subsection{Sobolev spaces}\label{ssec.Sobolev}
In the following, $(\manif, g)$ will be a fixed manifold with straight cylindrical ends
and \emph{compact ends.}
Let us recall now the definition of the Sobolev spaces on $\manif$.
We agree that $\NN = \{1, 2, \ldots, \}$ and that $\ZZ_+ = \NN \cup \{0\}$.
For simplicity of the notation, we introduce vector bundles only later.

\subsubsection{Positive order covariant Sobolev spaces}
By $\dvol $ we denote the induced volume form on $\manif$ associated to the metric $g$ on $\manif$.
(Note that the metric $g$ is fixed everywhere in the paper.)
For $p \in [1, \infty)$, we define the norm
\begin{equation}\label{eq.def.normp}
  \|u\|_{L^p(\manif)} \ede \left(\int_{\manif} |u(x)|^p \dvol(x)\right)^{1/p}\,.
\end{equation}
As usual, we define the space:
\begin{equation}\label{eq.def.Lp}
  L^p(\manif) \ede \{ u : \manif \to \CC
  \mid \|u\|_{L^p(\manif)}< + \infty \}/\ker (\|\cdot \|_{L^p(\manif)})\,.
\end{equation}
The space $L^\infty(\manif)$ and its norm are defined with the usual modification,
using the essential supremum.
By $(f, g) = (f, g)_{L^2(\manif)}$ we shall denote the scalar (inner) product on $L^2(\manif)$.
It is linear in the first variable and conjugate linear in the second one.

There are many ways of introducing Sobolev spaces on a manifold. For open subsets of compact
manifolds, the Sobolev spaces can be defined using partitions of unity supported in
coordinate patches (thus reducing to the Euclidean case), and the resulting
norms are equivalent for different choices of the partition of unity. In the non-compact
case, one cannont use any partition of unity. One can still use connections, though
(see, for instance \cite{HebeyBook, KohrNistor1}).

\begin{definition}\label{def.Sobolev}
Let $k \in \ZZ_+ := \{0, 1, \ldots \}$. We let
\begin{equation*}
  W^{k, p}(\manif) \ede \{ u: \manif \to \CC \mid \nabla^j(u) \in
  L^p(\manif;T^{*\otimes j} \manif) \,,\, \mbox{ for all }\, 0 \le j \le k \}
\end{equation*}
be the {\em order $k$, $L^p$--type Sobolev space} on $\manif$ with the norm
\begin{equation*}
  \|u\|_{W^{k, p }(\manif)} \ede \ell^p\mbox{--norm of }
  \{ \|\nabla^j (u)\|_{L^p(\manif;T^{*\otimes j} \manif)}\,,\ 0 \le j \le k \} \,.
\end{equation*}
\end{definition}

We let $H^{k} := W^{k, 2}$, as usual.
Since $\manif$ is complete (even with bounded geometry) the Sobolev space $H^s(\manif)$
is the domain of $(1 + \Delta)^{s/2}$, for $s \ge 0$ (we use the convention
$\Delta := d^* d \ge 0$). See \cite{AGN1, GrosseSchneider, Strichartz}  for related results
(valid also in the more general setting of manifolds with bounded geometry
and, sometimes, valid even in the setting of complete manifolds).

\subsubsection{Negative and non-integer order covariant Sobolev spaces}
We proceed as usual and we define the negative order Sobolev spaces (just for $p = 2$) by
\begin{equation}  \label{eq.def.neg.Sob}
  H^{-s}(\manif) \ede H^{s}(\manif)^* \,,
\end{equation}
where $V^*$ is the complex conjugate dual of $V$, as before.  We also define the spaces
$H^{r}(\manif)$ for $r\notin \ZZ$ by complex interpolation
between consecutive integers (see, e.g., \cite{BrezisBook, LionsMagenes1, Taylor1, Wong}
for further details).

\subsubsection{Alternative definition}\label{sssec.alternative}
Sometimes, the following definition of Sobolev spaces will be useful.
All manifolds $\pa \manif_0 \times (-R-k, -R-k+2)$, $k \in \NN$, are isometric and
hence the spaces $$H^m (\pa \manif_0 \times (-R-k, -R-k+2))\,, \quad k \in \NN \,,$$ are also canonically
isometric (they can be defined using a finite partition of unity on $\pa \manif_0$.)
Let $$\manif_R \ede \manif_0 \cup (\pa \manif_0 \times (-R, 0)) \seq \manif \smallsetminus
(\pa \manif_0 \times (-\infty, -R]) \,.$$ Then an
equivalent norm on $H^m(\manif)$, $m \in \RR$, is
\begin{equation}\label{eq.equiv.norm1}
  |||u|||_{m}^2 \ede \|u \|_{H^m(\manif_R)}^2 \, + \,
  \sum_{k=1}^\infty \|u \|_{H^m(\pa \manif_0 \times (-R-k, -R-k+2) )}^2\,.
\end{equation}
There exists a smooth partition of unity $(\phi_k^2)$, $k \in \ZZ_+$, where $\phi_k$ is
smooth with support in $\pa \manif_0 \times (-R-k, -R-k+2)$ and is the translate of
$\phi_1$ by $k-1$, for $k \ge 1$, and where $\phi_0$ is with support in $\manif_R$. Then the
norm on $H^m(\manif)$, $m \in \RR$, is also equivalent to
\begin{equation}\label{eq.equiv.norm2}
  |||u|||_{m}^{\prime 2} \ede \| \phi_0^2 u \|_{H^m(\manif_R)}^2 \, + \,
  \sum_{k=1}^\infty \| \phi_k^2 u \|_{H^m( \pa \manif_0 \times (-R-k, -R-k+2) )}^2\,.
\end{equation}
These equivalences follows, for instance, from \cite{GrosseSchneider}.

If $\manif = \link \times \RR$, the same results hold true if
we consider the $\ZZ$-translation invariant partition of
unity $(\phi_k^2)$, $k \in \ZZ$ (extending from $k \in \NN$).

\section{Basics on pseudodifferential operators}
\label{sec.backgr.psdos}

We include in this section a rather detailed and comprehensive overview of the
standard theory of pseudodifferential operators on $\RR^n$, on general manifolds,
and on compact manifolds. This serves to fix the notation, to help the reader, but also
to provide a couple of new results that are needed in the sequel.
The material of this section is based on (and expands)
our talks at the ``Summer school in Geometric Analysis'' at the Free University of
Bruxelles, in 2023. In order to avoid complications, we shall assume that our
manifolds are Hausdorff and paracompact. The results of this section are,
for the most part, not new, but rather classical. A quick, but lucid introduction
to pseudodifferential operators can be found in \cite{Taylor2}. Two more recent
accounts that contain also more detail as well as some needed non-standard
results are \cite{HAbelsBook} and \cite{HintzLN}. We thank Rafe Mazzeo for brining
HintzLN' book \cite{HintzLN} to our attention. The classical reference monograph on
pseudodifferential operators containing also historical material and many
extensions is \cite{Hormander3}. Otherwise, there are many other very good
textbooks containing material on pseudodifferential operators, among which we
mention \cite{H-W, RT-book2010, TrevesBookPsdo, Wong}. We include some proofs
for the benefit of the reader and as complementary material (not because any
of these results in this section are really new).

\subsection{The ``main formula'' and pseudodifferential operators on $\RR^n$}

Let \m{\langle x, \xi \rangle := x_1 \xi_1 + x_2 \xi_2 + \ldots x_n \xi_n}
for $x, \xi \in \RR^n$, $\imath := \sqrt{-1}$, and let
\begin{equation*}
    {  \hat u(\xi)} \seq \maF u(\xi) \ede  \int _{\RR}
    e^{ {  -} \imath \langle {  y}, \xi \rangle } \,
    {  u(y)\, dy}
\end{equation*}
be the  \emph{Fourier transform} of \m{u.}
The following formula
\begin{equation*}
   a(x, D) u(x) \ede \frac{1}{(2\pi)^n} \int_{\RR^n}\, e^{\imath \langle x,
  \xi \rangle}\, a(x, \xi)
  {  \hat u(\xi)} \, d\xi\,,
\end{equation*}
will be referred to as ``the main formula'' (of pseudodifferential theory) and, as we
will see, it gives the ``local form'' of pseudodifferential operators on manifolds.
In this formula, \m{a : \RR^n \times \RR^n \to \CC} is a suitable ``symbol''
that will depend on the requirements of the problem. (The terms in quotation marks will
be defined below.) Let us see what the main formula becomes in a few particular cases of
symbols $a : \RR^{2n} \to \CC$.

\begin{example}
  Let $a(x, \xi) = a(x)$ be a smooth function independent of $\xi$. Then
  \begin{equation*}
    a(x, D) u(x)  \ede \frac{1}{(2\pi)^n} \int_{\RR^n}\,
    e^{\imath \langle x, \xi \rangle}\, a(x)
    \hat u(\xi) \, d\xi \seq a(x)  u(x)\,,
  \end{equation*}
  by the  \emph{Fourier inversion formula,}
  \begin{equation*}
    (\maF^{-1} w)(x) \ede  \frac{1}{(2\pi)^n} \int_{\RR}
    e^{\imath \langle x, \xi \rangle } w(\xi)d\xi\,.
  \end{equation*}
\end{example}

Recall that the \emph{convolution} of two function $f, g : \RR^n \to \CC$ is
(when defined)
\begin{equation}\label{eq.def.convolution}
  f * g(x) \seq \int_{\RR^n} f(y)g(x - y)dy\,.
\end{equation}

\begin{example}
  Let next $a(x, \xi) = a(\xi)$ be a function independent of $\xi$,
  and integrable in $\xi$. Then
  \begin{align*}
    a(D)  u(x) \seq  a(x, D) u(x) \ede \frac{1}{(2\pi)^n} \int_{\RR^n}\, e^{\imath \langle x,
    \xi \rangle}\, a(\xi)\, \hat u(\xi) \, d\xi \seq (\maF^{-1}a) * u(x)
  \end{align*}
  is the \emph{convolution} operator with \m{\maF^{-1}a}. This remains true if
  $a \in \maS'(\RR^n)$ (that is, if the function $a$ defines a tempered distribution
  on $\RR^n$, in that case ``oscillatory integrals'' may be useful; see
  \cite{Wl-Ro-La} for a pedagogical introduction and \cite{HAbelsBook} for a
  thorough account). In particular, if
  $a(x, \xi) := a_j(\xi) := \imath \xi_j,$ then
  \begin{align*}
    a_j(D)  u(x)  \ede \frac{1}{(2\pi)^n} \int_{\RR^n}\,
    e^{\imath \langle x, \xi \rangle}\,   \imath \xi_j\, \hat u(\xi) \, d\xi
    \seq \frac{\partial u(x)}{\partial x_j} \, =: \, \pa_j  u(x) \,,
  \end{align*}
  because the Fourier transform interchanges multiplication by \m{\imath \xi_j}
  with \m{\pa_j.} Morover, if we denote by $a$ also the multiplication operator
  with $a$, then $a(D) = \maF^{-1}a \maF$. Thus, if both $a$ and $b := a^{-1}$ are in
  $\maS'(\RR^n)$, then
  \begin{equation*}
    a(D) b(D) \seq (\maF^{-1}a \maF) (\maF^{-1}b \maF) \seq \maF^{-1}ab \maF \seq 1\,.
  \end{equation*}
  In particular, if $a(\xi) = 1 + |\xi|^2$, then $a(D) = 1 - \Delta$. Let $b(\xi) = (1 + |\xi|^2)^{-1}$,
  then $a(D)^{-1} = b(D)$, which is the convolution operator with a true function.
\end{example}

The next example combines the first two examples.

\begin{example}
  Combinging the first two examples and using for any multi-index
  $\alpha = (\alpha_1, \alpha_2, \ldots, \alpha_n)\in \ZZ_+^n$,
  \m{|\alpha| := \sum_{j=1}^n \alpha_j,} and
  $$\pa^\alpha u \ede \pa_1^{\alpha_1} \pa_2^{\alpha_2} \ldots \pa_n^{\alpha_n} u
  \seq {  \maF^{-1}} \big( {  (\imath \xi)^\alpha} \hat u\big )\,, \quad
  (\imath \xi)^\alpha \ede \imath^{|\alpha|} \xi_1^{\alpha_1}  \xi_2^{\alpha_2} \ldots
  \xi_n^{\alpha_n}\,, $$ we obtain
  for \m{a(x, \xi) := \sum_{|\alpha| \le m} \, a_{\alpha} (x) \,
  (\imath \xi)^{\alpha},}  (a \emph{polynomial in} $\xi$):
  \begin{align*}   a(x, D)u(x) &
    \ede \frac{1}{(2\pi)^n} \int_{\RR^n} e^{\imath \langle x, \xi \rangle}
    \underbrace{\sum_{ |\alpha| \le m} \, a_{\alpha} (x) \,
    {(\imath \xi)^{\alpha}}}_{a(x, \xi)}\, \hat u(\xi) d\xi \\
    &  \seq \sum_{{  {|\alpha| \le m}}} \, a_{\alpha} (x) \,
    {\pa^{\alpha}} u(x)\,,
  \end{align*}
a \emph{differential operator}.
\end{example}

Let us specify now our choice of symbols.

\begin{notation} \label{not.symbols}
  We shall denote an element of $\RR^{2n}$ by $(x, \xi)$, with $x, \xi \in \RR^{n}$.
  For $a : U \times \RR^{n} \to \CC$, $U \subset \RR^n$ open,
  $\alpha, \beta \in \ZZ_+^n$, and $s \in \RR$, we let
  \begin{enumerate}
    \item $$p_{\alpha, \beta}^{[q]}(a) \ede \sup_{(x, \xi)\in U \times \RR^{n}}
    \left | \frac{\pa_x^\alpha \pa_\xi^\beta a(x, \xi)}{(1+|\xi|)^{q - |\beta|}} \right |
    \quad \mbox{and}$$
    \item $p_{N}^{[q]}(a) := \max_{|\alpha| + |\beta| \le N} p_{\alpha, \beta}^{[q]}(a).$

    \item We then let
    \begin{equation*}
       S_{unif}^{q}(U \times \RR^{n}) \ede \{ a : U \times \RR^{n} \to \CC \mid\
      p_{N}^{[q]}(a) < \infty \mbox{ for all } N \in \ZZ_+ \}\,.
    \end{equation*}
  \end{enumerate}
\end{notation}

This class of symbols was introduced by Kohn and Nirenberg in \cite{KohnNirenberg}.
See \cite{Hormander3} and \cite{Taylor2} for more historical information and for
other classes of symbols.

\begin{remark}\label{rem.symbols}
We think of $U \times \RR^{n}$ as $T^*U$. Similar definitions can be given for any
vector bundle $V \to M$ over a manifold $M$, replacing $\pa_x^\alpha$ with a suitable
covariant derivative and replacing $\pa_\xi^\beta$ with derivatives with respect to
vertical vector fields that are translation invariant in the fibers to obtan the
spaces $S_{unif}^m(V)$. Notice that in this notation we consider $V$ as a vector bundle
over $M$, not simply as a manifold, and hence this space should better be denoted
$S_{unif}(M; V)$ to indicated the additionals structure on which it depends.
Notice also that we consider estimates that are \emph{uniform} in $x$, as in
\cite{HAbelsBook, HintzLN, Hormander3, Taylor2, Wl-Ro-La}. These definitions can be localized on
compact subsets of $U$ as in \cite{H-W}. Also, a more general class of symbols is
$S_{\rho, \delta}^m(\RR^{2n})$ considered by H\"ormander, but we do not need it in
our applications.
\end{remark}

Let $M$ be a smooth manifold. Recall that a linear map \m{T: \CIc (M)\to \CIc(M)'} is
called \emph{continuous} if \m{\langle T\phi , \psi \rangle} is continuous with respect to
\m{\phi \in \CIc(M),} for any \m{\psi \in \CIc (M).}
  It is not difficult to see that, if $a \in  S_{unif}^q(\RR^{2n})$, then
  $a(x, D) : \CIc(\RR^n) \to \CI(\RR^n)$ is well-defined and continuous.
(A stronger statement is included below.)
This proposition shows that \m{a(x,D)} satisfies the hypotheses of the
Schwartz' kernel theorem, which we recall next.
Let \m{\langle \ , \ \rangle : \CIc (M)' \times  \CIc(M) \to \CC} be the
\emph{paring} between distributions and test functions.

\begin{theorem}[Schwartz' kernel theorem]\label{thm.Schwartz.k}
  Let $M$ be a smooth manifold.
  Let \m{T : \CIc (M)\to \CIc(M)'} be linear and continuous.
  Then there exists a unique \m{k_T \in \CIc(M \times M)'}  such that
  \[ \langle T\phi, \psi\rangle  \seq \langle k_T, \psi \boxtimes \phi \rangle \,,\]
  where
  \m{\psi \boxtimes \phi \in \CIc(M^2)} is defined by
  \m{(\psi \boxtimes \phi )(x,y) = \psi (x)\phi (y)\,.}
  The distribution \m{k_T} is called the \emph{distribution kernel} of \m{T.}
\end{theorem}

The converse is also true, in the sense that every \m{K \in \CIc(M^2)'}
corresponds to a continuous linear map \m{T : \CIc(M) \to \CIc(M)'}, namely, \m{K = k_T.}
This is the ``easy part'' of the theorem but will not be necessary to us in what
follows. A natural question that arises then is to
find the kernel $k_{a(x, D)}$ of the operator $a(x, D)$.

\begin{remark}\label{rem.second.mf}
  If in the  \emph{main formula}
  \begin{equation*}
    a(x, D) u(x)  \ede \frac{1}{(2\pi)^n} \int_{\RR^n}\, e^{\imath \langle x,
    \xi \rangle}\, a(x, \xi)
    {\hat u(\xi)} \, d\xi
  \end{equation*}
  we substitute the Fourier transform $\hat u(\xi) \seq \maF u(\xi) =
  \int _{\RR} e^{-i \langle {  y}, \xi \rangle } \, u(y)\, dy,$
  we obtain the  \emph{second main formula} of pseudodifferential operator theory:
  \begin{equation*}
    a(x, D) u(x) \seq \frac{1}{(2\pi)^n} \int_{\RR^n}  \left ( \int_{\RR^n}\,
    e^{\imath \langle x - y, \xi \rangle}\,  a(x, \xi) \, {  u(y) \, dy} \right ) d\xi\,.
  \end{equation*}
\end{remark}

This second main formula allows us to determine the distribution kernel of $a(x,D)$.

\begin{remark} \label{rem.kernel.formula}
  Assume first that $a \in S_{unif}^{-n-\epsilon}(\RR^{2n})$, $\epsilon > 0$, and let
  $$K(x, y) := \frac{1}{(2\pi)^n} \int_{\RR^n} e^{\imath \langle x -y, \xi \rangle}
  a(x, \xi) d\xi,$$ which is a continuous function of $(x, y)$ by the continuity of
  integrals depending on parameters by bounding with the integrable function $(1+|\xi|)^{-n-\epsilon}$.
  The second main formula (see Remark \ref{rem.second.mf}) gives
  \begin{align*}
    \langle a(x, D) u, v \rangle & =  \int_{\RR^n} \left ( \frac{1}{(2\pi)^n} \int_{\RR^n} \int_{\RR^n}
    e^{\imath \langle x-y, \xi \rangle} a(x, \xi) \, u(y) dy {  d\xi} \right ) \,  v(x) dx\\
    &  = \int_{\RR^n} \int_{\RR^n} \, K(x, y)  v(x) u(y) \, dy dx\\
    &  = \langle K , v \boxtimes u \rangle \,,
  \end{align*}
  using an integration with respect to $\xi.$ Therefore
  \begin{equation}\label{eq.kaxd}
    k_{a(x, D)}(x, y) \seq K(x, y) \ede
    \frac{1}{(2\pi)^n} \int_{\RR^n} e^{\imath \langle x -y, \xi \rangle}
    a(x, \xi) d\xi\,.
  \end{equation}
  In general, this formula for $k_{a(x, D)}(x, y)$
  must be interpreted in distribution sense, as the Fourier transform of a temperate
  distribution (this can be done using oscillatory integrals, as in \cite{HAbelsBook, Wl-Ro-La},
  for instance).
\end{remark}

We obtain the following.

\begin{proposition} \label{prop.kernels}
  Let $m \in \RR$ and $a \in S_{unif}^m(\RR^{2n}).$
  \begin{enumerate}
    \item The distribution $k_{a(x, D)}$ is defined by a smooth function
    on $\{(x,y) \in \RR^{2n} \mid x \neq y\}$
    (that is, away from the diagonal).

    \item If $k_{a(x, D)} = 0$, then $a = 0$.

    \item Consequently, $k_{a(x, D)}$ determines $a$ and hence
    the map $a(x, D) : \CIc(\RR^n) \to \CI(\RR^n)$
    also determines $a$.
  \end{enumerate}
\end{proposition}

\begin{proof}
  The stronger statement that $(x-y)^\alpha k_{a(x, D)}(x,y) \in \maC^j(\RR^{2n})$
  for $|\alpha|$ large is contained in the proof of Proposition 2.1 in \cite{Taylor2}
  (this result holds true also for suitable symbols in the H\"ormander class
  $S_{\rho, \delta}^m$). The second point is obtained from Equation \eqref{eq.kaxd}
  by the Fourier inversion formula.
\end{proof}

The following remark explains ``what is lost'' by looking only at the distribution
kernel on $\{(x,y)\mid x \neq y\}.$ It will not be needed for any proof.

\begin{remark}\label{rem.Peetre}
If $a, b \in S_{unif}^q(\RR^{2n})$ and $k_{a(x, D)} (x, y) = k_{b(x, D)} (x, y)$ for
$x \neq y$, then $a(x, D) - b(x, D)$ has distribution kernel supported on the diagonal
$(x, x) \subset \RR^{2n}$ and hence it is a differential operator, by Peetre's theorem.
In particular, the function $k_{a(x, D)}$ away from the diagonal determines any operator
of negative order.
\end{remark}

Let us now discuss the mapping properties of our operators. Recall that
\begin{equation*}
  H^{s}(\RR^n) \ede \Big  \{u \, \Big | \ \int_{\RR^n} (1 + |\xi|)^{2s}
  |\hat u(\xi)|^2\, d\xi < +\infty \, \Big  \}\,,\quad s \in \RR\,.
\end{equation*}
In particular, if $m \in \ZZ_+$,
\begin{equation*}
  H^{m}(\RR^n) \ede \big  \{u \, \big | \ \pa^\alpha u \in L^2(\RR^n)
  \ \mbox{ for all }\ |\alpha| \le {m}\, \big \} \,.
\end{equation*}
These definitions coincide, of course, with our previous definitions of
Subsection \ref{ssec.Sobolev}. We let
\begin{equation}\label{eq.def.PsiRn}
  \Psi_{unif}^m(\RR^n) \ede \{ a(x, D) \mid a \in   S_{unif}^{m}(\RR^{2n})\}
\end{equation}
and (for some fixed $p \in [1, \infty]$), we let
$$\maS(\RR^n) := \{ u : \RR^n \to \CC \mid x^\alpha \pa^\beta u \in L^p(\RR^n),\,
\forall\, \alpha ,\beta \in \ZZ_+^n\},$$
be the usual \emph{Schwartz space} on $\RR^n$ with the induced norms (the resulting
space does not depend on $p \in [1, \infty]$, but its semi-norms do).

\begin{theorem}\label{thm.mapping.properties}
  Let $s, m \in \RR$ and $\Psi_{unif}^m(\RR^n)$ be as in Equation \eqref{eq.def.PsiRn}.
  \begin{enumerate}
    \item The map $S_{unif}^m(\RR^{2n}) \times \maS(\RR^n) \ni (a, u) \to
    a(x, D)u \in \maS(\RR^n)$ is well defined and continuous.

    \item If \m{a \in S_{unif}^m(\RR^{2n})} and \m{b \in S_{unif}^{m'}(\RR^{2n}),} then
    \m{ab \in   S_{unif}^{m+m'}(\RR^{2n})} and
    $$b(x, D)a(x, D) - (ab)(x, D) = c(x, D) \in \Psi_{unif}^{m+m'-1}(\RR^n)\,.$$
    Moreover, for any $N \in \ZZ_+$, there is $C_N > 0$ such that (see \ref{not.symbols})
    $$p_N^{[m+m'-1]} (c) \le C_N
    \sum_{j=1}^n  \big[ p_N^{[m]}(a) p_N^{[m'-1]}(\pa_{\xi_j} b)
    + p_N^{[m-1]}(\pa_{\xi_j}a) p_N^{[m']}(b) \big] \,.$$
    In particular, if \m{P \in \Psi_{unif}^m(\RR^n)} and \m{Q \in \Psi_{unif}^{m'}(\RR^n)}, then
    \m{PQ \in \Psi_{unif}^{m+m'}(\RR^n)}.

    \item \m{P = a(x, D): H^s(\RR^n) \to H^{s-m}(\RR^n)} is bounded and its norm depends
    continuously on \m{a \in S_{unif}^m(\RR^{2n})}.

    \item Similarly, if
    $a(x, D) \in \Psi_{unif}^{m}(U)$, then $a(x, D)^* - \overline{a}(x, D) \in \Psi_{unif}^{m-1}(U)$.
    In particular, if $P \in \Psi_{unif}^m(\RR^n)$, then $P^* \in \Psi_{unif}^m(\RR^n)$.
  \end{enumerate}
\end{theorem}

\begin{proof} Most of these results are standard. We include only some
  references to the less standard material. The point (1) is a particular case
  of Theorem 3.6. in \cite{HAbelsBook}. See also Proposition 4.4 in \cite{HintzLN}.
  The norm estimate in (2) follows from Proposition 3.5 and Theorems 3.15 and 3.16
  in \cite{HAbelsBook}. (Proposition 3.5 and Theorem 4.8 and 4.16 from \cite{HintzLN}
  are also relevant in this regard.) The continuity between the Sobolev spaces is
  of course very well-known, the slightly stronger statement provided here is a
  consequence of the proof of continuity on $L^2(\RR^n)$ using H\"ormander's
  trick \cite{Hormander3} in \cite{HAbelsBook, HintzLN}, using also point (2).
\end{proof}

Let $S_{unif}^m/S_{unif}^{m-1}(\RR^{2n}) := S_{unif}^m(\RR^{2n})/S_{unif}^{m-1}(\RR^{2n})$
and
\begin{equation}
  \sigma_m : \Psi_{unif}^m(\RR^n) \to S_{unif}^m/S_{unif}^{m-1}(\RR^{2n})\,,
  \quad \sigma_m(a(x, D)) \ede a +  S_{unif}^{m-1}(\RR^{2n})\,,
\end{equation}
which is well-defined by Proposition \ref{prop.kernels}. Theorem \ref{thm.mapping.properties}(2)
then gives right away that $\sigma_{m+m'}(QP) = \sigma_{m'}(Q)\sigma_{m}(P).$

We shall need also a ``local form'' of the last theorem. Let \m{U \subset \RR^n} be open and let
\begin{equation*}
  \begin{gathered}
    \Psi_{comp}^{m}(U)  \ede \{a(x, D) \mid\, \supp {  k_{a(x,D)}} \Subset U \times U \}
    \quad \mbox{and}\\[1mm]
    S_{comp}^{m}(T^*U)  \ede \{a \in   S_{unif}^{m}(\RR^{2n}) \mid \supp a \subset K \times \RR^n\,,\
    K \Subset U\}\,,
  \end{gathered}
\end{equation*}
where $A \Subset B$ means that the closure of $A$ is a compact subset of the interior of
$B$. Remark \ref{rem.kernel.formula} shows that \dm{a(x, D) \in \Psi_{comp}^m(U)
\Rightarrow a \in S_{comp}^{m}(T^*U) \subset   S_{unif}^m(\RR^{2n})}
and that every $P \in \Psi_{comp}^{m}(U)$ maps $\CIc(U)$ to itself. In particular,
we can compose these actions.

\begin{theorem} \label{thm.diffeo.inv1}
  For \m{U \subset \RR^n} an open subset, $\Psi_{comp}^{\infty}(U) :=
  \cup_{m \in \ZZ} \Psi_{comp}^{m}(U)$ is an algebra invariant under
  diffeomorphisms and under adjoints.
\end{theorem}

The adjoints are taken with respect to the usual metric on $\RR^n$.
Let us make the diffeomorphism invariance more precise.

\begin{remark}\label{rem.precise.diffeo}
Let us consider in Theorem \ref{thm.diffeo.inv1} a diffeomorphism \m{\phi : U \to W} of
two open subsets of $\RR^n$. Then $\phi$ induces:
\begin{enumerate}
  \item a vector bundle isomorphism
  \m{\phi^* : T^*W \to T^*U,} where $$\phi^*(x, \xi) = (\phi^{-1}(x), (d\phi)^\top\xi)$$
  is the action  of $\phi$ between the cotangent bundles (with $L^\top$ being
  the \emph{transpose}, or dual, of $L$);
  \item a linear isomorphism \m{\phi_* : \CIc(U) \to \CIc(W),} where
  $\phi_*u := u \circ \phi^{-1}$; and
  \item a linear isomorphism \m{\phi_* : S_{comp}^m(T^*U) \to S_{comp}^m(T^*W),} where
  $\phi_*a = a \circ \phi^{*}$.
\end{enumerate}
Let \m{a(x, D) \in \Psi_{comp}^m(U)}, then \m{\phi_*a = a \circ \phi^{*} \in S_{comp}^{m}(T^*W)} and
a very important result \cite{HormanderB16, KohnNirenberg} is that
\begin{equation*}
  \phi_{*} \circ a(x, D) \circ \phi_{*}^{-1} - (\phi_*a)(x, D) \in \Psi_{comp}^{m-1}(W)\,.
\end{equation*}
Consequently, $\phi_{*} \circ a(x, D) \circ \phi_{*}^{-1} \in \Psi_{comp}^{m}(W)$ and
this gives the diffeomorphism invariance of the spaces $\Psi_{comp}^{m}(U)$ with $U \subset \RR^n$
open. (See \cite{Hormander3} for historical comments and \cite{Taylor2} for an inventive proof
of the diffeomorphism invariance using Egorov's theorem).
\end{remark}

Let \m{\phi : U \to W} be a diffeomorphism, as before. Then a consequence of the
above discussion is that
\begin{equation*}
  \xymatrix{
  S_{comp}^{m}(T^{*}U) \ni a & \phi_*(a) \in S_{comp}^{m}(T^{*}W)\\
  {  \Psi_{comp}^{m}(U)} \ni a(x, D) & \ \phi_{*}a(x, D)\phi_{*}^{-1} \in   \Psi_{comp}^{m}(W)
  \ar^{\phi_*}"1,1";"1,2"
  \ar^{}"1,1";"2,1"
  \ar^{}"1,2";"2,2"
  \ar^{}"2,1";"2,2"}
\end{equation*}
commutes up to  \emph{lower order symbols.}
This allows to obtain the local definition of the principal symbol on $\RR^n$.

\begin{theorem} \label{thm.diffeo.inv2}
  If $a \in  S_{unif}^m(\RR^{2n})$ is such that \m{a(x, D) \in \Psi_{comp}^m(U)}, then the
  \emph{principal symbol}
  $$\sigma_m(a(x,D)) \ede a +   S_{unif}^{m-1}(T^*U) \in   S_{unif}^{m}(T^*U)/  S_{unif}^{m-1}(T^*U)$$
  is well-defined and diffeomorphism invariant in the sense that
  $$ \sigma_m(\phi_{*}a(x, D)\phi_{*}^{-1}) \seq \phi_*(\sigma_m(a(x, D))) \ede
  \sigma_m(a(x, D)) \circ \phi^*\,.$$
\end{theorem}

One can then obtain the definition of the principal symbol on $\Psi_{unif}^m(\RR^n)$,
but we prefer to do it directly in the case of non-compact manifolds. For $\eta \in \RR^n$,
we let $e_\eta : \RR^n \to \CC$,
$e_\eta(x) := e^{\imath \langle x, \eta \rangle}$. The behavior of
the principal symbol can be obtained from the formula
\begin{equation}\label{eq.osc.testing}
  [a(x,D)e^{\eta}](x) \seq e^{\imath \langle x, \eta \rangle} a(x, \eta)\,,
\end{equation}
for suitable $a$, for instance, for $a \in S_{unif}^m(\RR^{2n})$ an order $m$ symbol
such that $a(x, D)$ with compactly supported distribution
kernel (this formula is stated in the proof of \cite{Hormander3} and in Equation (3.32) in
\cite{HAbelsBook}, but see also Equation (2.53) in \cite{HintzLN}).

\subsection{Pseudodifferential operators on manifolds}
The material of this section is completely standard, so we do not include proofs
or references, but, in addition to the textbooks mentioned above, see also
Seeley's paper \cite{SeeleyManifolds}.
Let \m{M} be a  \emph{smooth manifold,} assumed Hausdorff and paracompact,
for simplicity, as before. Let also \m{U \subset M } be open and let
\m{\phi : U \simeq W \subset \RR^n} be a
diffeomorphism. We use the notation of Remark \ref{rem.precise.diffeo}. We then define
\begin{multline}\label{eq.def.pseudos.onM}
  \Psi_{comp}^{m}(U) \ede \phi_*^{-1} \Psi_{comp}^{m}(W)\phi_* \ede
  \{ \phi_*^{-1} a(x, D) \phi_* \mid a \in S_{1, 0}^m(\RR^{2n})\\
  \mbox{ and } k_{a(x, D)} \mbox{ compactly supported in } W \times W \}\,.
\end{multline}
Theorem \ref{thm.diffeo.inv1} shows that the definition
of $\Psi_{comp}^{m}(U)$, Equation \eqref{eq.def.pseudos.onM},
does not depend on the choice of the diffeomorphism $\phi$.

\begin{definition}\label{def.psdo}
  A linear map \m{P : \CIc(M) \to \CI(M)} is a  \emph{pseudodifferential operator}
  of order \m{\le m} on \m{M} with symbol of class $S_{1,0}^m$ if, for any \m{U \simeq W \subset \RR^n} and
  \m{\eta \in \CIc(U),} we have \m{\eta P \eta \in \Psi_{comp}^{m}(U),} where
  $\Psi_{comp}^{m}(U)$ is as in Equation \eqref{eq.def.pseudos.onM}.
  We let \m{\Psi_{1,0}^m(M)} denote the set of all \emph{pseudodifferential operators}
  on \m{M} with symbols in the class $S_{1,0}^m$.
\end{definition}

Similarly, we define
\begin{multline}\label{eq.def.symb.gen}
  S_{1,0}^m(T^*M) \ede \{ b : T^*M \to \CC \mid
  \phi_*(b\vert_{U}) \ede (b\vert_{U}) \circ \phi^{*-1} \in S_{unif}^m(T^*W)\\
  \mbox{ for } \phi : U \simeq W \subset \RR^n
  \,,\ U \Subset M \,,\ W \Subset \RR^n \}
\end{multline}
and
\begin{equation}\label{eq.def.ps}
    \sigma_m : \Psi_{comp}^{m}(U) \to S_{unif}^{m}(T^*U)/S_{unif}^{m-1}(T^*U)
\end{equation}
by the formula
\begin{equation}\label{eq.def.ps2}
  \sigma_m(\phi_{*}^{-1} a(x, D)\phi_{*}) \ede  (\phi_*)^{-1}(\sigma_m(a(x, D))) \ede
  \sigma_m(a(x, D)) \circ \phi^{*-1}\,.
\end{equation}
Theorem \ref{thm.diffeo.inv2} shows that the definition
of $\Psi_{comp}^{m}(U)$, Equation \eqref{eq.def.ps},
does not depend on the choice of the diffeomorphism $\phi : U \to W$.
Moreover, for each $m \in \RR$, the map of Equation \eqref{eq.def.ps}
extends to a \emph{surjective map}
\begin{equation} \label{def.princ.symb}
  \sigma_m : \Psi_{1,0}^{m}(M) \to  S_{1,0}^{m}/ S_{1,0}^{m-1}(T^*M)
  \ede  S_{1,0}^{m}(T^*M)/ S_{1,0}^{m-1}(T^*M)\,.
\end{equation}

\begin{theorem}\label{thm.prop.psdos.gen}
Let $M$ be a smooth manifold, as before.
  \begin{enumerate}
    \item \m{\Psi_{1,0}^m(M) \subset \Psi_{1,0}^{m'}(M)} if \m{m < m'.}

    \item If \m{P \in \Psi_{1,0}^m(M)}, then \m{P : \CI_{  c}(M) \to \CI(M)}
    \emph{continuously} and hence $P$ has a distribution kernel \m{k_P}.

    \item \m{\Psi_{1,0}^m(M)} and \m{\sigma_m : \Psi_{1,0}^m(M) \to  S_{1,0}^{m}/ S_{1,0}^{m-1}(T^*M)}
    are  \emph{diffeomorphism invariant.}

    \item \m{\Psi_{1,0}^m(M)} contains all the  \emph{differential operators} of order \m{m}
    with smooth coefficients. In particular, \m{\CI(M) \subset \Psi_{1,0}^{0}(M).}

    \item \m{\Psi_{1,0}^m(M)} contains all operators with  \emph{smooth distribution kernel.}

    \item If \m{\phi} and \m{\psi} are smooth functions on \m{M}, then
    \begin{equation*}
      k_{\phi P \psi } \seq \phi k_{P} \psi\,,
    \end{equation*}
    Consequently, \m{k_P} is  \emph{smooth}  \emph{away from the diagonal}
    and, \emph{at the diagonal,} \m{k_P} has the  \emph{same behavior as the kernels of the
    operators in} \m{\Psi_{unif}(\RR^n)}.
  \end{enumerate}
\end{theorem}

If $P \in \Psi_{1,0}^m(M)$ is such that $P \CIc(M) \subset \CIc(M)$ and this action
extends by continuity to a linear map $P : \CI(M) \subset \CI(M)$, then we say
that $P$ is \emph{properly supported.}

\begin{theorem}\label{theorem.prod.gen}
  Let us assume that a smooth measure was chosen on $M$.
  \begin{enumerate}
    \item Let $P \in \Psi_{1,0}^m(M)$ and let $P^*$ be defined with respect to
    the inner product on $M$. Then $P^* \in \Psi_{1,0}^m(M)$. We have that $P$ is
    properly supported if, and only if, $P \CIc(M) \subset \CIc(M)$
    and $P^* \CIc(M) \subset \CIc(M)$. Thus $P$ is properly supported if,
    and only if, $P^*$ is properly supported.

    \item If $P \in \Psi_{1,0}^m(M)$ and $Q \in \Psi_{1,0}^{m'}(M)$
    and at least one is properly supported, then $PQ$ is defined,
    belongs to $\Psi_{1,0}^{m+m'}(M)$, and
    $$\sigma_{m+m'}(PQ) = \sigma_{m}(P)\sigma_{m'}(Q).$$
    If both $P$ and $Q$ are properly
    supported, then $PQ$ is also properly supported.
  \end{enumerate}
\end{theorem}

We also notice that, if $M = \RR^n$, then $\Psi_{unif}^m(\RR^n) \subset \Psi_{1,0}^m(\RR^n)$,
but we do not have equality (if $n > 0$). Similarly, $S_{unif}^m(\RR^n)
\subset S_{0, 1}^m(\RR^n)$, but we again fail to have equality (in general).

\subsection{Compact manifolds}
We can say more if our manifold $M$ is compact. Assume this is the case in this
subsection. To start with, all operators in $\Psi_{1,0}^m(M)$ are properly supported and
hence we can compose any two pseudodifferential operators on $M$. This simplifies
matters, but the main result we are looking for is the Fredholm property
(whose definition is reminded below). Recall that, for $M$ compact, the Sobolev
spaces $H^s(M)$ defined in Subsection \ref{ssec.Sobolev} are independent of
choices (they can be defined either using a finite partition of unity or a
connection).

\begin{theorem}\label{thm.mp.compact}
  Let us assume $M$ to be compact.
  \begin{enumerate}
    \item $\Psi_{1, 0}^\infty(M) :=
    \cup_{k \in \ZZ} \Psi_{1, 0}^k(M)$ is a filtered $*$-algebra.

    \item If \m{P \in \Psi_{1,0}^m(M)}, then \m{P : H^s(M) \to H^{s-m}(M)} is
    \emph{bounded} for all $s \in \RR$.
    If this map is invertible for some $s \in \RR$, then \m{P^{-1} \in \Psi_{1,0}^{-m}(M).}
    \item
    \m{H^s(M) \to H^{s'}(M)} is compact if \m{s > s'}, so
    any \m{P \in \Psi_{1,0}^q(M),} with \m{q < 0,} induces a  \emph{compact operator}
    \m{P : H^s(M) \to H^s(M).}
  \end{enumerate}
\end{theorem}

``Filtered'' in point (1) above means that the order of the product of two operators does
not exceed the sum of their orders. The point (2) goes back to Beals \cite{BealsSpInv}.
The point (3) is, of course, nothing but Rellich's lemma.
We notice that $\cup_{k \in \ZZ} \Psi_{1, 0}^k(M) =
\cup_{m \in \RR} \Psi_{1, 0}^m(M),$ but this is not true for the classical
pseudodifferential operators discussed below.
The Fredholm property is related to ellipticity, which we now discuss.

\begin{proposition} \label{prop.def.ell}
  Let \m{a \in  S_{1,0}^m(T^*M)}. The following are equivalent
  \begin{enumerate}
    \item There are \m{C, R > 0} such that \m{|a(\xi)| \ge C |\xi|^m} for
    \m{\xi \in T^*M\,,\ |\xi| \ge R.}

    \item There exist \m{b \in  S_{1,0}^{-m}(T^*M)} such that
    \m{ab - 1 \in  S_{1,0}^{-1}(T^*M).}
  \end{enumerate}
  If one of these properties is satisfied, then \m{a} is called
  \emph{elliptic} and this property is a property of
  its class in $S_{1,0}^m/ S_{1,0}^{m-1}(T^*M).$
\end{proposition}

(If $M$ is non-compact, the two properties of the above proposition are no longer equivalent,
for some trivial reasons, and we retain the second one as the definition of ellipticity. This
will not be a problem in our paper.) This leads to the following definition.

\begin{definition} \label{def.elliptic}
  \m{P \in \Psi_{1,0}^m(M)} is  \emph{elliptic} if
  \m{\sigma_m(P) \in  S_{1,0}^m/ S_{1,0}^{m-1}(T^*M)} is elliptic.
\end{definition}

Below, we shall also need the notion of ``injectively elliptic'' operators, which is
completely analogous (Definition \ref{def.inj.elliptic}).

\subsection{Classical symbols and vector bundles}
We now recall the definition of a smaller class of symbols that is better suited for
the study of the parametrices of elliptic differential operators, namely, the class of
``classical symbols.'' It is also better suited for the study of the method of layer potentials.
We continue to assume that $M$ is a smooth (Hausdorff, paracompact)
manifold, but we no longer assume $M$ to be compact.

\begin{definition}
  We shall say that \m{a \in S_{1,0}^{m}(T^*M)} is homogeneous of order \m{m,} if
  \dm{a_{k}(t\xi) = t^k a_k(\xi) \ \mbox{ for } \ \xi \in T^*M,\ t, |\xi| \ge 1.}
  We shall say that \m{a \in  S_{1,0}^m(\RR^n)} is \emph{classical} of order $m$ if
  there exist \m{a_{m-j} \in  S_{1,0}^{m-j}(M)} homogeneous of order \m{m-j,}
  $j \in \ZZ_+$ such that $a_{m}$ does not vanish identically and
  \dm{a - \sum_{j=0}^N a_{m-j} \in  S_{1,0}^{m-N-1}(M)\,.}
\end{definition}

In other words, we have the asymptotic expansion $a \sim \sum_{j=0}^\infty a_{m-j}.$
If we allow only \text{classical symbols,} then all the results of this
section so far remain true, except \ref{thm.prop.psdos.gen}(1).

Let $S_{cl}^{m}(T^*M)$ be the set of classical symbols of order $ \in \RR$.
Similarly, if $U \subset \RR^n$, we let $\Psi_{c, cl}^m(U) := \{a(x, D) \in \Psi_{comp}^m(U)
\mid a \in S_{cl}^m(T^*M)\}$ and, finally, we let $\Psi_{cl}^m(M)$ be defined using
the sets $\Psi_{c, cl}^m(U)$ with $U \Subset M$ open. Whenever using
classical pseudodifferential operators we shall include the index ``cl''
(or no index at all, beginning with the next section).
We let $S_{cl}^{m}/S_{cl}^{m-1}(T^*M) := S_{cl}^{m}(T^*M)/S_{cl}^{m-1}(T^*M)$. In addition to
the property proved so far, the classical pseudodifferential operators enjoy some
additional properties.

\begin{proposition}
  Let $m \in \RR$ and $S^*M$ be the set of vectors of length 1
  in $T^*M.$ Then restriction to $S^*M$ defines a bijection
  \begin{equation*}
    S_{cl}^m/S_{cl}^{m-1}(T^*M) \ede S_{cl}^m(T^*M)/S_{cl}^{m-1}(T^*M)
    \simeq \CI(S^*M)\,.
  \end{equation*}
  In particular, $\sigma_m : \Psi_{cl}^{m}(M) \to \CI(S^*M)$
  has kernel \m{\Psi_{cl}^{m-1}(M)}.
\end{proposition}

For a general statement of the Fredholm result, we now explain how we can include
\emph{vector bundles.}

\begin{remark}
   Let \m{{  E} \to M} be a complex vector bundle.
  Without loss of generality, we can assume that a metric has been chosen on $E$.
  \begin{enumerate}
    \item Let \m{E \subset \CC^N} be a smooth embedding and let
    \m{e \in \CI(M; M_N(\CC))} be the orthogonal projection \m{\CC^N \to E.}

    \item Then \m{\Psi_{1,0}^m(M; E) \ede e M_N(\Psi_{1,0}^m(M))e} acts on \m{H^s(M; E) \simeq e H^s(M).}

    \item The same constructions applies to classical pseudodifferential operators, but
    recall that $\Psi_{cl}^{\infty}(M; E) := \cup_{m \in \ZZ} \Psi_{cl}^{m}(M; E)$
    and that this space no longer contains $\Psi_{cl}^s(M; E)$ for $s \notin \ZZ$.

    \item The principal symbols then become
    \begin{equation*}
      \begin{gathered}
        \sigma_m: \Psi_{1, 0}^m(M; E) \to S_{1, 0}^m/ S_{1, 0}^{m-1}(T^*M; \End(E))
        \quad \mbox{and}\\
        \sigma_m: \Psi_{cl}^m(M; E) \to S_{cl}^m/ S_{cl}^{m-1}(T^*M; \End(E))
        \simeq \CI(S^*M; \End(E))\,.
      \end{gathered}
    \end{equation*}
  \end{enumerate}
\end{remark}

The distribution kernels of operators acting on vector bundles change in an
evident way in the following way:

\begin{remark}\label{rem.boxtimes}
  Let $V_i \to M_i$ be vector bundles, $i = 1, 2$.
  Then $V_1 \boxtimes V_2 \to M_1 \times M_2$ will denote their external tensor product.
  Thus $(V_1 \boxtimes V_2)_{(x_1, x_2)} = (V_1)_{x_1} \otimes (V_2)_{x_2}$. We are
  interested in the external tensor product because
  the distribution kernels of the operators in $\iPS{m}(\manif; E; F)$ will be distributions
  with values in $F \boxtimes E'$. (Thus, if $E = F = \CC$, then $E \boxtimes E' = \CC$, which
  simplifies the notation and recovers the old definition of the distribution kernels.)
\end{remark}

\subsection{Fredholm operators and ellipticity}

Let us recall next a few more needed general definions and results.

\begin{definition}
  Let \m{X, Y} be normed spaces. The space of continuous linear maps
  \m{T : X \to Y} will be denoted by \m{\maB(X; Y)}. If \m{X}
  and \m{Y} are complete (i.e. Banach), then \m{T \in \maB(X; Y)}
  is \emph{Fredholm} if
  \begin{center}
    \m{\dim\ker(P)\,, \ \dim(Y/PX) < \infty.}
  \end{center}
\end{definition}

The following theorem \cite{H-W, Wl-Ro-La} will be important for us.

\begin{theorem}[Atkinson]
  \label{thm.Atkinson}
  Let \m{X} and \m{Y} be
  Hilbert spaces; \m{T \in \maB(X; Y)} is  \emph{Fredholm} if,
  and only if, there exists \m{Q \in \maB(Y; X)} such that
  both \m{TQ-1_Y} and \m{QT - 1_X} are compact.
\end{theorem}

An important property of compact manifolds and of their elliptic
operators is the Fredholm property. We include a proof for
convenience.

\begin{theorem} \label{thm.Fredholm.compact}
  Let $E, F \to M$ be vector bundles over a complete Riemannian manifold
  $M$ and let $P \in \Psi_{1,0}^{m}(M; E, F)$, $m \in \RR.$
  \begin{enumerate}
    \item Assume that $M$ is compact and that
    $P$ is \emph{elliptic.} Then, for any \m{s \in \RR,}
    $P : H^s(M; E) \to H^{s-m}(M; F)$ is \em Fredholm.

    \item Conversely, if $P : H^s(M; E) \to H^{s-m}(M; F)$ is Fredholm
    for some $s \in \RR$, then $P$ is elliptic.
  \end{enumerate}
\end{theorem}

Note that for the converse implication we no longer assume \m{M}
to be compact.

\begin{proof}
  There is \m{b \in  S_{1,0}^{-m}(M; \End(E))}
  such that \m{\sigma_m(P)b - 1 \in  S_{1,0}^{-1}(M; \End(E))}
  because \m{P} is elliptic.
  The surjectivity of \m{\sigma_{-m}} and the multiplicativity
  of the principal symbol give that there exists \m{Q \in \Psi_{1,0}^{-m}(M; E)}
  such that \m{PQ -1 , QP - 1 \in \Psi_{1,0}^{-1}(M; E)\,.}
  Consequently, \m{PQ -1} and \m{QP - 1} are compact operators.
  Atkinson's theorem \ref{thm.Atkinson} then gives that \m{P} is Fredholm.

  The proof of the converse statement follows from the following sequence of
  lemmas of independent interest from \cite{GKN_in_progress}.
\end{proof}

For the following lemmas (obtained in discussions with Nadine Gro\ss e
\cite{GKN_in_progress}), we keep the assumptions of Theorem \ref{thm.Fredholm.compact},
unless explicitly stated otherwise (some of these assumptions are reminded
occasionally). We follow Section 19.5 in \cite{Hormander3}.

\begin{lemma}\label{lemma.SF2CE}
  Let $M$ be a complete Riemannian manifold, $m, s \in \RR$, and let
  $P : H^s(M; E) \to H^{s-m}(M; F)$ be a continuous linear map that has
  finite dimensional kernel and closed range. Then, for every $t \le s$
  there exists $C_t >0$ such that, for every $u \in H^{s}(M; E)$, we
  have the \emph{$L^2$-coercive regularity estimate}
  \begin{equation}\label{eq.coerc.reg}
    \|u\|_{H^{s}(M; E)} \le C_t \big ( \|Pu\|_{H^{s-m}(M; F)} +
    \|u\|_{H^{t}(M; E)} \big)\,.
  \end{equation}
\end{lemma}

For operators acting on vector bundles, the injectivity of the principal symbol does not
imply its invertibility. We thus introduce the following definition.

\begin{definition} \label{def.inj.elliptic}
  A symbol $a \in S_{1, 0}^m(T^*M; \Hom(E; F))$ is
  \emph{injectively elliptic} if there exists $b \in S_{1, 0}^{-m}(T^*M; \Hom(E; F))$ such that
  $ba - 1 \in S_{1, 0}^{-1}(T^*M; \End(E))$ (i.e. $a$ is left invertible modulo
  lower order symbols). If $P \in \Psi_{1,0}^m(M; E, F)$, we say that $P$ is
  \emph{injectively elliptic} if $\sigma_m(P)$ is injectively elliptic.
\end{definition}

As usual, we obtain the following alternative description of injectively elliptic
symbols.

\begin{lemma}\label{lemma.inj.elliptic}
  We have that $a \in S_{1,0}^m(T^*M; \Hom(E; F))$
  is injectively elliptic if, and only if, for every $x \in M$ and any compact
  neighborhood $V$ of $x$ in $M$, there exist $R_V, C_V > 0$ such that, for all
  $w \in E\vert_{V}$ with $\|w\|\ge R_V$, we have $$\|a(\xi) w \| \ge C_V \|w\|\,.$$
\end{lemma}

It is easy to see then that $a \in S_{1, 0}^m{T^*M; \Hom(E; F)}$ is injectively elliptic
if, and only if, $a^*a$ is elliptic (in $S_{1, 0}^{2m}(T^*M; \End(E))$). (This is
used, implicitly, in the proof of Theorem 19.5.1 in \cite{Hormander3}.) Also, it
follows from definitions (since $E$ and $F$ are finite dimensional vector
bundles), that $a \in S_{1,0}^m(T^*M; \Hom(E, F))$ is elliptic if, and only if,
both $a$ and $a^*$ are injectively elliptic.

\begin{lemma}\label{lemma.CE2IE}
  Let $M$ be a complete Riemannian manifold and $m, s, t \in \RR$, $t < s$.
  Let $U \subset M$ be open and let $P \in \Psi_{1,0}^{m}(M; E, F)$ satisfy an
  \emph{$L^2$-coercive regularity  estimate} in $U$, that is, there exists $C_t > 0$
  such that, for all $u \in \CIc(U; E)$, we have the estimate of Equation \eqref{eq.coerc.reg}.
  Then $P$ is injectively elliptic in $U$.
\end{lemma}

We can now complete the proof of Theorem \ref{thm.Fredholm.compact}. Recall that we
have proved the direct part and it remains to prove the converse, that is point (2)
of that Theorem.

\begin{proof}[End of the proof of Theorem \ref{thm.Fredholm.compact}]
  We use twice the last two lemmas, first for $P$, and then for $P^* \in \Psi^m(M; F, E)$.
  Since $P : H^s(M; E) \to H^{s-m}(M; F)$ is Fredholm, it has finite dimensional kernel and
  closed range. Lemma \ref{lemma.SF2CE} implies that $P$ satisfies the $L^2$-coercive
  regularity estimate of Equation \eqref{eq.coerc.reg} with $t = s-1$. In turn, Lemma
  \ref{lemma.CE2IE} then gives that $P$ is injectively elliptic. The same argument
  applied to the Fredholm operator $P^*$ gives that $P^*$ is also injectively
  elliptic. Consequently, $P$ is elliptic.
\end{proof}

The same ideas as the ones used in proof of Lemma \ref{lemma.CE2IE} in \cite{GKN_in_progress}
give the following result. (It is also a direct consequence of the analogous results in 
the case of compact manifolds \cite{Hormander3, Taylor2}.)

\begin{proposition}\label{prop.compact0}
  Let $M$ be a complete Riemannian manifold, $m, s \in \RR$, and $P \in \Psi_{cl}^m(M; E, F)$
  be such that $P : H^s(M; E) \to H^{s-m}(M; F)$ is compact. Then $\sigma_m(P) = 0$.
\end{proposition}

\section{Translation invariance in a neighborhood of infinity}
\label{translation-inv-oper}

Pseudodifferential operators have long played an important role in the study of layer potential
operators \cite{H-W, Taylor2}.
In this paper, we consider mostly two classes of pseudodifferential operators associated to
our manifolds with straight cylindrical ends. These two classes differ only in the regularizing
operators in a sense that will be made clear below. In this section, we introduce
the smallest of these two classes of operators and study their properties.
The operators in this class are ``$b$-pseudodifferential''
operators (in the sense of Melrose \cite{MelroseActa, MelroseAPS} and Schulze
\cite{Schulze, SchulzeBook91}; see the nice paper
by Lauter and Seiler \cite{LauterSeiler} for a comparison of the constructions of
Melrose and Schulze). A few of the results in the following two subsections are consequences
of the properties of the $b$-calculus and many are similar to some analogous results
that are valid for the $b$-calculus. We nevertheless include the proofs, for the
convenience of the reader and because our approach is different from the one used for
the $b$-calculus. We let $n$ denote the dimension of our manifold with cylindrical ends $\manif$.

\subsection{Definition of the $\inv$-calculus}
\label{ssec.transl.inv}
Let $\manif = \manif_0\cup \left(\partial \manif_0\times (-\infty,0]\right)$ be our fixed manifold
with straight cylindrical ends, as in Remark \ref{rem.def.ce}. Recall that we assume that
$\manif$ has compact ends, that is, that $\manif_0$ is compact. This assumption will remain
in place until the end of the paper.

For $R, s\geq 0$, we let
\begin{equation}\label{eq.def.Phi}
  \Phi _s :\partial \manif_0 \times (-\infty , -R]\to \partial \manif_0\times (-\infty ,-R-s] \,,
  \quad \Phi _s(x,t)=(x,t-s) \,,
\end{equation}
denote the bijective isometry given by translation with $-s$ in the $t$-direction.
If $s<0$, then $\Phi _s$ is defined as the inverse of $\Phi _{-s}$.
Let also $\Phi_s(f) := f\circ \Phi _s$, whenever the compositions are defined.
If $\varepsilon >0$, the set
\begin{equation}\label{eq.def.eps_neigh}
  \{(x,y)\in \manif \times \manif \mid \dist_g(x,y) < \varepsilon \}
\end{equation}
will be called \emph{the $\varepsilon$--neighborhood of the diagonal}
$\delta_\manif := \{(x,x) \mid x \in \manif \}$ of $\manif$.

\begin{definition} \label{def.translation-invariant}
  A linear and continuous operator $P : \CIc(\manif)\to \CI(\manif)'$ is called
  {\it translation invariant in a neighborhood of infinity} if:
  \begin{enumerate}
    \item there exists $\varepsilon_P > 0$ such that its Schwartz
    (or distribution) kernel is supported in the
    $\varepsilon_P$--neighborhood of the diagonal
    (Equation \eqref{eq.def.eps_neigh}).

    \item \label{item.transl-inv} there exists $R_P >\varepsilon_P$ such that, for all
    $u \in \CIc \big(\partial \manif_0\times (-\infty ,-R_P)\big)$ and all $s>0$,
    we have $P \Phi _s(u) = \Phi _s P(u).$
  \end{enumerate}
  We let $\iPS m(\manif)$ denote the {\it space of all pseudodifferential operators $P \in \Psi_{1,0}^m(\manif)$
  of order $\le m \in \RR \cup \{-\infty\}$ that are translation invariant in a neighborhood of infinity}.
  We let $\iPS{\infty}(\manif) := \cup_{m \in \ZZ} \iPS{m}(\manif)$.
\end{definition}

\begin{remark}
Let us notice that, if $u$ has support in $\partial \manif_0\times (-\infty ,-R_P)$, then
$\Phi_s(u)$ has support in $\partial \manif_0\times (-\infty ,-R_P + \varepsilon_P)$, so $P\Phi_s (u)$
and $\Phi_s(Pu)$ are both defined since $R_P > \varepsilon_P$. Consequently, the
condition $P \Phi _s(u) = \Phi _s P(u)$ makes sense.
\end{remark}

The calculus $\iPS{m}(\manif)$ will be called the \emph{$\inv$-calculus.}
Our $\inv$-calculus is smaller than the ``$b$-calculus'' in the following sense.

\begin{remark} \label{rem.inv.subset.b}
  Let us compactify $\manif$ to $\overline{\manif}$ using the function
  $r := e^{t}$ (the Kondratiev transform), so that $\manif = \{r > 0\}$ and
  $\pa \overline{\manif} = \{r = 0\}$. Then it is not difficult to show that
  $$\iPS{m}(\manif) \subset \Psi_b^{m}(\overline{\manif})\,,$$
  where $\Psi_b^{m}(\overline{\manif})$ is the ``$b$-calculus'' of Melrose
  \cite{MelroseActa, MelroseAPS} and Schulze \cite{Schulze, SchulzeBook91}.
  (Schulze calls it the ``cone-calculus,'' see the nice paper by Lauter and
  Seiler \cite{LauterSeiler} comparing the two approaches for further references.
  See also \cite{CSS07, GrieserBCalc, SchulzeWongBCalc, Seiler1999}.) Some of the properties of
  $\iPS{m}(\manif)$ proved in this paper would follow from this inclusion.
  The space $\Psi_b^{m}(\manif)$ is, nevertheless, not too far from our space
  $\iPS{m}(\overline{\manif})$,
  since the former consists of operators whose kernels are ``exponentially close'' to
  being translation invariant, whereas ours are exactly translation invariant (in a neighborhood
  of infinity). Our proofs can thus be regarded as a different approach to the $b$-calculus.
\end{remark}

We have already encountered operators in the $\inv$-calculus $\iPS{m}(\manif)$,
as explained in the next remark.

\begin{remark}\label{rem.def.CIinv}
  The endomorphisms and, more generally, the differential operators in
  $\iPS{m}(\manif)$ have already been used.
  Indeed, it follows from definitions that
  \begin{equation*}
    \begin{gathered}
    \CI_{\inv}(\manif) \seq \CI(\manif) \cap \iPS{m}(\manif)
    \quad \mbox{ and}\\
    \operatorname{Diff}_{\inv}^m(\manif) \seq \{P
    \mbox{ a differential operator on } \manif \} \cap \iPS{m}(\manif)\,.
    \end{gathered}
  \end{equation*}
\end{remark}

\subsection{The first properties of the $\inv$-calculus}
In this subsection, we recall the definition of the $\inv$-calculus $\iPS{m}(\manif)$
and prove its first properties.
These properties are direct analogues of the corresponding properties for the
classical pseudodifferential operators on compact manifolds (or of the $b$-calculus).
However, none of the results of this subsection seems to follow from the
inclusion of our $\inv$-calculus in the $b$-calculus. We begin with a characterization
of the distribution kernels of the negative order operators.

We obtain more examples of pseudodifferential operators that are translation invariant
in a neighborhood of infinity by describing their distribution kernels.
Recall that $k_P \in \CIc(\manif \times \manif)'$ denotes the distribution kernel
of a continuous, linear operator $P : \CIc(\manif) \to \CIc(\manif)'$, Theorem
\ref{thm.Schwartz.k}. Also, recall that
$\manif$ has dimension $n$, that $r_{\manif}$ denotes its injectivity radius, and
that $\NN = \{1, 2, \ldots \}$.
We let $\exp: T\manif \to \manif$ be the exponential map (which is defined everywhere since
$\manif$ is geodesically complete).

\begin{proposition}\label{prop.kernel.descr}
  Let $\manif$ be a manifold with straight cylindrical ends, as before,
  and $m < 0$, $m \notin \ZZ$. A continuous, linear operator $P : \CIc(\manif) \to \CIc(\manif)'$ is
  \emph{a classical pseudodifferential} operator in
  $\iPS{m}(\manif)$ if, and only if, it satisfies the following conditions:
  \begin{enumerate}
    \item \label{item.cond.one}
    Its distribution kernel $k_P$ is a function with support in some
    $\epsilon$-neighborhood of the diagonal (Equation \ref{eq.def.eps_neigh})
    and is smooth away from the diagonal.

    \item \label{item.cond.two} There exists $R_P > 0$ such that,
    if $x, y \in \pa \manif_0$, $t, s \le -R_P$, $(x, t) \neq (y, s)$, and $\lambda \ge 0$, then
    \begin{equation*} 
      k_P(x, t-\lambda, y, s - \lambda) \seq k_P (x, t, y, s) \,.
    \end{equation*}

    \item \label{item.cond.three}
    Let $S\manif \subset T\manif$ be the subset of vectors of length one
    and $p : S\manif \to \manif$ be the canonical projection.
    Then there exist functions $a_j \in \CI(S\manif)$ and
    $C_j \ge 0$, $j \in \ZZ_+$, such that, for all $\xi \in S\manif$, all $t \in (0, r_\manif /2]$, and
    all $N\in \NN$, we have
    \begin{equation*}
      |k_P(p(\xi), \exp(t \xi)) - \sum_{j=0}^N a_j(\xi) t^{-n -m + j}| \le
      C_N t^{-n-m + N + 1} \,.
    \end{equation*}
  \end{enumerate}
  If that is the case, then $a_j \in \CI_{inv}(S\manif)$, more precisely,
  $a_j(\eta, t) = a_j(\eta, t-s)$ for
  all $s \ge 0$ and $t < R_P$ (independent of $j$!), where $(\eta, t)$ are the canonical
  coordinates on the cylindrical end of $S\manif$.

  If $-m \in \NN$, the result is true if, for each $j \ge m+n$, we replace the term
  $a_j(\xi)t^{-n -m + j}$ with $a_j(\xi)t^{-n -m + j} + p_j(t\xi) \ln t$, where
  $p_j$ is a polynomial in $\xi$ of degree $-n-m+j$ with smooth coefficients.
\end{proposition}

\begin{proof}[Proof of Proposition \ref{prop.kernel.descr}]
  This follows almost directly from the definitions and classical results, see,
  for instance, Proposition 2.8 of \cite{Taylor2}.
  The support condition, Condition \eqref{item.cond.one} is part of the definition of
  $\iPS{m}(\manif)$.
  Condition \eqref{item.cond.two} is equivalent to the invariance
  of the operators in a neighborhood of infinity. Finally, the last condition,
  Condition \eqref{item.cond.three}, states that $P$ is an order $m < 0$ classical pseudodifferential
  operator, see \cite{Hormander3, Taylor2, Wl-Ro-La}. This proves the desired
  equivalence. The last statement (including $a_j \in \CI_{inv}(S\manif)$) follows from the uniqueness of
  the coefficients $a_j$ and the translation invariance at infinity of the operator $P$ (with $R_P$
  the same in our statement and in the definition of translation invariance at infinity for $P$).
\end{proof}

Recall the definition of the principal symbol, Equation \eqref{def.princ.symb}.
We notice that the principal symbol of a pseudodifferential operator that is translation
invariant in a neighborhood of infinity is also translation invariant at infinity. We
spell this and some related needed facts in the following remark.

\begin{remark}\label{rem.def.inv.symb}
The manifold
$T^*\manif$ also has cylindrical ends (but does not have compact ends).
The translation invariance at infinity of a pseudodifferential operator implies the
same property for its principal symbol (by the behavior of the principal symbol with
respect to diffeomorphisms). Let
\begin{equation*}
    S_{inv}^m(T^*\manif) \ede S_{1, 0}^m(T^*\manif) \cap \CI_{inv}(T^*\manif\,.)
\end{equation*}
We obtain that $\sigma_m$ defines a linear map
\begin{equation*} 
  \sigma_m : \iPS{m}(\manif) \to S_{inv}^m(T^*\manif)/S_{inv}^{m-1}(T^*\manif)\,.
\end{equation*}
\end{remark}

We mention the following property that is not satisfied by the $b$- or $c$-calculi.


\begin{remark}\label{rem.ell.imp.ue}
Corollary \ref{cor.lemma.ue} implies that every elliptic operator $P \in \iPS{m}(\manif)$
is also \emph{uniformly elliptic,} in fact, $\sigma_m(P)^{-1}$ is a principal symbol of order $-m$
by Corollary \ref{cor.lemma.ue}. Again, this is one of the reasons which shows that the
$\inv$-calculus is simpler than the $b$-calculus and why we relegate the use of the $b$-calculus
and of other related calculi, such as the $c$-calculus, to another paper.
\end{remark}

The following remark introduces the quantization map.

\begin{remark}\label{rem.def.quant.map}
Let $S_{\inv}^m(T^*\manif)$ be as in Remark \ref{rem.def.inv.symb}.
That is, it is the set of (scalar) symbols of order $\le m$ that
are translation invariant in a neighborhood of infinity in the usual sense, that is that
$a \in S_{\inv}^m(T^*\manif)$ if, and only if, there exists $R_a > 0$ such that
$a(x, t, v) = a(x, t-s, v)$ for all $s > 0$ and $(x, t, v) \in T^*\manif$ with
$(x, t) \in \manif$ cylindrical coordinates with $x \in \pa \manif_0$ and $t < -R_a$
and $v \in T_{(x, t)}^*\manif$. Let
$\chi : \manif \times \manif \to [0, 1]$ be a smooth cut-off function that is supported in an
$\varepsilon$-neighborhood of the diagonal, is equal to 1 in a $\varepsilon'$-neighborhood
of the diagonal, and is translation invariant in a neighborhood of infinity, in the sense that it
satisfies the condition of Proposition \ref{prop.kernel.descr}\eqref{item.cond.two}.
We shall also need the usual quantization map
\begin{equation*}
  \begin{gathered}
    q : S^m(T^*\manif) \to \Psi_{1,0}^{m}(\manif) \,,\\
    \big[q(a) u\big](x) \ede (2\pi)^{-n} \int_{T_x^*\manif} a(w)
    \left (\int _{T_x\manif} e^{-\imath \langle w, z \rangle}
    \chi(x, \exp(z)) u(\exp(z)) dz\right ) dw\,.
  \end{gathered}
\end{equation*}
(To see that this recovers the usual formula, one should notice that the value at $x$
is computed in coordinates centered at $x$, so it is natural to use the substitution $z = y -x$,
which leads us to the usual formula, except maybe for the cut-off factor $\chi$.) The volume forms
(measures) on $T_x^*\manif$ and $T_x\manif$ are those induced by the metric.
Then $q (S_{inv}^m(T^*\manif)) \subset \Psi_{inv}^{m}(\manif)$
\end{remark}


We have the following notion of asymptotic completeness.

\begin{remark} \label{rem.asympt.compl}
  Let $m \in \RR$ and $a_j \in S_{\inv}^{m-j}(T^*\manif)$, $j \in \ZZ_+$ be such that one can
  choose \emph{the same $R_{a_j}$ for all $j$} (i.e. there exists $R > 0$
  such that $a_j(x, t - s, v) = a_j(x, t, v)$ for all $j \in \ZZ_+$, $t < -R$, and $s > 0$).
  Then there exists $a \in S_{\inv}^{m}(T^*\manif)$ such that
  \begin{equation*}
    a - \sum_{j=0}^N a_j \, \in \, S_{\inv}^{m-N-1}(T^*\manif)\,, \quad \forall\, N \in \NN\,.
  \end{equation*}
  This follows from the asymptotic completeness of the usual symbols.
  (Here we have not included the most general result, but rather the form of the
  result that we will need.) See Proposition 3.14 in \cite{HintzLN} for a useful general
  extension of this result.
\end{remark}

We have the following result extending one of the usual properties of pseudodifferential operators.

\begin{proposition} \label{prop.onto.q}
  The quantization map $q : S_{\inv}^m(T^*\manif) \to \iPS{m}(\manif)$, $m \in \ZZ$,
  of Remark \ref{rem.def.quant.map} satisfies
  \begin{equation*}
    \sigma_m(q(a)) \seq a + S_{\inv}^{m-1}(T^*\manif)\,.
  \end{equation*}
  Consequently, the induced map
  $$q : S_{\inv}^m(T^*\manif)/S_{\inv}^{-\infty}(T^*\manif)
  \to \iPS{m}(\manif)/\iPS{-\infty}(\manif)$$ is a linear bijection. The maps
  $q$ for all $m$ are compatible and they combine to a map $q : S_{\inv}^\infty(T^*\manif)
  \to \iPS{\infty}(\manif)$.
\end{proposition}

\begin{proof}
  The translation invariance of the exponential map shows that
  $q(a) \in \iPS{m}(\manif)$ whenever $a \in S_{\inv}^m(T^*\manif)$. The rest is
  as in the usual case of pseudodifferential operators (it is a consequence of
  the corresponding properties for pseudodifferential operators and of Remark
  \ref{rem.asympt.compl}).
\end{proof}

\begin{remark} \label{rem.top.op}
  Proposition \ref{prop.onto.q} gives, in particular, that the map
  \begin{equation*}
    \overline{q} : S_{\inv}^m(T^*\manif) \oplus \iPS{-\infty}(\manif)
    \ni (a, R) \to q(a) + R \in \iPS{m}(\manif)
  \end{equation*}
  is onto, which allows us to endow the later space with the quotient topology (of
  an inductive limit of Fr\'echet spaces). he kernel of this map is 
  $$\ker \overline{q} \seq \{(a, - q(a))\mid a \in S_{\inv}^{-\infty}(T^*\manif)\}\,.$$
\end{remark}

We summarize now some of the easy properties of the $\inv$-calculus
$\iPS{m}$ (these properties can be found in many of the works
on the $b$-calculus, once one takes into account the simpler definition of ellipticity
for the operators in the $\inv$-calculus, see
Theorem \ref{thm.diffeo.inv2} and Proposition \ref{prop.def.ell},
they are also implicit in \cite{Mitrea-Nistor}). Operators that are (almost)
translation invariant were also studied recently in \cite{HintzTrInv}.
We begin with the properties that are direct analogues of the corresponding
properties for the usual pseudodifferential calculus on compact manifolds.

\begin{theorem} \label{thm.prop.inv-calc}
  Let $\manif = \manif_0 \cup \big ( \pa \manif_0 \times (-\infty, 0] \big )$ be a manifold with
  straight cylindrical ends.
  \begin{enumerate}
    \item $\iPS{m}(\manif)$ is stable under adjoints and, for $m , q \in \RR$,
    \begin{equation*}
      \iPS{q}(\manif)\iPS{m}(\manif)
      \subset \iPS{q + m}(\manif)\,.
    \end{equation*}
    \item The principal symbol $\sigma_m : \iPS{m}(\manif) \to S_{inv}^m(T^*\manif)/S_{inv}^{m-1}(T^*\manif)$
    of Remark \ref{rem.def.inv.symb} is onto and its kernel is $\iPS{m-1}(\manif)$.
    It is $*$-multiplicative, in the sense that
    \begin{equation*}
      \sigma_{q + m}(QP) \seq \sigma_{q}(Q) \sigma_{m}(P)\ \mbox{ and }\
      \sigma_m(P^*) \seq \sigma_m(P)^*\,.
    \end{equation*}
  \end{enumerate}
\end{theorem}

\begin{proof}
  To prove (1), let $P \in \iPS{m}(\manif)$ and $Q \in \iPS{q}(\manif)$.
  Both operators are properly supported, by assumption, so the composition $QP$ is defined
  and $QP \in \Psi^{m+q}(\manif)$. Moreover, the support of the distribution kernel of
  $QP$ is contained in the $(\varepsilon_P + \varepsilon_Q)$--neighbhorhood of the diagonal.
  The invariance in a neighborhood of infinity follows from $QP\Phi_s(u)= Q\Phi_s Pu = \Phi_s QPu$
  whenever $u \in \CIc(\manif)$ has support in $\pa \manif_0 \times (-\infty, -R)$ for
  $R$ large enough (say $R > R_P + R_Q$). Thus $QP \in \iPS{q + m}(\manif)$. The relation
  $\iPS{m}(\manif)^* = \iPS{m}(\manif)$ is similar.

  The surjectivity in (2) is a consequence of Proposition \ref{prop.onto.q}.
  The multiplicativity follows from the usual multiplicativity of the principal symbols.
  The kernel of the usual principal symbol is $\Psi^{m-1}(\manif)$. The result then follows
  from $\Psi^{m-1}(\manif) \cap \iPS{m}(\manif) = \iPS{m-1}(\manif)$.
\end{proof}

\subsection{Limit operators}
An important concept associated to the operators in the $b$-calculus is
the ``normal operator'' (in Melrose's terminology) or ``conormal principal symbol''
(in Schulze's terminology, who uses the term ``cone-calculus'' instead of
the $b$-calculus). This concept is especially easy to define for our translation invariant
at infinity operators. We shall call these operators ``limit at infinity operators,''
(or ``limit operators,'' for short) since this is the terminology that is used for their
generalizations by many other authors. A few results of
this section follow from the properties of the $b$-calculus, but we include their
simple proof for completeness (and since the proofs are different than the ones for
the $b$-calculus). Recall that we assume that $\manif$ has compact ends (that is,
that $\pa \manif_0$ is compact). It is, of course, a manifold with straight cylindrical ends.

To simplify the notation, we shall set as in Melrose's work
\begin{equation}\label{eq.def.suspended}
  \iPSsus{m} {\link} \ede \iPS{m} (\link \times \RR)^{\RR}\,.
\end{equation}
We begin with some simple observations on the supports of our operators
that will help us to define the normal
(or limit) operator associated to an operator that is
translation invariant in a neighborhood of infinity.

\begin{remark}\label{rem.inv=pr}
  First, it follows from the definition of
  translation invariant pseudodifferential
  operators in a neighborhood of infinity that every $P\in \iPS {m}(\manif)$ is properly
  supported, that is, $P(\CIc(\manif)) \subset \CIc(\manif)$. We use this
  observation for $\manif = \link \times \RR$ to conclude that
  $\iPSsus{m} {\link} := \iPS{m} (\link \times \RR)^{\RR}$
  consists of the set of translation invariant, properly supported, pseudodifferential operators
  on $\link \times \RR$ (invariant with respect to the translations by $t \in \RR$),
  that is
  \begin{equation*}\label{eq.inv=pr}
    \iPSsus{m} {\link} \seq \{ P \in \Psi^{m}(\link \times \RR)^{\RR}
    \mid \ P \mbox{ properly supported}\}.
  \end{equation*}
\end{remark}

This type of operators appear in the following result that will provide the definition
of limit operators. Let $\Phi_s : \pa \manif_0 \times \RR \to \pa \manif_0  \times \RR$ be the
translation by $-s$, $\Phi_s(f)(x) = f(x -s)$, that is, we allow $R = -\infty$ in
\eqref{eq.def.Phi}.

\begin{lemma}\label{lemma.def.indicial}
  Let $P\in \iPS {m}(\manif)$ and $u \in \CIc(\pa \manif_0 \times \RR)$.
  Then, there exists $R_{P, u} > 0$ such that
  $P\Phi_s(u)$ is defined and independent of $s > R_{P, u}$. Let
  \begin{equation*} 
    \widetilde{P}(u) \ede \Phi _{-s}P\Phi_s(u)\,,
  \end{equation*}
  for any $s > R_{P, u}$.
  \begin{enumerate}
    \item If $a \in \CI_{\inv}(\manif) \subset \iPS{m}(\manif)$,
    then $\tilde{a} = \maR_{\infty}(a)$ (Definition \ref{def.CIinv}).
    \item In general, $\In (P) := \widetilde{P} \in \iPSsus{m} {\pa \manif_0}
    := \iPS{m} (\pa \manif_0 \times \RR)^{\RR}$.
  \end{enumerate}
\end{lemma}

\begin{proof}
  For $s$ large, $\Phi_s(u)$ has support in $\pa \manif_0 \times (-\infty,0)
  \subset \manif$, and hence $P \Phi_s(u)$ is defined.
  The independence of $s$ follows from the definition
  of $\iPS {m}(\manif)$ (see Definition \ref{def.translation-invariant}).
  Indeed, we have independence of $s$
  as soon as $s$ is large enough for $\Phi_s(u)$ to have support in
  $\pa \manif_0 \times (-\infty, -R_P)$. The relation $\tilde{a} = \maR_{\infty}(a)$
  follows from the definition of $\CI_{\inv}(\manif)$ (Definition \ref{def.CIinv}).
  The resulting operator $\In(P)$ is clearly invariant by translations.
  This completes the proof.
\end{proof}

In analogy with Proposition \ref{prop.kernel.descr}, we have the following proposition,
whose proof is completely similar. We do not include its proof since we will not use
this proposition, for the same reason, we state it explicitly only in the case $E = \CC$.
Recall that our manifold $\manif$ has dimension $n$. The following results,
while not used in this paper, is meant to be used for $\link = \pa \manif_0$, which thus
has dimension $n-1$. (The asymptotic of the distribution kernel is, however, on
$\link \times \RR$, thus on a manifold of dimension $n$.)

\begin{proposition}\label{prop.kernel.descr2}
  Let $\link$ be a closed manifold of dimension $n-1$ and $m < 0$, $m \notin \ZZ$.
  A continuous, linear map $T : \CIc(\link \times \RR) \to \CIc(\link \times \RR)'$ is
  a \emph{classical pseudodifferential} operator in
  $\iPSsus{m} {\link} := \iPS{m}(\link \times \RR)^\RR$ if, and only if,
  it satisfies the following conditions.
  \begin{enumerate}
    \item
    Its distribution kernel $k_T$ is a function with support in some
    $\epsilon$-neighborhood of the diagonal (Equation \ref{eq.def.eps_neigh})
    and is smooth away from the diagonal.

    \item
    For all $x, y \in \link$, $t, s, \lambda \in \RR$, $(x, t) \neq (y, s)$, we have
    \begin{equation*} 
      k_T(x, t-\lambda, y, s - \lambda) \seq k_T(x, t, y, s) \,.
    \end{equation*}

    \item
    Let $S (\link \times \RR) \subset T (\link \times \RR)$ be the subset of vectors of length one, as usual,
    and $p : S(\link \times \RR) \to \link \times \RR$ be the canonical projection.
    Then there exist functions $a_j \in \CI\big(S(\link \times \RR)\big)^\RR$ and
    $C_j \ge 0$, $j \in \ZZ_+$, such that,
    for all $N \in \NN$, $\xi \in S (\link \times \RR)$, and $t \in (0, r_\link /2]$
    \begin{equation*}
      |k_T(p(\xi), \exp(t \xi)) - \sum_{j=0}^N a_j(\xi) t^{-n -m + j}| \le
      C_N t^{-n-m + N + 1} \,.
    \end{equation*}
  \end{enumerate}
\end{proposition}

We also have the following property of the distribution kernel of $\tilde P = \In(P)$,
which is a defining property for $m < 0$ (by Peetre's theorem).

\begin{proposition}\label{prop.desc.kInP}
  Let $P \in \iPS{m}(\manif)$. Then
  \begin{equation*}
    k_{\widetilde P}(x, t, y, s)  \seq k_{P}(x, t - \lambda, y, s - \lambda)
  \end{equation*}
  for $\lambda$ large enough and $(x,t) \neq (y, s) \in \pa \manif_0 \times (-\infty, 0)$.
\end{proposition}

\begin{proof}
  This follows from the translation invariance at infinity, knowing that
  the distribution kernel of a pseudodifferential operator is a smooth
  function outside the diagonal.
\end{proof}

We now introduce an important concept without analog in the compact case.

\begin{definition}  \label{def.indicial-oper}
  Let $P\in \iPS {m}(\manif)$ and
  $$\widetilde{P} \in \iPSsus{m} {\pa \manif_0} \ede
  \Psi_{inv}^m(\pa \manif_0 \times \RR)^{\RR}$$
  be as in Lemma \ref{lemma.def.indicial}.
  Then $\widetilde{P}$ is called the \emph{limit operator} associated to $P$.
  The resulting map
    \begin{equation*} 
      \In : \iPS {m}(\manif)\ni P\longmapsto
      \widetilde{P}\in \iPSsus{m} {\pa \manif_0}
    \end{equation*}
  will be called the \emph{limit operator map.}
\end{definition}

For the proof and for further use, we shall need the following cut-off function $\eta$
that will be fixed throughout the paper:

\begin{notation}\label{not.rem.eta}
  We let $\eta : \pa \manif_0 \times \RR \to [0, 1]$ be a fixed smooth ``cut-off'' function
  \begin{equation*}  
    \eta(x, t) \seq
    \begin{cases}
    \ 1 \ \mbox{ for  }\ t \le -2\\
    \ 0 \ \mbox{ for }\ t \ge -1\,.
    \end{cases}
  \end{equation*}
  We also regard $\eta$ as a function
  on $\manif = \manif_0 \cup (\pa \manif_0 \times (-\infty, 0])$ by setting it to be $=0$ on $\manif_0$.
\end{notation}

The cut-off function $\eta$ of \ref{not.rem.eta} will be used to show that the limit-operator
map $\In$ is surjective by constructing a cross-section for it.

\begin{definition}\label{def.s0}
  For $T \in \iPSsus{m} {\pa \manif_0} := \iPS {m} ( \pa \manif_0 \times \RR )^{\RR}$, we let
  \begin{equation*} 
    s_0(T)  \ede \eta T\eta : \CIc(\manif) \to \CIc(\manif)\,.
  \end{equation*}
\end{definition}

It is known that the map $s_0$ is well defined from the properties of the $b$-calculus.
We nevertheless prove directly some of its needed properties, for completeness and for
the benefit of the reader.

\begin{lemma} \label{lemma.onto}
  Let $T \in \iPSsus{m} {\pa \manif_0}$
  and $s_0(T) := \eta T\eta$ (Definition \ref{def.s0}).
  \begin{enumerate}
    \item $s_0(T)$ is well defined.
    \item $s_0(T) \in \iPS{m}(\manif)$.
    \item We have $\In \circ s_0 = id$, where $\In$ is as in Definition \ref{def.indicial-oper}.
  \end{enumerate}
\end{lemma}

\begin{proof}
  (1)\ Let $u \in \CIc(\manif)$.
  Then $\eta u$ has support in $\pa \manif_0 \times (-\infty, 0)$,
  and hence it can be identified with a function in $\CIc(\pa \manif_0 \times \RR)$, so
  $T (\eta u) \in \CIc(\pa \manif_0 \times \RR)$ is defined. Then $\eta T (\eta u)$
  in turn has support in $\pa \manif_0 \times (-\infty, 0) \subset \manif$, and hence
  can be identified with a function in $\CIc(\manif)$.
  Thus $s_0(T)u := \eta T(\eta u) \in \CIc(\manif)$ is well defined and we have (1).

  (2) Clearly $s_0(T)$ is an order $m$ pseudodifferential operator,
  by the general properties of pseudodifferential operators.
  Let $u \in \CIc(\pa \manif_0 \times (-\infty, -R))$ for some large $R>0$.
  Then $\eta \Phi_s (u) = \Phi_s(u)$, for $s \ge 0$. The same is true for $u$
  replaced with $T \Phi_s(u)$, $s \ge 0$, since $T$ is invariant and properly supported (so it
  increases the support of $\Phi_s(u)$ by at most $\varepsilon_T$). We thus have
  \begin{equation*}
    s_0(T) \Phi_s (u) = \eta T(\eta \Phi_s(u)) = T(\Phi_s(u))
    = \Phi_s T (u) = \Phi_s \eta (T(\eta u) ) = \Phi_s s_0(T)u\,.
  \end{equation*}
  Hence $s_0(T) \in \iPS{m}(\manif)$, as claimed. To prove (3) the same reasoning gives
  \begin{equation*}
    \Phi_{-s} s_0(T) \Phi_s (u) = \Phi_{-s} \eta T(\eta \Phi_s(u)) = \Phi_{-s} T(\Phi_s(u))
    = T (u)
  \end{equation*}
  for $s$ large enough. Thus $\In (s_0(T)) = T$. 
\end{proof}

We shall need to take a closer look at the principal symbols of operators in
$\iPSsus{m} {\pa \manif_0}$. To this end, we
specialize Theorem \ref{thm.diffeo.inv2} and Proposition \ref{prop.def.ell} to these operators.

\begin{remark}\label{rem.boundary.restriction}
  Let us replace $\manif$ with $\pa \manif_0 \times \RR$ in Remark
  \ref{rem.def.inv.symb} and restrict to translation invariant operators.
  Recall that
    $\iPSsus{m} {\pa \manif_0} \ede \iPS{m} (\pa \manif_0 \times \RR)^{\RR}$.
  Then the principal symbol yields a map
  \begin{equation*} 
    \sigma_m : \iPSsus{m} {\pa \manif_0}
    \to S_{\inv}^m(T^*(\pa \manif_0 \times \RR))^{\RR}/ S_{\inv}^{m-1}(T^*(\pa \manif_0 \times \RR))^{\RR}
    \,.
  \end{equation*}
  Next, we can identify $$S^m(T^*(\pa \manif_0 \times \RR))^{\RR}
  \simeq S^m(T^*(\pa \manif_0 \times \RR)/\RR)$$ and, furthermore, we
  can also identify
  \begin{equation*}
    T^*(\pa \manif_0 \times \RR)/\RR \, \simeq \, T^*(\pa \manif_0 \times \RR)\vert_{t = 0}
    \, \simeq \,  (T^*\manif_0) \vert_{\pa \manif_0} \, \simeq \,  \pa T^*\manif_0 \,,
  \end{equation*}
  where $T^*\manif_0\vert_{\pa \manif_0} := p^{-1}(\pa \manif_0)$ is the restriction of the bundle
  $p : T^*\manif_0 \to \manif_0$ to the boundary of $\manif_0$. (All these identifications
  are obtained using that the standard decomposition of $T^*\manif$ as
  a manifold with straight cylindrical ends is
  $T^*\manif \seq T^*\manif_0 \cup \big[ \pa T^*\manif_0 \times (-\infty, 0] \big ]\,.$)
  Let $\maR_{\infty}$ be as in Definition \ref{def.CIinv}.
  By taking into account all these identifications, we obtain the map
  \begin{equation*} 
    \sigma_m^\RR \seq \maR_{\infty} \circ \sigma_m : \iPSsus{m} {\pa \manif_0}
    \to
    S_{inv}^m(\pa T^*\manif_0)/S_{inv}^{m-1}(\pa T^*\manif_0)
    \,.
  \end{equation*}
  Let $S^*\manif_0$ be the set of unit vectors in $T^*\manif_0$.
  Then, in the case of classical symbols, the last quotient identifies, as usual
  with $\CI(\pa S^*\manif_0)$.
  \end{remark}


We obtain the following proposition.

\begin{proposition}\label{prop.onto.inv.symb}
  We use the notation of Remark \ref{rem.boundary.restriction}.
  \begin{enumerate}
    \item For any $T \in \iPS{m}(\manif)$, we have
    $$\sigma_m^\RR(\In(T)) \seq \maR_{\infty}(\sigma_m(T))\,.$$
    \item  The principal symbol
    \begin{equation*}
      \sigma_m^\RR : \iPSsus{m} {\pa \manif_0} \
      \longrightarrow\ S_{inv}^m(\pa T^*\manif_0)/S_{inv}^{m-1}(\pa T^*\manif_0)
    \end{equation*}
    is surjective and $*$-multiplicative with kernel $\iPSsus{m-1}{\pa \manif_0}$.
  \end{enumerate}
\end{proposition}

\begin{proof}
  By definition, we have $\In (T) = T$ on $\pa \manif_0 \times (-\infty, -R)$,
  for some large $R$. Hence
  $$\sigma_m^\RR (\In(T)) \seq \sigma_m^\RR (T) \seq \maR_{\infty} \circ \sigma_m(T)\,,$$
  see Remark \ref{rem.boundary.restriction}, thus (1).
  The surjectivity in (2) follows from the fact that the quantization
  map of Remark \ref{rem.def.quant.map} is translation invariant and
  from Proposition \ref{prop.onto.q}, all applied to $\manif_0 \times \RR$
  instead of $\manif$.
  The other properties of $\sigma_m^{\RR}$ follow from
  Theorem \ref{thm.prop.inv-calc}(2), again applied to $\manif_0 \times \RR$
  instead of $\manif$.
\end{proof}

\subsection{Further properties of the $\inv$-calculus}
\label{ssec.further.prop}
We now include some properties of our $\inv$-calculus that are no longer direct
analogues of the classical calculus on compact manifolds, but they still have
counterparts in the $b$-calculus. The mapping properties of this
subsection are a consequence of the corresponding properties for the $b$-calculus,
but the proof presented here are new. Recall the limit operator map
$\In : \iPS {m}(\manif) \to \iPSsus{m} {\partial \manif_0}
:= \iPS{\infty }(\partial \manif_0\times {\RR})^{\RR}$ of
Definition \ref{def.indicial-oper}.

\begin{definition}\label{def.c.supp}
We let $\Psi_{comp}^m(\manif)$ denote the space of pseudodifferential operators on $\manif$
with compactly supported distribution kernel.
\end{definition}

\begin{theorem} \label{thm.prop.inv-calc.in}
  Let $\manif = \manif_0 \cup \big ( \pa \manif_0 \times (-\infty, 0] \big )$ be a
  manifold with straight cylindrical ends.
  \begin{enumerate}
    \item $\In : \iPS{m}(\manif) \to \iPSsus{m} {\pa \manif_0}$
    is surjective and its kernel consists of $\Psi_{comp}^m(\manif)$, the set of operators
    $T \in \iPS{m}(\manif)$ whose distribution (or Schwartz) kernel is compactly supported.
    \item We have a direct sum decomposition
    \begin{equation*}
      \begin{gathered}
      \iPS{m}(\manif) \seq s_0(\iPSsus{m} {\pa \manif_0}) + \Psi_{comp}^m(\manif)\\
      P \to (s_0(\In(P)), P - s_0(\In(P)))\,.
      \end{gathered}
    \end{equation*}
  \end{enumerate}
\end{theorem}

\begin{proof} The surjectivity in (1) follows from Lemma \ref{lemma.onto}(3).
  The form of its kernel follows from the definition of operators that are translation
  invariant in a neighborhood of infinity. Indeed, let $P \in \iPS{m}(\manif)$ with
  $\In(P) = 0$. Let $R_P > 0$ be as in Definition \ref{def.translation-invariant}, that is
  $P \Phi_s(u) = \Phi_s P(u)$ for all $u \in \CIc \big(\partial \manif_0\times (-\infty ,-R_P)\big)$
  and all $s>0$. By the definition of $\In(P)v = \Phi_{-s} P \Phi_s(v)$ for $s$ large,
  we obtain that $P u = 0$ if $u$ has support in
  $\partial \manif_0\times (-\infty ,-R_P) =: \manif_{R_P}^c := \manif \smallsetminus \manif_{R_P}$.
  This shows that the distribution kernel of $P$ is contained in $\manif \times \manif_{R_P}$.
  The same argument applied to $P^* \in \iPS{m}(\manif)$ shows that the distribution kernel
  of $P$ is contained in $\manif_{R_{P^*}} \times \manif$. Therefore, the distribution kernel
  of $P$ is contained in $\manif_{R_{P^*}} \times \manif_{R_{P}}$, which is a compact set.

  (2) follows from (1) and the relation $\In \circ s_0 = id$ of Lemma \ref{lemma.onto}. Indeed,
  \begin{equation*}
    \In (P - s_0(\In(P))) \seq \In (P) - \In(s_0(\In(P))) \seq \In (P) - \In (P) \seq 0\,.
  \end{equation*}
  Therefore $P - s_0(\In(P)) \in \Psi_{comp}^m(\manif)$ by (1).
\end{proof}

We now discuss a property of the $\inv$-calculus that has no analogue in the classical case of
pseudodifferential operators on a compact manifold or for the $b$-calculus. We shall denote by $t$
the cylindrical variable in $(x, t) \in \pa \manif_0 \times (-\infty, 0) \subset \manif$
and with the same letter the corresponding multiplication operator.

\begin{proposition} \label{prop.comm.rho}
  Let $t \in \RR$ be the second coordinate function on $\pa \manif_0 \times \RR$ and
  let $\rho \in \CI(\manif)$ be such that
  $\rho (x, t) = t$ for all $(x, t) \in \pa \manif_0 \times (-\infty, -1]$.
  \begin{enumerate}
    \item Let $T \in \iPSsus{m} {\pa \manif_0}$. Then
    $[t, T] \in \iPSsus{m} {\pa \manif_0}$.

    \item If $T \in \iPSsus{m} {\pa \manif_0}$,
    then $[\rho, s_0(T)] = s_0([t, T])$.

    \item Let $P \in \iPS{m}(\manif)$, we have $[\rho, P] \in \iPS{m}(\manif)$ and
    $\In ([\rho, P]) = [t, \In(P)]$.
  \end{enumerate}
\end{proposition}

\begin{proof}
  (1) is well known from the $b$-calculus. We include the argument in \cite{Mitrea-Nistor}.
  First, the operator $[t, T]$ is pseudodifferential of order $m-1$. It is also
  properly supported. It is translation invariant since
  \begin{equation*}
    \Phi_{-s} [t, T] \Phi_s \seq [\Phi_{-s} t \Phi_s, \Phi_{-s} T \Phi_s] \seq [t - s, T] \seq [t, T]\,,
  \end{equation*}
  because $s$ is a constant. This completes the proof of (1).

  The operators $t$, $\rho$, and $\eta$ are multiplication operators, and hence
  they commute. Moreover, the assumption gives that $\rho \eta = t \eta$.
  The relation (2) is thus immediate from
  \begin{equation*}
    [\rho, s_0(T)] \ede \rho \eta T \eta - \eta T \eta \rho \seq \eta [t, T] \eta \,=:\, s_0([t, T])\,.
  \end{equation*}

  To prove (3), let us write first $P = s_0(\In(P)) + P_c$, where $P_c \in \Psi_{comp}^m(\manif)$.
  Then
  \begin{equation*}
    [\rho, P] \seq [\rho, s_0(\In(P))] + [\rho, P_c] \seq s_0([t, \In(P)]) + [\rho, P_c] \in
    \iPS{m}(\manif) \,,
  \end{equation*}
  by (1) and (2) already proved and by Theorem \ref{thm.prop.inv-calc.in}(2),
  using also the fact that $[\rho, P_c] \in \Psi_{comp}^m(\manif)$.
  Finally, let $u \in \CIc(\pa \manif_0 \times \RR)$ and $s \in \RR$ be large enough. Then
  \begin{equation*}
    \In ([\rho, P]) u \seq \Phi_{-s}([\rho, P])\Phi_{s}u \seq [t - s, \In(P)]u \seq [t, \In(P)]u\,.
  \end{equation*}
  This completes the proof.
\end{proof}

We shall need also the following boundedness result. In a certain sense, it
should have been part of Theorem \ref{thm.prop.inv-calc}, but it is easier to prove now.
Also, we formulate it as a lemma, not as theorem,
since stronger results will be proved later on. Given normed spaces $X$ and $Y$, we let
$\maB(X; Y)$ denote the space of bounded, linear operators $T : X \to Y$. We let
$\maB(X) := \maB(X; X)$, as usual.

\begin{lemma}\label{lemma.bounded}
  If $P \in \iPS{m}(\manif)$ and $s \in \RR$, then the map $P : \CIc(\manif) \to \CIc(\manif)$
  extends by continuity to a map $H^{s}(\manif) \to H^{s-m}(\manif).$ Moreover, the map
  $$\iPS{m}(\manif) \to \maB(H^{s}(\manif), H^{s-m}(\manif))$$ is continuous
  for the natural topology on $\iPS{m}(\manif)$ (see Remark \ref{rem.top.op}).
\end{lemma}

\begin{proof}
  Let us write $P = P_c + s_0(P_i) \in \iPS{m}(\manif)$,
  where $P_c$ has compactly supported distribution kernel and $P_i := \In(P)$,
  as in Theorem \ref{thm.prop.inv-calc.in}(2).
  Then $P_c : H^{s}(\manif) \to H^{s-m}(\manif)$ is bounded
  by the mapping properties of pseudodifferential operators on compact manifolds.
  Let us prove next that $s_0(P_i) : H^{s}(\manif) \to H^{s-m}(\manif)$
  is bounded. Since $s_0 (P_i) = \eta P_i \eta$ and $\eta$ is bounded on all
  $H^{s}$ spaces, it is enough to prove that $P_i : H^{s}(\pa \manif_0 \times \RR)
  \to H^{s-m}(\pa \manif_0 \times \RR)$ is bounded.
  Let us consider a $\ZZ$-translation invariant partition of unity $(\phi_k)$, $k \in \ZZ$, as in
  Section \ref{sssec.alternative}. Then $\phi_k P \phi_j: H^{s}(\pa \manif_0 \times \RR)
  \to H^{s-m}(\pa \manif_0 \times \RR)$ is bounded,
  by the case of pseudodifferential on compact manifolds. Moreover, $\phi_k P \phi_j = 0$
  for $|j - k|$ large and the operators $\phi_k^2 P \phi_j^2$ and $\phi_{k+\ell}^2P \phi_{j+\ell}^2$
  are unitary equivalent for $\ell \in \NN$ and $k$ (equivalently $j$) large enough.
  It follows then by standard arguments (a lemma of Cotlar, Knapp, Stein,
  Calderon, Vaillancourt, Bony, and Lerner, Lemma 18.6.5 in \cite{Hormander3})
  that $P = \sum_{k, j \in \ZZ_+} \phi_k^2 P \phi_j^2$
  is bounded. The continuity for our algebras follows also from the compact case by the
  same arguments.
\end{proof}

We shall need to following ``order reduction operators'' (a stronger version
will be obtained later).

\begin{corollary}\label{cor.order.red}
  Let $s \in \RR$. There exists an elliptic operator $\Lambda_s \in \iPS{s}(\manif)$
  such that the induced map $\Lambda_s : H^s(\manif) \to L^2(\manif)$ is an isomorphism.
\end{corollary}

\begin{proof}
  Let $a_{s,t}(\xi) := (t^2 + \|\xi\|^2)^{s/2}$, where $t > 0$ is a parameter and
  $\xi \in T^*M$.
  We have that $a_{s,t} \in S^{s}(T^*\manif)$ and $a_{s,t}a_{-s,t} = 1$. Then
  \begin{align*}
    q(a_{s,t})q(a_{-s,t}) - 1 \to 0 \ \mbox{ for } t \to \infty
  \end{align*}
  in the natural topology on $\iPS{0}(\manif)$ given by Remark \ref{rem.top.op}
  by the standard formula for the product of pseudodifferential operators in local
  coordinates, see Theorem \ref{thm.mapping.properties}(2).
  See \cite{SchroheFrechet} for some similar results. Hence $q(a_{s,t})q(a_{-s,t})$ will
  be invertible on each $H^m(\manif)$, as long as $t$ is large enough. We use this for $m = 0$ and $m = s$.
  Let $\Lambda_s := q(a_{s,t})$. Then $\Lambda_s : H^s(\manif) \to L^2(\manif)$ is continuous,
  by Lemma \ref{lemma.bounded}. The same argument gives that $Q := q(a_{-s,t})$
  maps $L^2(\manif)$ to $H^s(\manif)$ continuously. Since $\Lambda_s Q$ is invertible
  on $L^2(\manif)$ (for $t$ large), it follows that $\Lambda_s$ is surjective. Similarly (for $t$ large),
  we may assume that $Q\Lambda_s$ is invertible on $H^s(\manif)$. If follows that $\Lambda_s$ is
  also injective. Thus we can choose $\Lambda_s := q(a_{s,t})$, for $t$ large.
\end{proof}

We shall use the function $\rho$ of Proposition \ref{prop.comm.rho},
thus $$\rho \in \CI(\manif)\,,\ \rho < 0 \mbox{ everywhere, and }\ \rho(x, t)
\seq t\ \mbox{ for }\ t \le -1\,.$$

\begin{corollary}\label{cor.weighted.con}
  Let $\rho$ be as above,
  and $m, s \in \RR$. Let $k \in \ZZ_+$ and $\rho^k H^s(\manif)$
  be the space of those functions $u$ such that $\rho^{-k}u \in H^s(\manif)$.
  On $\rho^k H^s(\manif)$ we consider the induced topology. Then every $P \in \iPS{m}(\manif)$
  defines a continuous map
  \begin{equation*}
    P : \rho^k H^s(\manif) \to \rho^k H^{s-m}(\manif)\,.
  \end{equation*}
\end{corollary}

\begin{proof}
  The result is known for $k = 0$, by Lemma \ref{lemma.bounded}. Let us
  prove it by induction on $k \in \ZZ_+$. Let us assume the result is true for some
  $k \in \ZZ_+$ and prove it for $k + 1$. We know that $[\rho, P] \in
  \iPS{m-1}(\manif)$. Then
  \begin{align*}
    \| \rho^{-k} P(\rho^k u) \|_{H^{s-m}}
    %
    & \leq \| \rho^{-k+1} P (\rho^{k-1} u) \|_{H^{s-m}}
    + \| \rho^{-1} \rho^{-k+1}[P, \rho] (\rho^{k-1} u) \|_{H^{s-m}}\\
    & \leq C \| u \|_{H^{s}}
    + C \| \rho^{-k+1}[P, \rho] (\rho^{k-1} u) \|_{H^{s-m}}\\
    & \seq C \| u \|_{H^{s-m}} \,,
  \end{align*}
  where $C> 0$ is a generic constant (different at each occurrence)
  and in the last two steps we have used the induction properties
  and the fact that $\rho^{-1} \in W^{\infty, \infty}(\manif)$.
\end{proof}

We stress that the weight used in the last lemma is not the one typically used
in the $b$-calculus (where the weights $e^{\lambda \rho}$ are used
instead).
We now spell out some simple consequences of the results in this subsection, for later use.
First, by replacing $\manif$ with $\link \times \RR$ in Lemma \ref{lemma.bounded}, we obtain

\begin{corollary} \label{cor.cont1}
  Let $P \in \iPSsus{m}{\link} := \iPS{m}(\link \times \RR)^{\RR}$.
  Then $$P: H^s(\link \times \RR) \to H^{s-m}(\link \times \RR)$$
  is well-defined and continuous for all $s \in \RR$.
  Moreover, the map $\iPSsus{m}{\link} \to \maB(H^s(\link \times \RR); H^{s-m}(\link \times \RR))$
  is continuous.
\end{corollary}

By introducing an extra factor, we see that the action on the first
component yields a continuous action on the product:

\begin{corollary} \label{cor.contProd}
  Let $M$ be a Riemannian manifold.
  Then $\iPS{0}(\manif)$ acts continuously on the product
  $L^2(\manif \times M)$.
\end{corollary}

\begin{proof}
  We know that $P$ acts continuously on $L^2(\manif)$.
  The corollary then follows from the fact that $L^2(\manif \times M) \simeq
  L^2(\manif) \overline{\otimes} L^2(M)$, where $\overline{\otimes}$
  is the Hilbert space tensor product and the action of $P$ on the
  tensor product is $P \overline{\otimes} 1$ with respect to this identification.
\end{proof}

We also have the following characterization of Sobolev spaces.

\begin{corollary} \label{cor.char.Sobolev}
  Let $u \in H^{t}(\manif)$, for some $t \in \RR$.
  We have $u \in H^s(\manif)$ if, and only if, $Pu \in L^2(\manif)$ for
  all $P \in \iPS{s}(\manif)$.
\end{corollary}

\begin{proof} If $u \in H^s(\manif)$, then $Pu \in L^2(\manif)$ for
  all $P \in \iPS{s}(\manif)$ by Lemma \ref{lemma.bounded}. Let $P$
  be as in Corollary \ref{cor.order.red} and $Q$ as in its proof.
  By the same proof, we may assume $QP$ to be an isomorphism of $H^{s}(\manif)$.
  Let us assume $P u \in L^2(\manif)$. Then $QP u \in H^s(\manif)$
  and hence $u \in  H^s(\manif)$, because we have chosen $P$ and $Q$
  such that $QP$ is an isomorphism of $H^{s}(\manif)$ and $u = (QP)^{-1}QPu$.
\end{proof}

An important further property of the $\inv$-calculus is the characterization of the
Fredholm operators in $\iPS{m}(\manif)$. Of course, this can be obtained from the
corresponding result in the (bigger) $b$-calculus. In this paper, however, we shall
obtain this result in the (also bigger) class of
``essentially translation invariant in a neighborhood of infinity operators,'' which we
discuss in the next section (see Theorem \ref{thm.ADN.Fredholm}).

\section{Essentially translation invariant operators}
\label{ess-trans-inv}

The algebra $\iPS {\infty }(\manif)$ is \emph{not stable under inversion} in the sense that
if $T\in \iPS {\infty }(\manif)$ is an invertible operator on the space $L^2(\manif)$, then its
inverse does not necessarily belong to $\iPS {\infty }(\manif)$. As in Mitrea and Nistor
\cite[Section 2.1]{Mitrea-Nistor}, we now introduce a suitable enlargement
$\ePS {\infty} (\manif)$ of $\iPS {\infty} (\manif)$ that will be stable under
the inversion of elliptic operators of non-negative order. The operators in $\ePS {\infty} (\manif)$
are called \emph{essentially translation invariant in a neighborhood of infinity}
(or \emph{essentially translation invariant at infinity} or, simply,
\emph{essentially translation invariant}; they were called  \emph{almost invariant} in
\cite{Mitrea-Nistor}).

Recall that $\manif = \manif_0 \cup (\pa \manif_0 \times (-\infty, 0])$ is a manifold
with \emph{straight, compact cylindrical ends.}

\subsection{Definition of essentially translation invariant operators}
Let us first look at the distribution kernels of the operators in $\iPS{-\infty}(\manif)$.
Recall the discussion in Remark \ref{rem.boxtimes}.
Recall that a continuous linear map $P : \CIc(\manif) \to \CIc(\manif)'$
is in $\Psi^{-\infty}(\manif)$ if, and only if, its distribution kernel $k_P$ is smooth.
In particular,

\begin{remark}\label{rem.comp.kernel}
  Recall that $\Psi_{comp}^m(\manif)$ denotes the space of pseudodifferential operators on
  $\manif$ with compactly supported distribution kernel.
  We then have $P \in \Psi_{comp}^{-\infty}(\manif)$ if, and only if, $k_P \in \CIc(\manif \times \manif)$.
\end{remark}

Recall that $\link$ denotes a closed Riemannian manifold.
Also, similarly to Proposition \ref{prop.kernel.descr}, we have the following proposition.

\begin{proposition}\label{prop.kernel.descr3}
  Let $\link$ be a closed Riemannian manifold, as usual.
  A continuous, linear operator $T : \CIc(\link \times \RR) \to \CIc(\link \times \RR)'$ is in
  $\Psi^{-\infty}(\link \times \RR)^\RR$ if, and only if, there exists
  $C_T \in \CI(\link \times \link \times \RR)$
  such that
  \begin{equation*}
    k_T(x, t, y, s) \seq C_T(x, y, t-s) \in \CC \,.
  \end{equation*}
  The kernel $C_T$ is called the \emph{convolution kernel} of $T$.
  We have $T \in \iPS{-\infty}(\link \times \RR)^\RR$ if, and only if,
  $C_T \in \CIc(\link \times \link \times \RR)$.
\end{proposition}

The main point of the last proposition is that the \emph{convolution} kernel $C_T$
of $T$ is compactly supported, unlike its \emph{distribution} kernel $k_T$.
Explicitly, we have
\begin{equation}\label{eq.conv.op}
  Tu(x, t) \seq \int_{\link \times \RR} C_T(x, y, t-s) u(y, s) \dvol_{\link}(y) ds\,.
\end{equation}
One of the main points of the upcoming definition of $\ePS{m}(\manif)$ is that
this formula (Equation \ref{eq.conv.op})
makes sense also for \emph{Schwartz functions,} whose standard definition
we recall next using a function $\rho$ as in Proposition \ref{prop.comm.rho}:

\begin{definition}\label{def.Schwartz}
  Let $0 > \rho \in \CI(\manif)$, $\rho(x, t) = t$ for $t \le -1$.

  \begin{enumerate}
    \item We let $\maS(\manif \times \manif)$ be the space of functions
    $u : \manif \times \manif \to \CC$ such that
    \begin{equation*}
      \int_{\manif^2} |\rho(x)^i \rho(y)^j \nabla_x^k  \nabla_y^l u(x, y)|^2
      \dvol_{\manif^2}(x, y) < \infty\,.
    \end{equation*}
    \item Similarly, let $\maS(\link \times \link \times \RR)$ be the space of functions
    $u : \link \times \link \times \RR \to \CC$ such that
    \begin{equation*}
      \int_{\link^2 \times \RR} |t^i \pa_t^j\nabla_x^k \nabla_y^l  u(x, y, t)|^2
      \dvol_{\link^2 \times \RR}(x, y, t) < \infty\,.
    \end{equation*}
  \end{enumerate}
  In these two formulas, $i, j, k, l \in \ZZ_+$ are arbitrary.
\end{definition}

\begin{remark}\label{rem.SchwartzSobolev}
  Of course, to obtain the Schwartz spaces of Definition \ref{def.Schwartz}, we could have
  used the $L^\infty$-norm (or any other $L^p$-norm, $1 \le p \le \infty$) instead of
  the $L^2$-norm appearing in the definition. In particular, we have
  \begin{equation*}
    \begin{gathered}
      \maS(\manif \times \manif) \seq \bigcap_{k \in \ZZ} \, \rho(x)^{-k} \rho(y)^{-k}
      W^{\infty, p} (\manif \times \manif ) \ \mbox{ and}\\
      \maS(\link \times \link \times \RR) \seq \bigcap_{k \in \ZZ} \rho(t)^{-k}\,
      W^{\infty, p} (\link \times \link \times \RR)\,.
    \end{gathered}
  \end{equation*}
\end{remark}

\begin{definition} \label{def.ess.zero}
  Let $\link$ be a closed Riemannian manifold.
  \begin{enumerate}
    \item We let
    $$\ePSsus{-\infty} {\link} \subset \Psi^{-\infty}(\link \times \RR)^{\RR}$$ be the set of
    linear operators $T \in \Psi^{-\infty}(\link \times \RR)^{\RR}$
    whose convolution kernel satisfies $C_T \in \maS(\link \times \link \times \RR)$
    (see Proposition \ref{prop.kernel.descr3} and Equation \ref{eq.conv.op}).

    \item Then we let
    $\ePSsus{m} {\link} \ede \iPSsus{m} {\link} + \ePSsus{-\infty} {\link}
    \subset \Psi^m(\link \times \RR)^{\RR}$ (see \ref{eq.def.suspended}).
  \end{enumerate}
\end{definition}

\begin{remark}\label{rem.incl.one}
  Definition \ref{def.ess.zero} gives right away that
  \begin{equation*}
    \iPSsus{-\infty} {\link} \ede \iPS{-\infty}(\link \times \RR)^\RR
    \subset \ePSsus{-\infty} {\link}\,.
  \end{equation*}
  Indeed, the first space consists of operators with \emph{compactly} supported convolution
  kernels, whereas the second one consists of operators with \emph{Schwartz} convolution kernels.
\end{remark}

Recall the cut-off function $\eta : \manif \to [0, 1]$ of \ref{not.rem.eta}
(defined by $\eta$ smooth with support in $\pa \manif_0 \times (-\infty, -1)$
and $\eta = 1$ on $\pa \manif_0 \times (-\infty, -2]$). Also, recall
that, for any $T \in \iPSsus{m} {\pa \manif_0}
:= \iPS{m}(\pa \manif_0 \times \RR)^{\RR}$, we have
$s_0(T) := \eta T \eta \in \iPS{m}(\manif)$,
see Lemma \ref{lemma.onto} and its proof. Recall the following definition from \cite{Mitrea-Nistor}.

\begin{definition} \label{def.ess.one}
  Let $\manif$ be a manifold with straight
  cylindrical ends $\manif = \manif_0 \cup (\pa \manif_0 \times (-\infty, 0])$.
  \begin{enumerate}
    \item We let $\Psi_\maS^{-\infty}(\manif)$ be the set of pseudodifferential
    operators in $\Psi^{-\infty}(\manif)$ with distribution kernel in
    $\maS(\manif \times \manif)$.

    \item Let $\ePSsus{-\infty} {\pa \manif_0}$ be as in
    Definition \ref{def.ess.zero}.
    We let $$\ePS{-\infty}(\manif) \ede s_0(\ePSsus{-\infty} {\pa \manif_0})
    + \Psi_\maS^{-\infty}(\manif)\,.$$

    \item Finally, we define the space of essentially translation invariant operators
    in a neighborhood of infinity of order $m$ as
    $$\ePS{m}(\manif)\ede \iPS{m}(\manif) + \ePS{-\infty}(\manif)\,.$$
  \end{enumerate}
\end{definition}

In particular, $\ePSsus{-\infty} {\link} =: \ePS{-\infty} (\link \times \RR)^{\RR}$.

\begin{remark}\label{rem.incl.tow}
  Recall that  $\Psi_{comp}^m(\manif)$ denotes the space of pseudodifferential operators on $\manif$
  with compactly supported distribution kernel. The last two definitions provide the following simple
  inclusions. First,
  \begin{equation*}
    \Psi_{comp}^{-\infty}(\manif) \subset \Psi_\maS^{-\infty}(\manif)\,,
  \end{equation*}
  because the space of Schwartz sections contains the space of compactly supported
  sections. Let $s_0$ be as in the last definition (it was first introduced in
  Definition \ref{def.s0}). Using also the inclusion of Remark \ref{rem.incl.one},
  we obtain
  \begin{multline*}
    \iPS{-\infty}(\manif) \seq s_{0}(\iPSsus{-\infty} {\pa \manif_0})
    + \Psi_{comp}^{-\infty}(\manif)\\
    \subset s_{0}(\ePSsus{-\infty} {\pa \manif_0})
    + \Psi_\maS^{-\infty}(\manif) \,=: \, \ePS{-\infty}(\manif)\,.
  \end{multline*}
\end{remark}

\subsection{Properties of essentially translation invariant operators}

We now proceed to prove some of the basic properties of $\ePS{m}(\manif)$
(many of which were implicit in \cite{Mitrea-Nistor}; their proof in that paper
were based on suitable norms on $\iPS{-\infty}(\manif)$). Here we propose a
more direct approach, which is, however, more computational. The main techniques
used in the proof are summarized in the following remark. Alternatively, one
can prove some of these results using the $c$-calculus of Melrose, but we leave
that for another paper.

\begin{remark} \label{rem.proof}
  Most of the proofs of the properties of $\ePS{m}(\manif)$
  are \emph{standard} in the sense that they use one of the following:
  \begin{enumerate}[(i)]
    \item the definitions,
    \item Fubini's Theorem (including derivation under the integral sign),
    \item the fact that the function $\rho^{-1}$ of Definition \ref{def.Schwartz}
    is square integrable on $\manif$, and
    \item the fact that one can use the $L^\infty$-norm instead of the $L^2$-norm
    when defining our spaces (Definition \ref{def.Schwartz} again).
  \end{enumerate}
  These standard arguments will be omited. In addition to the standard
  arguments, we use some arguments (also rather common) on the kernels of
  operators. Namely, let $P$ be a properly supported pseudodifferential
  operator and $Q$ be a regularizing operator, both on $\manif$. Then
  \begin{equation*}
    k_{PQ}(x, y) \seq P_x k_{Q}(x, y)\,,
  \end{equation*}
  where $P_x$ is ``$P$ acting on the $x$-variable,'' more precisely,
  $P_x u(x, y)$ is $(Pu_y)(x)$, with $u_y(x) = u(x, y)$. Also,
  $k_{Q^*}(x, y) = k_{Q}(y, x)^*$. In particular, if $P$ is a
  multiplication operator by the function $P(x)$, then $k_{PQ}(x, y)
  = P(x)k_Q(x, y)$ and, similarly, $k_{QP}(x, y) = k_{Q}(x, y) P(y)$.
\end{remark}

\begin{lemma}\label{lemma.comp.maS}
  The space $\Psi_\maS^{-\infty}(\manif)$ is a complex algebra stable under adjoints,
  acting continuously on $L^2(\manif)$ and with\
  $\eta \Psi_\maS^{-\infty}(\manif) \subset \Psi_\maS^{-\infty}(\manif)$.
\end{lemma}

\begin{proof}
  The proof is a simple calculation based on the standard arguments outlined in
  Remark \ref{rem.proof}. The continuous action on $L^2(\manif)$ follows from the
  Riesz Lemma applied to the kernel of the operator.
\end{proof}

We next establish some properties of
$$\ePSsus{-\infty} {\link} \simeq \maS(\link \times \link \times \RR)
\subset \Psi^{-\infty}(\link \times \RR)^{\RR}\,,$$
with the isomorphism defined by the convolution kernel, see Definition \ref{def.ess.zero}.

\begin{lemma}\label{lemma.comp.inv}
  Let $\link$ is a closed manifold, as usual, and let $\ePSsus{-\infty} {\link}$ be as above.
  \begin{enumerate}
    \item $\ePSsus{-\infty} {\link}$ is a complex algebra stable
    under adjoints acting continuously on $L^2(\link \times \RR)$.

    \item $\maS(\link \times \RR) \ePSsus{-\infty} {\link}
    \subset \Psi_\maS^{-\infty}(\link \times \RR)$

    \item $[\eta^k, \ePSsus{-\infty} {\link}]
    \subset \Psi_\maS^{-\infty}(\link \times \RR)$, $k \in \NN$.

    \item Let $P, Q \in \ePSsus{-\infty} {\link}$.
    Then $s_0(P)s_0(Q) - s_0(PQ) \in \Psi_\maS^{-\infty}(\link \times \RR)$.
  \end{enumerate}
\end{lemma}

\begin{proof}
  (1) is again a simple calculation based on standard arguments (see Remark
  \ref{rem.proof}. The continuity follows from the Riesz lemma applied to the
  convolution kernel of the operator.
  For (2), we use the explicit form of the kernel
  of the composition (see Remark \ref{rem.proof}).
  For (3), it is enough
  to prove the statement for $k = 1$ by the properties of the commutator.
  Then we notice that the kernel of the commutator $[\eta, P]$ is
  $$k_{[\eta, P]}(z, z') \seq (\eta(z)- \eta(z')) k_P(z,z')\,$$
  so it is enough to look at $z = (x, t)$ and $z' = (x', t')$
  with $t \le -1$ and $t' \ge -2$ and at its symmetric $t \ge -2$ and $t' \le -1$
  (with the points of $\manif_0$ being included in the part with $t \ge -2$).
  On these regions, we bound $|\rho(x)^i\rho(y)^j(\pa_t^k\eta(x) - \pa_t^l\eta(y))|$ and
  with $(1 + |t-t'|)^{i+j}$, which, in turn, is absorbed in $k_{P}$, which is
  a function of rapid decay in $|t-t'|$ (since it is a convolution kernel
  with kernel in $\maS$). For point (4), we first notice that (2) implies also that
  \begin{equation}\label{eq.useful.ideal}
    \Psi_\maS^{-\infty}(\link \times \RR) \ePSsus{-\infty} {\link}
    \subset \Psi_\maS^{-\infty}(\link \times \RR)\,.
  \end{equation}
  The rest follows from the other points.
\end{proof}

We now specialize to $\link = \pa \manif_0$ to obtain some properties of
$\ePS{-\infty}(\manif)$.

\begin{corollary}\label{cor.comp.alg}
  We have
  \begin{enumerate}
    \item If $P \in \ePSsus{-\infty} {\pa \manif_0}$ and $R \in \Psi_\maS^{-\infty}(\manif)$,
    then $$s_0(P)R\,,\  Rs_0(P) \in \Psi_\maS^{-\infty}(\manif)\,.$$

    \item $\ePS{-\infty}(\manif) := s_0(\ePS{-\infty}(\pa \manif_0
    \times \RR)^{\RR}) + \Psi_\maS^{-\infty}(\manif)$ is a an algebra stable under adjoints
    (a $*$-algebra).
  \end{enumerate}
\end{corollary}

\begin{proof}
  (1) is yet again a simple calculation based on standard arguments and on the
  explicit form of the distribution kernels, see Remark \ref{rem.proof}.
  (It also follows from Lemma \ref{lemma.comp.inv}(2) or
  from Equation \eqref{eq.useful.ideal}.)
  For (2), we use the definition of $\ePS{-\infty}(\manif)$, Lemma \ref{lemma.comp.inv}(4)
  and the point (1) just proved.
\end{proof}

We need one more lemma before stating the main properties of the
essentially invariant calculus $\ePS{m}(\manif)$ of Definition \ref{def.ess.one}.

\begin{lemma}\label{lemma.comp.alg}
  Let $m \in \RR$. We have
  \begin{enumerate}
    \item If $R \in \ePSsus{-\infty} {\pa \manif_0}$ and
    $P\in \iPSsus{m} {\pa \manif_0} \simeq \iPS{m}(\pa \manif_0 \times \RR)^{\RR}$, then
    $$PR\,,\ RP \in \ePSsus{-\infty} {\pa \manif_0}\,.$$

    \item If $R \in \Psi_\maS^{-\infty}(\manif)$ and
    $P \in \iPS{m}(\manif)$,
    then $$PR\,,\ RP \in \Psi_\maS^{-\infty}(\manif)\,.$$

    \item If $R \in \ePS{-\infty}(\manif)$ and $P \in \iPS{m}(\manif)$,
    then $$PR\,,\ RP \in \ePS{-\infty}(\manif) \,.$$

    \item Any operator in $\ePS{-\infty}(\manif)$ maps $H^{s}(\manif)$ to
    $H^{s'}(\manif)$ continuously for all $s, s' \in \RR$.
  \end{enumerate}
\end{lemma}

\begin{proof}
  (1) and (2) are similar and are again simple calculations based on standard arguments
  (see Remark \ref{rem.proof}), but using also the explicit form of the convolution
  kernels and Corollaries \ref{cor.weighted.con} and \ref{cor.contProd}
  and Remark \ref{rem.SchwartzSobolev}. The point (3) is an immediate consequence of
  the definition of $\ePS{-\infty}(\manif)$, of Lemma \ref{lemma.comp.inv}(4),
  and of points (1) and (2) just proved.

  For $s = s' = 0$, the statement (4) follows from Definition \ref{def.ess.one}
  and the continuous actions in \ref{lemma.comp.maS}(1) and \ref{lemma.comp.inv}(1).
  For the other values of $s$, we combine the case $s=s'= 0$ with the statement
  (3) and the fact that $T \in \ePS{-\infty}(\manif)$ maps $H^{s}(\manif)$ to
  $H^{s'}(\manif)$ continuously for some $s, s' \in \RR$ if, and only if,
  $Q_{s'} T Q_{-s}$ is continuous on $L^2(\manif)$ for all $Q_{-s} \in \iPS{-s}(\manif)$
  and $Q_{s'} \in \iPS{s'}(\manif)$, by Corollary \ref{cor.order.red}.
\end{proof}

We are ready now to prove the main properties of the calculus
\begin{equation*}
  \ePS{m}(\manif) \ede \iPS{m}(\manif)) + \ePS{-\infty}(\manif)
\end{equation*}
of Definition \ref{def.ess.one}.

\begin{theorem}\label{thm.prop.ePS}
  Let $\manif = \manif_0 \cup \pa \manif_0 \times (-\infty, 0]$ be a manifold with
  straight cylindrical ends, as before.
  Let $m, m', s \in \RR$.
  \begin{enumerate}
    \item $\ePS{m}(\manif)\ePS{m'}(\manif) \subset \ePS{m+m'}(\manif)$
    and $\iPS{m}(\manif)^* \subset \iPS{m}(\manif)$ and the principal symbol is
    $*$-multiplicative:
    \begin{equation*}
      \sigma_{m+m'}(PQ) \seq \sigma_{m}(P)\sigma_{m'}(Q)
      \ \mbox{ and }\
      \sigma_{m}(P^*) \seq \sigma_{m}(P)^*\,.
    \end{equation*}

    \item If $P \in \ePS{m}(\manif)$, the $P :H^{s}(\manif) \to
    H^{s-m}(\manif)$ is well defined and continuous and the induced map
    $\ePS{m}(\manif) \to \maB(H^{s}(\manif); H^{s-m}(\manif))$
    is continuous.

    \item $\ePS{-\infty}(\manif) \subset \Psi^{-\infty}(\manif)$ and hence,
    \begin{equation*}
      \sigma_m: \ePS {m}(\manif)/\ePS {m-1}(\manif)\to
      S_{inv}^m(T^*\manif)/S_{inv}^{m-1}(T^*\manif)
    \end{equation*}
    induces an isomorphism.

    \item \label{item.parametrix}
    An operator $P \in \ePS{m}(\manif)$ is elliptic if, and only if, there exists
    $Q \in \iPS{-m}(\manif)$ such that $PQ - 1, QP - 1 \in \ePS{-\infty}(\manif)$.
  \end{enumerate}
\end{theorem}

\begin{proof}
  The first point follows from Theorem \ref{thm.prop.inv-calc}(1),
  Corollary \ref{cor.comp.alg}(2), and Lemma \ref{lemma.comp.alg}(3). The point (2)
  follows from Lemma \ref{lemma.bounded} and Lemma \ref{lemma.comp.alg}(4).
  The point (3) follows from the fact that $\sigma_m(\ePS {-\infty }(\manif))=0$.
  Finally, let $\sigma_m(P)$ be invertible. Then we can choose a parametrix
  $b \in S_{\inv}^{-m}(T^*\manif)$, by Corollary \ref{cor.lemma.ue}. The result follows
  first by using the surjectivity of $\sigma_{-m}$ in Theorem \ref{thm.prop.inv-calc}
  to choose $Q \in \iPS{-m}(\manif)$ with $\sigma_{-m}(Q) = b$.
  This gives $R_1 := 1 - PQ, R_2 := 1 - QP \in \iPS{-1}(\manif)$. Then (4) follows
  from the asymptotic completeness, as in the classical case, applied to the
  Neumann series of $1 - \sigma_{-1}R_j$, $j = 1, 2$ (that is, using Remark
  \ref{rem.asympt.compl}).
\end{proof}

\begin{remark}\label{rem.not.subset}
The algebra $\ePS{\infty}(\manif)$ is \emph{not} a subset of the {\em $b$-calculus} introduced by
Melrose \cite{MelroseActa, MelroseAPS} and  Schulze \cite{SchulzeBook91, Schulze}, contrary to what
was stated in \cite{Mitrea-Nistor}. This is because the decay at infinity for the kernels of the
regularizing operators $R \in \ePS{-\infty}(\manif)$ with zero limit operator (i.e. $\In (R) = 0$)
is only faster than any polynomial, whereas the decay is exponential for the
analogous operators in the $b$-calculus. The algebra $\ePS{\infty}(\manif)$ is, rather,
a subset of the $c$-calculus of Melrose, as seen from Lemma 2.6 of \cite{Mitrea-Nistor}.
\end{remark}

\subsection{Translation invariance and the Fourier transform}
The translation invariance of the limit operators allows us
to  consider the one-dimensional Fourier transform as in the case of the
$b$-calculus. In this section, we recall this construction.
We do not include proofs, since they are the same as those for the $b$-calculus.
Let $\link$ be a smooth, compact, boundary less manifold (i.e. a closed manifold),
as before. We denote as usual
\begin{equation} \label{Fourier}
    \maF :L^2\big(\link \times \RR; E\big)\to L^2\big(\link \times \RR; E\big),\quad
    \maF(f)(y,\tau) \ede  \int _{\RR} e^{-\imath\tau x} f(y,x)dx\,,
\end{equation}
with $\imath ^2=-1$.

\begin{remark}\label{rem.Fourier}
Let $Q \in \ePSsus{m} {\link} = \ePS {m}(\link \times \RR)^{\RR}$. Since $Q$ is
invariant with respect to the action of $\RR$ by translations, the transformed operator
$\maF Q \maF^{-1}$ commutes with the multiplication operators in $\tau \in \RR$.
Therefore, $\maF Q \maF^{-1}$ is a {\it decomposable operator} in the sense that there
exist operators
\begin{equation*} \label{operator-hat-P}
  \widehat{Q}(\tau) : \CIc (\link) \to L^2(\link)
\end{equation*}
such that
\[
  \maF \widetilde{Q} \maF^{-1}(f)(x,\tau) \seq \big[ \widehat{Q}(\tau) f_\tau \big ](x)\,,
\]
for all $f \in \CIc(\link \times \RR)$ and where $f_\tau(x) := f(x, \tau)$.
Working in local coordinates, we see that $$\widehat{Q}(\tau) \in \Psi^m(\link)\,,$$
in particular, $\widehat{Q}(\tau)$ are classical pseudodifferential
operators \cite{MelroseAPS, SchulzeBook91}.
\end{remark}

This construction is useful for checking the invertibility \cite{MazzeoEdge,
MazzeoMelroseAsian, Schulze}.

\begin{remark}\label{rem.invert}
  An operator $Q:H^s(\link \times \RR)\to H^{s-m}(\link \times \RR)$ with
  $Q \in \ePSsus{m} {\link}$
  is invertible if, and only if, the operators $\widehat{Q}(\tau):
  H^s(\link)\to H^{s-m}(\link)$ are invertible for all $\tau \in {\RR}$.
\end{remark}

We let $t$ denote the second coordinate function on
the half-infinite cylinder $\partial \manif_0 \times (-\infty, 0]$ and we extend it to
a smooth function (which we also denote by $t$) on $\manif = \manif_0 \cup (\partial \manif_0 \times (-\infty, 0])$
such that $t \ge 0$ on $\manif_0$. Then, $\partial \manif_0 \times (-\infty, 0) = \{t < 0\}$.

\begin{lemma}\label{lemma.def.indicial2}
  If $P\in \ePS {m}(\manif)$ and $u \in \CIc(\pa \manif_0 \times \RR)$, then, for $s$
  large enough, $P\Phi_s(u)$ is defined and we have the following norm convergence
  in $L^2(\pa \manif_0 \times \RR)$,
  \begin{equation*} 
    \lim_{s \to \infty} \Phi _{-s}\big [ P\Phi_s(u)\vert_{\{t < 0\}} \big]
    \ = :\ \widetilde{P}(u) \in L^2(\pa \manif_0 \times \RR)\,,
  \end{equation*}
  and the resulting operator satisfies
  $\In (P) := \widetilde{P} \in \ePSsus{m} {\pa \manif_0}$.
\end{lemma}

\begin{proof}
  As in Lemma \ref{lemma.def.indicial}, for $s$ large, $\Phi_s(u)$ has support in
  $\pa \manif_0 \times (-\infty,-2) \subset \manif$, and hence $\eta \Phi_s(u) = \Phi_s(u)$
  and $P \Phi_s(u)$ is defined
  for the same reasons as   in the proof of Lemma \ref{lemma.def.indicial}. We may assume
  $E = F$. The result is known for   $P\in \iPS {m}(\manif)$ from the same lemma. We can thus assume
  that $P \in \ePS {-\infty}(\manif)$. Let us assume first that $P \in \Psi_\maS^{-\infty}(\manif)$.
  Then $\lim_{s \to \infty} P\Phi_s(u) = 0$ in $L^2(\manif)$ from the form
  of the distribution kernel of $P$. Let us assume next that $P = s_0(P_0) = \eta P_0 \eta$
  with $P_0 \in \ePS {-\infty}(\pa \manif_0 \times \RR)^\RR$. Then
  $\Phi_{-s} P_0 \eta \Phi_s(u) = P_0 u$ is independent of $s$ for $s$ large.
  We have (still for $s$ large)
  \begin{equation*}
    \Phi_{-s} P_0 \eta \Phi_s(u) - \Phi_{-s} \eta P_0 \eta \Phi_s(u)
    \seq \Phi_{-s} (1 - \eta) \Phi_s P_0 u \to 0 \ \mbox{ for }
    \ s \to \infty\,,
  \end{equation*}
  again from the form of the convolution kernel. This completes the proof.
\end{proof}

Using the Fourier transform for $Q = \In(P)$ we obtain Melrose's indicial operators.

\begin{definition}\label{def.indicial}
If $P \in \ePS{m}(\manif)$, then the family
$$\{\widehat{P}(\tau)\}_{\tau \in \RR} \ede \{\widehat{\In(P)}(\tau)\}_{\tau \in \RR}
\ede \{\maF \In(P) \maF^{-1}(\tau)\}_{\tau \in \RR}$$
is called the {\it indicial family} of $P$. The individual operators
$\widehat{P}(\tau)$ are called \emph{indicial operators} of $P$.
\end{definition}

As with the limit operators, the indicial operators turn out to be multiplicative.

\begin{theorem}\label{thm.complet.indicial}
  Let $\manif = \manif_0 \cup (\pa \manif_0 \times (-\infty, 0])$ be a manifold with 
  cylindrical ends. Let $m , m' \in \RR$.
  \begin{enumerate}
    \item Let $P \in \ePS m (\manif)$ and $Q \in \ePS {m'} (\manif)$. Then
    $$\In (PQ) \seq \In (P)\In(Q) \ \mbox{ and } \ \In(P^*) = \In(P)^*\,.$$

    \item For $P$ and $Q$ as in $(1)$, their indicial familiess (Definition 
    \ref{def.indicial}) satisfy
    \begin{equation*}
      \widehat{QP}(\tau) \seq \widehat{Q}(\tau) \widehat{P}(\tau)\,, \quad \tau \in \RR\,.
    \end{equation*}

    \item  We have $\In \circ s_0 = id$ on $\ePSsus{m} {\pa \manif_0}$,
    and hence the map $\In : \ePS{m}(\manif) \to \ePSsus{m} {\pa \manif_0}$ is surjective
    with kernel $\Psi_\maS^{-\infty}(\manif) + \Psi_{comp}^{m}(\manif)$.

    \item For any $T \in \ePS{m}(\manif)$, we have
    $$\sigma_m^\RR(\In(T)) \seq \maR_{\infty}(\sigma_m(T))\,.$$
  \end{enumerate}
\end{theorem}

\begin{proof}
  The points (1) and (2) and the relation $\In \circ s_0 = id$ 
  follow directly from the definitions.
  The relation $\In \circ s_0 = 1$ yields the surjectivity of $\In$. The relation (4)
  is true if $T \in \ePS{-\infty}(\manif)$ (both maps vanish on this space). It is also 
  true for $T \in \iPS{m}(\manif)$, by Proposition \ref{prop.onto.inv.symb}(1). The 
  result then follows from $\ePS{m}(\manif) := \iPS{m}(\manif) + \ePS{-\infty}(\manif).$ 
\end{proof}

\subsection{Including vector bundles}

As in the previous section, it is easy to include vector bundles.

\begin{remark} \label{rem.vector.bundles}
  Let $E, F \to \manif$ be
  vector bundles with embedding in a trivial bundle $\CC^N \to \manif$
  and which correspond to projections with respect to self-adjoint
  matrix projections $e, f \in M_N(\CI_{\inv}(\manif))$. Then
  \begin{equation*}
    \begin{gathered}
      H^s(\manif; E) \ede e (H^s(\manif))^N\\
      \iPS{m}(\manif; E, F) \ede e M_N(\iPS{m}(\manif)) f \ \mbox{ and }\\
      \ePS{m}(\manif; E, F) \ede e M_N(\ePS{m}(\manif)) f
    \end{gathered}
  \end{equation*}
  and similarly for other definitions.
\end{remark}

The results above remain valid for operators acting on vector bundles, after
suitable modifications, some of which are explained in the following remarks.
For simplicity, we shall assume for the rest of this section $E = F$. The general result is obtained
by replacing $E$ with $E \oplus F$ and taking a corner of the resulting algebra.

\begin{remark}\label{rem.prop.kernel.descr}
  Proposition \ref{prop.kernel.descr}
  remains true and yields an operator in $\iPS{m}(\manif; E)$, provided that one
  modifies the statement as follows (see Remark \ref{rem.boxtimes}): $a_j \in \CI_{inv}(S\manif; \End(E))$, we take
  $k_P\in \CI( \manif \times \manif \smallsetminus \delta_M; E \boxtimes E')$,
  where $p_1, p_2 : \manif \times \manif \to \manif$ are the two projections,
  and we include the parallel transport $\exp^E$ on $E$ in the estimate \eqref{item.cond.three}
  by replacing $k_P(p(\xi), \exp(t \xi))$ by $k_P(p(\xi), \exp(t \xi)) \circ \exp^E(-t \xi)$.
  If $-m \in \NN$, one has to also include terms of the form $b_j(\xi) \ln t$, where $b_j$
  is the restriction of a polynomial function of degree $-n-m+j$ translation invariant
  at infinity, see \cite{Taylor2}[Proposition 2.8].
\end{remark}

Also, we notice that, for $R$ large enough, we can extend the maps $\Phi_s$
to sections of the vector bundle $E$
and of other bundles that are compatible with the straight cylindrical ends
structure on $\manif$ (for instance, the tangent, cotangent, and tensor bundles).

\section{\ADN-elliptic operators}
\label{sec.six}

In view of our planned applications to the Stokes operator, we will work
with \ADN-elliptic operators, which are recalled in the first subsection
of this section.
The second subsection contains some first applications of the results of
the previous subsections to regularity and self-adjointness of our operators.

\subsection{Definition of Douglis-Nirenberg ellipticity}
We recall now the definition of \ADN-elliptic operators in a
more general setting than that of manifolds with cylindrical ends.

\begin{definition}
  \label{def.ADN}
  Let $M$ be a smooth manifold, $E_i \to M$, $F_j \to M$
  be vector bundles, $E:=E_1\oplus \cdots \oplus E_k$, $F:=F_1\oplus \cdots \oplus F_k$,
  and $s_i, t_j\in \RR$, $i, j = 1,\ldots, k$. Some of the bundles $E_i$ or $F_j$
  may be zero. We let
  \begin{equation*}
    \Psi^{[\mathbf{s+t}]}(M; E, F) \ede
    \{ T=[T_{ij}] \mid\, T_{ij} \in \Psi^{s_i + t_j}(M; E_j, F_i),\ i, j = 1, \ldots, k \}\,.
  \end{equation*}
  An operator $T = [T_{ij}] \in \Psi^{[\mathbf{s+t}]}(M; E, F)$ is said to be of
  \emph{\ADN-order $\le [\mathbf{s+t}]$}.
\end{definition}

Many of the usual definitions extend to the \ADN-framework in a straightforward way.

\begin{definition}\label{def.ADN2}
  We use the notation of Definition \ref{def.ADN}. Let $T = [T_{ij}] \in \Psi^{[\mathbf{s+t}]}(M; E, F)$.
  \begin{enumerate}
    \item The \emph{principal symbol} of $T$ is the matrix
    $$\Symb(T) \ede [\sigma _{s_i+t_j}(T_{ij})]\,.$$

    \item The operator $T$ is said to be {\it \ADN-elliptic} if its principal symbol matrix
    $\Symb(T)$ is invertible on $S^*M$.

    \item If $M$ is Riemannian and $E_i$ are Hermitian, we let
    \begin{equation*}
      H^{m+[\mathbf{t}]}(M; E) \ede \displaystyle \bigoplus_{j=1}^k
      H^{m+t_j}(M; E_j) \,.
    \end{equation*}
    Analogously, we let $H^{m-[\mathbf{s}]}(M; F) \ede \displaystyle \oplus_{i=1}^k
    H^{m-s_i}(M; F_i)$.
  \end{enumerate}
\end{definition}

The most prominent example of an \ADN-elliptic operator is the Stokes operator.
See also \cite[p. 334]{Wl-Ro-La} in the case of compact Riemannian manifolds.

\begin{definition}\label{def.PsiADN}
  Let us assume in the above definition that $M = \manif$ is a manifold with
  straight cylindrical ends and that $E_i, F_j \to \manif$ are compatible (with the
  straight cylindrical ends structure) Hermitian vector bundles. We let
  \begin{multline*}
    \ePS{[\mathbf{s+t}]}(\manif; E) \ede \ePS{\infty}(\manif; E, F) \cap \Psi^{[\mathbf{s+t}]}(\manif; E, F)\\
    \seq \{ T=[T_{ij}] \mid\, T_{ij} \in \ePS{s_i + t_j}(\manif; E_j, F_i),\ i, j = 1, \ldots, k \}\,.
  \end{multline*}
\end{definition}

We record the following simple facts for further use.

\begin{remark}\label{rem.bounded.ADN}
  We keep the notation of Definiton \ref{def.PsiADN}.
  \begin{enumerate}
    \item The mapping properties of operators in $\ePS{\infty}$ (Theorem \ref{thm.prop.ePS}(3))
    show that any $T = [T_{ij}] \in \ePS{[\mathbf{s+t}]}(\manif; E, F)$ induces a continuous map
    \begin{multline*}
      T \seq [T_{ij}] : H^{m+[\mathbf{t}]}(\manif; E) \ede \displaystyle \bigoplus_{j=1}^k
      H^{m+t_j}(\manif; E_j) \\ \to \displaystyle \bigoplus_{i=1}^k
      H^{m-s_i}(\manif; F_i) \,=:\, H^{m-[\mathbf{s}]}(\manif; F)\,.
    \end{multline*}

    \item Let us assume $P = [P_{ij}]$ with $P_{ij} \in \ePS{s_i + t_j}(\manif; E_j, F_i)$ and
    $Q = [Q_{\ell i}]$ with $Q_{\ell i} \in \ePS{r_\ell - s_i}(\manif; F_i, G_\ell)$,
    then $QP \in \ePS{[\textbf{r + t}]}(\manif; E, G)$, that is $QP = [R_{\ell j}]$ with
    $R_{\ell j} \in \ePS{r_\ell + t_j}(\manif; E_j, G_\ell)$

    \item If $P$ and $Q$ are as in (2), then
    the principal symbols map $\Symb$ is $*$-multiplicative (in the usual sense that
    $\Symb(QP) = \Symb(Q)\Symb(P)$ and $\Symb(P^*) = \Symb(P)^*$).

    \item The principal symbol $\Symb(P) = [\sigma_{s_i + t_j}(P_{ij})]$ has components that are
    homogeneous of different degrees: $\sigma(P_{ij})(\lambda\xi) = \lambda^{s_i + t_j}\sigma(P_{ij})(\xi)$.
    Let $\lambda^{[\mathbf{s}]}$ be the diagonal matrix with entries $\lambda^{s_i}$ and
    $\lambda^{[\mathbf{t}]}$ be the diagonal matrix with entries $\lambda^{t_j}$. Then, for $\lambda > 0$,
    we have
    \begin{equation*}
      \Symb(P)(\lambda \xi) \seq [\sigma(P_{ij})(\lambda \xi)]
      \seq \lambda^{[\mathbf{s}]} [\sigma(P_{ij})(\xi)] \lambda^{[\mathbf{t}]}\,.
    \end{equation*}
    It follows that $\Symb(P)(\lambda\xi)$ is invertible if, and only if,
    $\Symb(P)(\xi)$ is invertible (for $\lambda > 0$). Moreover, if $P$ is \ADN-elliptic
    (i.e. $\Symb(P)(\xi)$ is invertible for all $\|\xi\|=1$, Definition \ref{def.ADN}),
    then $\Symb(P)^{-1}$ is the symbol of an \ADN-elliptic operator of order $[\textbf{-t - s}]$.

    \item Similarly, the normal (or limit) operator map $\In$ yields a $*$-multiplicative map
    \begin{equation*}
      \In : \ePS{[\mathbf{s+t}]}(\manif; E, F) \to \ePSsus{[\mathbf{s+t}]} {\pa \manif_0; E, F}
      \ede \ePS{[\mathbf{s+t}]}(\pa \manif_0 \times \RR; E, F)^\RR
    \end{equation*}
    that is compatible with the involutions (i.e. adjoints).

    \item Moreover, the Fourier transform yields $*$-multiplicative maps
    \begin{equation*}
      \ePS{[\mathbf{s+t}]}(\manif; E, F) \ni [T_{ij}] \to [\widehat T_{ij}(\tau)]
      \in \Psi^{[\mathbf{s+t}]}(\pa \manif_0; E, F)\,,
    \end{equation*}
    by Theorem \ref{thm.complet.indicial} (see also Definition \ref{def.indicial}).
  \end{enumerate}
\end{remark}

This remark will be useful when proving some of our main general results in
Section \ref{sec.SPInvFredholm}. As a rule, we will try to formulate our subsequent
results in the framework of \ADN-elliptic operators.

\subsection{Consequences of regularity}
We now formulate some regularity results and some of their consequences. These results
are very well known in the classical case of closed (i.e. boundaryless, compact) manifolds.
In this section, we formulate them for manifolds with straight cylindrical ends
and essentially translation invariant operators.
These results, while classical in the spirit, are valid also in more general settings
(see, for instance, \cite{Grosse-Kohr-Nistor-23}). However, a novel feature here is that we formulate
our results for \ADN-elliptic operators. We thus keep the notation of the previous subsection,
in particular, that of the Definition \ref{def.ADN}, so $s_i, t_j \in \RR$, $i, j = 1, \ldots, k$,
$E := E_1 \oplus E_2 \oplus \ldots E_k$, and $F := F_1 \oplus E_2 \oplus \ldots F_k$ are given
and fixed throughout this subsection.

\begin{proposition} \label{prop.regularity}
  Let $m, m' \in \RR$ and $P \in \ePS{[\mathbf{s+t}]}(\manif; E, F)$ be \ADN-elliptic.
  Let $u \in H^{m'-[\mathbf{s}]}(\manif; E)$ be such that $P u \in H^{m-[\mathbf{s}]}(\manif; F)$.
  Then $u \in H^{m+[\mathbf{t}]}(\manif; E)$.
\end{proposition}

\begin{proof}
  The ellipticity hypothesis ensures the existence of
  $Q \in \ePS{[-t-s]}(\manif; F, E)$ such that $\Symb(Q) = \Symb(P)^{-1}$ on
  $T^*M \smallsetminus 0$. Since $\Symb$ is multiplicative, we have
  $R := 1 - Q P \in \ePS{-1}(\manif; E)$. Let $N \ge \max \{m + s_i + t_j - m', 1\}$.
  By replacing $Q$ with $(1 + R + R^2 + \ldots + R^{N-1})Q$,
  we may assume $R \in \ePS{-N}(\manif; E)$. Then the hypothesis $u \in H^{m'-[\mathbf{s}]}(\manif; E)$ implies
  $R u \in H^{m' + N - [\mathbf{s}]}(\manif; E) \subset H^{m+[\mathbf{t}]}(\manif; E)$, by Theorem \ref{thm.prop.ePS}(2)
  applied to $R$ (see also Remark \ref{rem.bounded.ADN}(1)). The same theorem applied to
  $Q$ (which has \ADN-order $[\textbf{-t-s}]$) gives $QP u \in H^{m+[\mathbf{t}]}(\manif; E)$. We obtain
  \begin{equation*}
    u \seq QP u + Ru \in H^{m+[\mathbf{t}]}(\manif; E)\,.
  \end{equation*}
  This completes the proof.
\end{proof}

The regularity property (Proposition \ref{prop.regularity}) also gives the following result.

\begin{corollary} \label{cor.closed}
  Let us assume that $s_i, t_j \ge 0$, $i,j=1, \ldots, k$,
  and that $P \in \ePS{[\mathbf{s+t}]}(\manif; E, F)$ is \ADN-elliptic. We regard $P$
  as an \emph{unbounded} operator on $H^{m-[\mathbf{s}]}(\manif; F)$ with domain
  $H^{m+[\mathbf{t}]}(\manif; E)$. Then $P$ is closed.
\end{corollary}

\begin{proof}
  Let $\xi_n \in H^{m+[\mathbf{t}]}(\manif; E)$ and $\xi, \zeta \in H^{m-[\mathbf{s}]}(\manif; E)$ be such that
  $\xi_n \to \xi$ and $P \xi_n \to \zeta$, both convergences being in $H^{m-[\mathbf{s}]}(\manif; E)$.
  Then, since $P$ is continuous in distribution sense, we have $P \xi_n \to P \xi$ in distribution
  sense, and hence $P\xi = \zeta$ in distribution sense. Proposition \ref{prop.regularity}
  for $u = \xi$ and $m' = m$ gives that
  $\xi \in H^{m+[\mathbf{t}]}(\manif; E)$. Hence the graph of $P$ is closed.
\end{proof}

Regularity is also used in the proof of the next result, which states that if
$T\in \ePS{m}(\manif;E)$, $m>0$, is an elliptic symmetric operator on $L^2(\manif;E)$,
then it is essentially self-adjoint, that is, its closure is a self-adjoint operator on $L^2(\manif;E)$.
The proof is standard, see \cite[Corollary 2.2]{Mitrea-Nistor} or
\cite[Lemma 5.1]{Grosse-Kohr-Nistor-23} for instance (it is crucial here that a
manifold with straight cylindrical ends is complete). Here we formulate this result
in slightly greater generality.

\begin{lemma} \label{elliptic-self-adj}
  Let $\manif$ be a manifold with straight cylindrical ends and $E_i \to \manif$ be
  compatible Hermitian vector bundles, $i = 1, \ldots, k$. Assume all $t_j = 0$ and $s_i \ge 0$.
  Let $T = [T_{ij}]\in \ePS{[\textbf{s+0}]}(\manif;E)$, $m \ge 0$, be an \ADN-elliptic operator.
  If $T:\CIc (\manif;E)\to L^2(\manif;E)$ is
  symmetric $($i.e. $(Tf,g)=(f,Tg)$, for all $f,g\in \CIc (\manif;E)${$)$}, then the closure
  of $T$ is self-adjoint on $L^2(\manif; E)$.
\end{lemma}

\section{Restriction, Fredholm, and spectral invariance properties}
\label{sec.SPInvFredholm}

The following three theorems will be crucial for us. They are general results of independent
interest. They are extensions to our setting (cylindrical ends, \ADN-elliptic operators) of 
some very well known and useful results from pseudodifferential operators on closed manifolds.
We shall use the notation of Definition \ref{def.ADN}, in particular, $s_i \in \RR$,
$i = 1, \ldots, k$, and $E = E_1 \oplus \ldots \oplus E_k$ and $F = F_1 \oplus \ldots \oplus E_F$
are compatible hermitian vector bundles on~$\manif$.

Recall that $\manif$ is a \emph{manifold
with compact, straight cylindrical ends.} For symplicity, from now on we consider only
\emph{classical} pseudodifferential
operators.

\subsection{Restriction to submanifolds and layer potentials} 
Let us assume that $N \subset \manif = \manif_0 \cup (\pa \manif_0 \times (0, \infty])$
is a smooth submanifold such that there exists $R > 0$ and a submanifold
$N_\infty \subset \pa \manif_0$ such that
\begin{equation}\label{eq.def.N}
  N \cap (\pa \manif_0 \times (0, R]) \seq N_\infty \times (-\infty, R]\,.
\end{equation}
Let us assume that $N$ has codimension $q$ (and hence its dimension is $n-q$).
If $P$ is a pseudodifferential operator on $N$, we shall let
\begin{equation}\label{eq.def.restriction}
  k_{P}\vert\vert_N \ede k_{P}\vert_{N \times N}\,.
\end{equation}

The following theorem will be crucial in our applications to the method of layer
potentials.

\begin{theorem} \label{thm.kernel.restriction}
  Let $N \subset \manif$ be as in Equation \eqref{eq.def.N}. We assume that
  $s_i + t_j < -q$ for all $i$ and $j$, where, we recall $q$ is the
  codimension of $N$ in $\manif$. Let $P \in \ePS{[\mathbf{s+t}]}(\manif; E, F)$ be a \emph{classical}
  pseudodifferential operator. Then the
  restriction $k_{P}\vert\vert_{N}$ of $k_P$ to $N \times N \subset \manif \times \manif$
  defines a distribution on $N \times N$, which, in turn, defines a classical
  pseudodifferential operator $P\vert_{N} \in \ePS{[\mathbf{s+t}+q]}(\manif; E, F).$
  The limit operators satisfy $$\widetilde{P\vert_{N}} = \widetilde{P}\vert_{N_\infty \times \RR}\,.$$
\end{theorem}

The resulting operator $P\vert_N$ will be called \emph{the restriction of $P$ to $N$.}

\begin{proof}
  First of all, the result is true trivally if $P \in \ePS{[\mathbf{s+t}]}(\manif; E, F)$
  by Proposition \ref{prop.kernel.descr}, since the asymptotic is the same and
  $t^{-n-m+j} = t^{-(n-q) - (m+q) +j}$. Then, if $P \in \Psi_{\maS}^{-\infty}(M; E, F)$,
  the result follows again trivially  from the definition. Similarly, the result is true
  if $\manif = \pa \manif_0 \times \RR$,
  $N = N_{\infty} \times \RR$, and $P \in \ePSsus{-\infty} {\pa \manif_0}$ again by
  the definitions (see, for instance, Definition \ref{def.ess.zero}). For $P \in
  \ePS{m}(\manif; E, F)$, the result follows by linearity from the
  the definition of $\ePS{m}(\manif; E, F)$. Finally, the limit operator
  is also obtained by restriction, so the result follows by the fact that restrictions
  commute.
\end{proof}


This theorem together with the spectral invariance result of the following section
will be used to define our pseudodifferential operators appearing
in the method of layer potentials which, we recall,
are defined on the boundary $\pa \Omega$ of our domain with straight cylindrical ends
on which our boundary value problem is posed. This works directly for the single layer
potential operator $S$. These results will be included in Subsection \ref{ssec.layer.pot}.

Before applying the restriction theorem, we will need however to show that our Partial Differential
Operator $P$ is invertible on the whole manifold (without boundary). This will be done
using Fredholm conditions and the ``Mitrea-Taylor trick'' \cite{KohrNistor-Stokes}.
Once $P$ was shown to be invertible, we will conclude that $P^{-1}$ is also a pseudodifferential
operator using spectral invariance. Then we will apply the restriction theorem to $P^{-1}$
as explained.

\subsection{Spectral invariance}
We begin with an extension of a result in \cite{Mitrea-Nistor} which states that
$\ePS{\infty}$ is spectral invariant (i.e. it is stable under inversion, in a suitable sense).
See \cite{HAbelsSpi, BealsSpInv, CoriascoToft, DasgupaWongSpinv, ToftSpi, Schrohe19,
SchroheFrechet} for further results on spectral invariance. Most importantly for us, see the
paper by Mazzeo and Melrose \cite{MazzeoMelroseAsian} for the spectral invariance of the
$c$-calculus. We begin with some preliminary results and definitions.

Recall that $\maB(X)$ denotes the algebra of bounded, linear operators on a some normed space
$X$. Also, recall that a $C^*$-algebra is a Banach algebra with an involution $*$ that is isometrically
$*$-isomorphic to a closed, self-adjoint subalgebra of $\maB(\maH)$, for some Hilbert space $\maH$.
For any algebra $A$, we let $A + \CC = A$ if $A$ has a unit
and $A + \CC = A \oplus \CC 1$ (algebra with an adjoint unit) if $A$ does not have a unit.

\begin{definition}
  Let $A$ be a $C^*$-algebra and $j : B \to A$ a $*$-morphism.
  We say that \emph{$j$ is spectrum preserving} if it has the following property:
  \begin{center}
  ``If $b \in B$ and $\lambda \in \CC$ are such that $j(b) - \lambda$ is
  invertible in $A + \CC$, then $b - \lambda$ is invertible in  $B + \CC$.''
  \end{center}
  If this is the case, we shall also say that $B$ is \emph{spectrally invariant
  in $A$} (or, simply, that it \emph{has the spectral invariance property}).
\end{definition}

In applications, given $B$, the $C^*$-algebra $A$ is usually clear from
the context and, in any case, we can always enlarge $A$. For this reason, we shall usually
say that $B$ ``has the spectral invariance property'' instead of saying that ``it is
spectrally invariant in (some algebra) $A$.'' Also, in our applications,
$j$ will be usually an inclusion (i.e. injective). We shall use the following result from
\cite{LMNsi}.

\begin{proposition} \label{prop.sp.inv}
  Let us assume $J \subset B$ is a self-adjoint ideal. Let $j : B \to A$ be
  a $*$-morphism to a $C^*$-algebra $A$. If
  $j : J \to \overline{j(J)}$ and $j : B/J \to A/\overline{j(J)}$
  are spectrum preserving, then $j : B \to A$ is spectrum preserving.
\end{proposition}

We can now prove the second of the aforementioned crucial results.

\begin{theorem}\label{thm.spectral.inv}
  We use the notation of Definition \ref{def.ADN},
  in particular, $\manif$ is a manifold with compact, straight cylindrical ends and
  $E_i, F_i \to \manif$, $i = 1, \ldots, k$ are compatible vector bundles.
  Let $T\in \ePS {[\mathbf{s+t}]}(\manif;E,F)$, $s_i, t_j \geq 0$, and $m \in \RR$ be such that
  $T$ is invertible as a $($possibly unbounded$)$ operator on $H^{m-[\mathbf{s}]}(\manif;E)$
  with domain
  \begin{equation*}
    \maD(T) \ede \maD(T_{max}) \ede
    \{ \xi \in H^{m-[\mathbf{s}]}(\manif;E) \mid T\xi \in H^{m-[\mathbf{s}]}(\manif;F)\}\,.
  \end{equation*}
  If one of the $s_i$ or $t_j$ is $>0$, we assume also that $T$ is \ADN-elliptic.
  Then $T^{-1}\in \ePS {[-s-t]}(\manif; F, E)$.
\end{theorem}

We split the proof in two parts, for clarity. We begin with the case
$m = s_i = t_j = 0$, which was proved directly in \cite{Mitrea-Nistor} (that is,
without making appeal to Proposition \ref{prop.sp.inv}, but essentially reproving it).
Recall that, in this case, we need not assume $T$ to be elliptic.

\begin{proof}[Proof for $m = s_i = t_j = 0$]\
  Let us assume first that $E = F$.
  We need to prove that if $T \in \ePS{0}(\manif; E)$ is invertible on $L^2(\manif; E)$
  (as a bounded operator on that space), then $T^{-1} \in \ePS{0}(\manif; E)$.
  Let us consider in $\ePS{0}(\manif; E)$ the ideals
  \begin{equation}\label{eq.def.ideals}
    J_0 \ede \Psi_\maS^{-\infty}(\manif; E)\ \
    \mbox{ and }\ \ J_1 \ede \ePS{-\infty}(\manif; E)\,.
  \end{equation}

  We have that $J_0$ is spectrally invariant in $\maB(L^2(\manif; E))$
  since we can  characterize $J_0$ as
  \begin{equation}
    J_0 \seq \{T  \mid \, \rho^{k}P T Q \rho^{k} \in \maB(L^2(\manif; E)) \mbox{ for all }
    k \in \NN\,,\ P, Q \in \iPS{\infty}(\manif; E) \}\,.
  \end{equation}
  The verification of the spectral invariance property of $J_0$ is
  then a simple exercise (done, for instance, in
  \cite{LMNsi}, where more references can be found).

  The quotient $J_1/J_0 \simeq \ePSsus{-\infty} {\pa \manif_0; E}$ is isomorphic, via
  the Fourier transform to
  the algebra $$\maS(\RR) \hat{\otimes} \CI((\pa \manif_0)^2; E \boxtimes E') \simeq
  \maS(\RR) \hat{\otimes} \Psi^{-\infty}(\pa \manif_0; E)$$ (projective tensor
  product), which again can be proved to have the spectral invariance property directly.
  (The proof is either as for $\Psi_\maS^{-\infty}(\manif; E)$ or by a
  direct calculation. Also, notice that the product on $\maS(\RR)$ is the
  convolution product \emph{before} the Fourier transform and is the pointwise
  product \emph{after} the Fourier transform.) Let us compactify $\manif$
  as explained earlier to $\overline{\manif}$.
  Finally, we also have that the map
  $\sigma_0 : \ePS{0}(\manif; E)/ J_0 \to \maC(S^*\overline{\manif}; \End(E))$
  is trivially spectrum preserving. Two applications of Proposition
  \ref{prop.sp.inv} give that $\ePS{0}(\manif; E)$ has the spectral invariance
  property.

  This proves the result for $E = F$. For the general case, we use the result
  just proved for $T^*T$ and $TT^*$.
\end{proof}

We shall use the result of Proposition \ref{prop.sp.inv} to prove that the order
reduction operators of Corollary \ref{cor.order.red} are essentially translation
invariant operators.

\begin{corollary}\label{cor.order.red2}
  Let $s \in \RR$. There exists $\Lambda_s \in \iPS{s}(\manif; E)$ such that the induced
  map $\Lambda_s : H^s(\manif; E) \to L^2(\manif; E)$ is an isomorphism and $\Lambda_s^{-1} \in
  \ePS{-s}(\manif; E)$.
\end{corollary}

\begin{proof}
  Let $Q$ and $P:=\Lambda _s$ be as in the proof of Corollary \ref{cor.order.red}.
  Then $QP$ and $PQ$ are invertible in $\iPS{0}(\manif; E)$. Therefore $(QP)^{-1}, (PQ)^{-1} \in
  \ePS{0}(\manif; E)$. We thus have that $(QP)^{-1} Q \in \ePS{-s}(\manif; E)$
  is a left inverse of $P$ and that $Q (PQ)^{-1} \in \ePS{-s}(\manif; E)$ is
  a right inverse of $P$. Therefore $P$ is invertible and its inverse is
  $$P^{-1} \seq (QP)^{-1} Q \seq Q (PQ)^{-1} \in \ePS{-s}(\manif; E)\,.$$
  This completes the proof.
\end{proof}

We shall choose, for each $s \in \RR$, an operator $\Lambda_s \in \iPS{s}(\manif; E)$ as in
Corollary \ref{cor.order.red2} (thus, an invertible isomorphism of Sobolev spaces).
We can now complete the proof of Theorem \ref{thm.spectral.inv}.

\begin{proof}[Proof of Theorem \ref{thm.spectral.inv} for $s_i$ and $t_j$ general]
  We shall assume $m, s_i, t_j \in \ZZ$, for the simplicity of the notation
  (we replace $\ePS{\infty}(\manif; E, F)$ with $\sum_{s \in \RR} \ePS{s}(\manif; E, F)$ 
  in general). Let $T$ be the given operator assumed to be invertible.
  If $s_i = t_j = 0$ for all $i, j = 1, \ldots, k$, then $T \in \ePS{0}(\manif; E, F)$
  and hence $T$   is bounded on $H^m(\manif; E)$ and we consider the
  bounded operator $T : H^m(\manif; E) \to H^m(\manif; F)$, whose inverse we want
  to prove to be in $\ePS{0}(\manif; F, E)$. In case at least one of the $s_i$ or $t_j$ 
  are $>0$, we use elliptic regularity (Proposition \ref{prop.regularity})
  to conclude that the domain of $T$ is $H^{m+[\mathbf{t}]}(\manif; E)$. Thus 
  the assumption is
  that
  \begin{multline}  \label{eq.alt.assumpt}
    T \seq [T_{ij}] : H^{m+[\mathbf{t}]}(\manif; E) \ede \displaystyle\bigoplus _{j=1}^kH^{m+t_j}(\manif ;E_j)
    \\ \to
    \displaystyle\bigoplus _{i=1}^kH^{m-s_i}(\manif ;F_i) \, =: \, H^{m-[\mathbf{s}]}(\manif; F)
  \end{multline}
  is invertible and we want to prove that its inverse is in $\ePS{[\textbf{-t-s}]}(\manif; F, E)$.

  We shall use the
  order reduction operators $\Lambda_{t_i} \in \ePS{t_i}(\manif; E_i)$, $i=1,\ldots ,k$, constructed as in Corollary
  \ref{cor.order.red2}. Let
  \begin{align}
    \label{D-si}
    S_{t_1,\ldots ,t_k} \ede \left(\begin{array}{cccc}
    \Lambda_{t_1} & 0 & \cdots & 0\\
    0 & \Lambda_{t_2} & \cdots & 0 \\
    \vdots & \vdots & \ddots & \vdots \\
    0 & 0 & \cdots & \Lambda_{t_k}\\
    \end{array}
    \right)\,.
  \end{align}
  Then $S_{t_1,\ldots ,t_k} \in \ePS{\infty}(\manif; E)$ and is invertible
  there. The operator
  $$Q_0 \ede S_{m + t_1, m + t_2, \ldots, m + t_k}^{-1} : L^2(\manif; E) \to
  H^{m+[\mathbf{t}]}(\manif; E)$$
  is an isomorphism. Similarly (but with $E_i$ replaced with $F_i$),
  $$Q_1 \ede S_{m - s_1, m - s_2, \ldots, m - s_k} : H^{m-[\mathbf{s}]}(\manif; F) \to
  L^2(\manif; F)$$
  is an isomorphism, where $Q_1$ is defined similarly, but using the vector
  bundles $F_j$. Therefore,
  \begin{equation*}
    T \seq [T_{ij}] : H^{m+[\mathbf{t}]}(\manif; E) \to H^{m-[\mathbf{s}]}(\manif; F)
  \end{equation*}
  is invertible if, and only, if the operator
  \begin{equation*}
    Q_1 T Q_0 : L^2(\manif ;E)\to L^2(\manif ;F)
  \end{equation*}
  is invertible. We have $Q_1 T Q_0 \in \ePS{0}(\manif; E, F)$, and hence we can use
  the already proved case of our theorem to conclude that $(Q_1 T Q_0)^{-1} \in \ePS{0}(\manif; F, E)$.
  This gives $T^{-1} = Q_0 (Q_1 T Q_0)^{-1}Q_1 \in \ePS{[\textbf{-t-s}]}(\manif; F, E)$, as
  desired.
\end{proof}

We shall need also the following simple remark:

\begin{remark}\label{rem.also}
  In Theorem \ref{thm.spectral.inv} for the case one of the $s_i$ or $t_j$ is $>0$,
  instead of ellipticity and invertibility as an unbounded operator
  (with a domain implicitly given), we can specify explicitly the domain. More
  precisely, we assume that $T \in \ePS{[\mathbf{s+t}]}(\manif; E, F)$ is such that
  $T \seq [T_{ij}] : H^{m+[\mathbf{t}]}(\manif; E)
  \to H^{m-[\mathbf{s}]}(\manif; F)$ is invertible (see the discussion
  surrounding Equation \eqref{eq.alt.assumpt}).
  Then we still obtain that $T^{-1}\in \ePS {[\textbf{-s-t}]}(\manif; F, E)$.
  This follows by examining the proof (which is unchanged).
\end{remark}

\subsection{Definition of layer potential operators}
\label{ssec.layer.pot}

\begin{definition} \label{def.layerPot}
  Let $\manif$ be a manifold with straight cylindrical ends and $N \subset \manif$
  be such that $N \cap \pa \manif_0 \times (-\infty, -R) = N_{\infty} \times (-\infty, -R)$
  for some $R > 0$ and some smooth submanifold $N_{\infty}$ of the closed manifold 
  $\pa \manif_0$. Let $P \in \ePS{q}(\manif; E, F)$ be a classical pseudodifferential 
  operator such that $P : H^q(\manif; E) \to L^2(\manif; E)$ is invertible. Let us 
  assume that $q$ is larger than the codimension of $N_{\infty}$ in $\pa \manif_0$.
  Then \emph{the single layer potential operator 
  $S_P$ associated to $P$} is the restriction of $P^{-1}$ to $N$. 
\end{definition}

The results of the last section show that $S_P$ is a pseudodifferential operator
of order $\le \dim M - \dim N - q = n - \dim N - q < 0$. The double layer potential operator 
is defined when $N$ is of codimension one with oriented normal bundle. Let us choose a normal 
unit vector to $N$, which exists by the orientability hypothesis.

\begin{definition} \label{def.layerPot2}
  We use the notation of Definition \ref{def.layerPot}, but assume that $N$ is of 
  codimension one and that it has a normal unit vector $\nu$, which we extend to a 
  smooth vector field defined everywhere. 
  Also, we assume that $P$ is an order $2$ pseudodifferential operator such that 
  $P : H^2(\manif; E) \to L^2(\manif; E)$ is invertible. Then \emph{the double
  layer potential operator $K_P$ associated to $P$} is the operator with kernel 
  $(\pa_{\nu_y} k_{P^{-1}}(x, y))_{N \times N}$ (the restriction of the normal derivative 
  with respect to the second variable kernel to $N$), defined 
  in a principal value sense (when it converges).
\end{definition}

To deal with the double layer potential operator $K$, we will
use the following remark.

\begin{remark}\label{rem.for.K}
  In theorem \ref{thm.kernel.restriction}, we can replace the condition $s_i + t_j < -q$
  with the condition $s_i + t_j < -q+1$ if the restriction to $N \times N$ of the highest
  order singularity $a_{0}$ vanishes on $SN$ (we use the notation of Proposition
  \ref{prop.kernel.descr}). It is know that this applies to the double layer potential 
  (associated to the Laplacian) if the
  boundary is smooth. In that case, the resulting operator will be of lower order.
\end{remark}

Layer potential operators on manifolds with conical points were recently studied 
by Carvalho, C\^ome, and Qiao in \cite{CCQ_proc, CarvalhoQiao}.

\subsection{Fredholm conditions}

The next theorem (the third of aforementioned three crucial results) extends the Fredholm conditions 
on closed manifolds to our setting. It is due to Kondratiev \cite{Kondratiev} (in the setting of 
differential operators) and to Melrose and Mendoza \cite{Melrose-Mendoza} and Schulze \cite{Schulze} 
in general for pseudodifferential operators in the $b$-calculus. Most importantly for us, the case of 
the $c$-calculus was treated by Mazzeo and Melrose in \cite{MazzeoMelroseAsian}. See also 
\cite{CNQ, DegeratuMazzeo, Kapanadze-Schulze, SchroheSpInvFred, SchroheFrechet} and others, who have 
proved various extensions of this result. 

Recall the definition of the principal symbol $\sigma_m^{\RR}$ for translation invariant 
operators in $\ePSsus{m} {\pa \manif_0; E}$ in Remark \ref{rem.boundary.restriction}. For 
the following result, we shall need  the algebra
\begin{multline}\label{eq.def.fibered.prod}
  \mathfrak{A} \ede \{ (P,  a) \in \ePSsus{m} {\pa \manif_0; E} \oplus
  \CI_{\inv}(S^*\manif; \End(E)) \mid\\
  \sigma_m^\RR (P) = \In(a) \in \CI(\pa S^*\manif_0; \End(E))\,,\ P \ 
  \mbox{ classical } \}\,.
\end{multline}

\begin{proposition}\label{prop.rem.deFacut}
  The joint map 
  \begin{equation*}
    (\In, \sigma_m) : \ePS{m}(\manif; E) \cap \Psi_{cl}^{m}(\manif; E)\to \mathfrak{A}
  \end{equation*}
  is surjective with kernel $\Psi_\maS^{-\infty}(\manif; E) + \big (\Psi_{comp}^{m-1}(\manif; E)
  \cap \Psi_{cl}^{m-1}(\manif; E) \big )$. The analogous result holds for 
  $\iPS{m}(\manif; E)$ with the obvious modifications: $P \in \iPSsus{m}{\manif; E}$
  and the kernel is $\Psi_{comp}^{m-1}(\manif; E)
  \cap \Psi_{cl}^{m-1}(\manif; E)$.
\end{proposition}

\begin{proof}
  All the pseudodifferential operators in this proof are classical. 
  Let $(P, a) \in \mathfrak{A}$. Using the surjectivity of the principal symbol, 
  Theorem \ref{thm.prop.inv-calc}, we know that there exists $Q \in \ePS{m}(\manif; E)$ such 
  that $\sigma_m(Q) = a$. Let $R := \In(Q) - P$. We have 
  \begin{align*}
    \sigma_m^{\RR}(R) & = \sigma_m^{\RR}(\In(Q) - P)\\
    & = \In (\sigma_m(Q)) - \sigma_m(P)\\
    & = \In (a) - \sigma_m(P) = 0\,,
  \end{align*}
  where the second equality is by Theorem \ref{thm.complet.indicial}(4).
  Therefore $R \in \ePSsus{m-1}{\manif; E}$, and hence $s_0(R) \in \ePS{m-1}(\manif; E)$.
  The operator $Q_1 := Q - s_0(R)$ then satisfies
  \begin{equation*}
    \begin{gathered}
      \sigma_m(Q_1) \seq \sigma_m(Q - s_0(R)) \seq \sigma_m(Q) \seq a\,.\\
      \In(Q_1) \seq \In(Q) - R \seq P\,.
    \end{gathered}
  \end{equation*}
  This proves the desired surjectivity for essentially translation invariant 
  operators. The form of the kernel follows from the Theorems \ref{thm.prop.ePS}(3)
  and \ref{thm.complet.indicial}.
  The proof for operators that are translation invariant in a neighborhood 
  of infinity is the same.
\end{proof}

We can now characterize the compact operators in our 
essentially translation invariant calculus.

\begin{proposition} \label{prop.compact}
  Let $\manif$ be a manifold with compact, straight cylindrical ends and $E, F \to \manif$
  be compatible vector bundles. Let $T \in \ePS{0}(\manif; E, F)$ be a \emph{classical} 
  pseudodifferential operator. The following three conditions are equivalent:
  \begin{enumerate}
    \item The operator $T$ is compact;

    \item $\sigma_0(T) = 0$ and $\In(T) = 0$;

    \item $T \in \Psi_\maS^{-\infty}(\manif; E, F) + \Psi_{comp}^{-1}(\manif; E, F)$.
  \end{enumerate}
\end{proposition}

\begin{proof}
  The equivalence of (2) and (3) follows from Proposition \ref{prop.rem.deFacut}.

  (3) $\Rightarrow$ (1):\
  Let us prove that each of the spaces $\Psi_\maS^{-\infty}(\manif; E, F)$ and
  $\Psi_{comp}^{-1}(\manif; E, F)$ act as compact operators from $L^2(\manif; E)$
  to $L^2(\manif; F)$. 
  Let $T \in \Psi_{comp}^{-1}(\manif; E, F)$. Then $T$ has a compactly supported distribution
  kernel, by the very definition of $\Psi_{comp}^{-1}(\manif; E, F)$ (Definition \ref{def.c.supp}).
  Therefore $T : L^2(\manif; E) \to L^2(\manif; F)$ is compact by Proposition \ref{prop.compact0}. Since
  $\Psi_{comp}^{-\infty}(\manif; E, F) \simeq \CIc(\manif^2; F \boxtimes E')$
  is dense in $\Psi_\maS^{-\infty}(\manif; E, F) \simeq \maS(\manif^2; F \boxtimes E')$ in the natural
  (Schwartz) topology on the second space and hence also in the norm of operators, it follows that
  $\Psi_\maS^{-\infty}(\manif; E, F)$ also consists of compact operators.
  So (3) implies (1).

  (1) $\Rightarrow$ (2):\ Let us assume that $T \in \ePS{0}(\manif; E, F)$ is compact.
  Let $\phi \in \CIc(\manif)$. Then $\phi T \phi \in \Psi_{comp}^0(\manif; E, F)$ is
  also compact. Proposition \ref{prop.compact0} then gives $\phi \sigma_0(T)\phi =
  \sigma_0(\phi T \phi) = 0$. Hence $\sigma_0(T) = 0$, since $\phi \in \CIc(\manif)$ was arbitrary.
  Next, $$\In(T) u \seq \lim_{s \to \infty} \Phi _{-s}\big [ T\Phi_s(u)\vert_{\{t < 0\}} \big] \seq 0$$
  since $\Phi_s(u) \to 0$ weakly in $L^2(\manif; E)$ and $T$ is compact
  (see Lemma \ref{lemma.def.indicial2}). Thus (1) implies (2) and the proof is complete.
\end{proof}

We now obtain the following characterization of the Fredholm operators in the essentially translation
invariant calculus.

\begin{theorem}\label{thm.ADN.Fredholm}
  We keep the assumptions and the notations of Theorem \ref{thm.spectral.inv}
  and of Definition \ref{def.ADN}. In particular, $\manif$ is a manifold with compact, 
  straight cylindrical ends, $E_i, F_i \to \manif$, $i = 1, \ldots, k$ are compatible vector bundles,
  and $E = \oplus_{i=1}^k E_i$ and $F = \oplus_{j=1}^k F_j$.
  Let $T=[T_{ij}] \in \ePS{[\mathbf{s+t}]}(\manif; E, F)$ be a \emph{classical operator}
  and $m \in \RR$. The map $T : H^{m+[\mathbf{t}]}(\manif; E) \to H^{m-[\mathbf{s}]}(\manif; F)$
  is Fredholm if, and only, if
  \begin{enumerate}
    \item $T$ is \ADN-elliptic and
    \item the limit operator
    \begin{equation*}
      \In(T) \ede [\In(T_{ij})] : H^{m+[\mathbf{t}]}(\pa \manif_0 \times \RR; E) \to
      H^{m-[\mathbf{s}]}(\pa \manif_0 \times \RR; F)
    \end{equation*}
    is invertible.
  \end{enumerate}
\end{theorem}

This result is, of course, well-known in the setting of the $b$-calculus
\cite{MelroseAPS, Melrose-Mendoza}.

\begin{proof}
  Let us assume first that $m = s_i = t_i = 0$, $i = 1, \ldots, k$, as in the
  proof of Theorem \ref{thm.spectral.inv}. We shall use Theorem \ref{thm.prop.ePS} repeatedly.

  Let us assume next that $T : L^2(\manif; E) \to L^2(\manif; F)$ is Fredholm.
  We know that $T$ is elliptic by Theorem \ref{thm.Fredholm.compact}(2)
  (but we shall obtain this fact directly here from Proposition \ref{prop.compact}).
  Then $T^*$ is Fredholm and hence $0 \le T^*T \in \ePS{0}(\manif; E)$ is also Fredholm.
  We can then find $R \in \Psi_\maS^{-\infty}(\manif; E,F)$ such that $T^*T + R^*R
  \in \ePS{0}(\manif; E)$ is invertible as an unbounded operator on $L^2(M; E)$. Let
  $$P \ede (T^*T + R^*R)^{-1} T^* \in \ePS{0}(\manif; F, E) \,.$$
  A direct calculation based on the symbolic properties of the $\ePS{\infty}$ calculus gives that
  $$P T - 1 \seq -(T^*T + R^*R)^{-1} R^*R \in \Psi_\maS^{-\infty}(\manif; E)\,$$
  see Corollary \ref{cor.comp.alg}(1), Lemma \ref{lemma.comp.alg}(2), and
  Theorem \ref{thm.prop.ePS}(1). Since $R$ is compact, by Proposition \ref{prop.compact},
  we obtain that $P T - 1$ is also compact. Proposition \ref{prop.compact}
  implies
  \begin{equation*}
    \begin{gathered}
      \sigma_0(P) \sigma_0(T)  -1 \seq \sigma_0(PT-1 ) \seq 0 \quad \mbox{and}\\
      \In(P)\In(T) - 1 \seq \In(PT- 1) \seq 0\,.
    \end{gathered}
  \end{equation*}
  Therefore, both $\sigma_0(T)$ and $\In(T)$ are left invertible. Similarly, they are
  right invertible, and hence invertible. This proves one implication of the desired equivalence.

  To prove the other implication, let us assume that $\sigma_0(T) \in \CI_{\inv}(S^*\manif; \Hom(E, F))$ 
  and $\In(T)\in \ePSsus{0}{\manif; E, F}$ are invertible. Then $\In(T)^{-1} \in \ePSsus{0}{\manif; F, E}$ 
  by the spectral invariance theorem, Theorem \ref{thm.spectral.inv}. Similarly, Corollary 
  \ref{cor.lemma.ue} gives that $\sigma_0(T)^{-1} \in \CI_{\inv}(\manif; \Hom(F, E))$. 
  Theorem \ref{thm.complet.indicial}(4) then gives
  \begin{multline*}
    \sigma_{0}^{\RR}(\In(T)^{-1}) \seq \sigma_{0}^{\RR}(\In(T))^{-1} \seq \In(\sigma_0(T))^{-1}\\
    \seq \In(\sigma_0(T)^{-1}) \in \CI(\pa S^*\manif_0; \Hom(F, E))\,.
  \end{multline*} 
  Proposition \ref{prop.rem.deFacut} then implies that there 
  exists $Q \in \ePS{0}(\manif; E, F)$ such that $\sigma_0(Q) = \sigma_0(T)^{-1}$
  and $\In(Q) = \In(T)^{-1}$. Consequently,
  \begin{equation*}
    \sigma_0(PQ) \seq 1 \,,\ \sigma_0(QP) \seq 1\,,\ \In(PQ) \seq 1\,
    \ \mbox{ and } \ \In(QP) \seq 1\,.
  \end{equation*}
  Therefore $QP -1 \in \Psi_{\maS}^{-\infty}(\manif; E) + \Psi_{comp}^{-1}(\manif; E)$, by 
  Proposition \ref{prop.rem.deFacut}, and hence $QP - 1$ is compact, by Proposition 
  \ref{prop.compact}. Similarly, $PQ - 1$ is also compact. Atkinson's theorem 
  \ref{thm.Atkinson} then allows us to conclude that $P$ is Fredholm.
  This proves the result if $m = s_i = t_i = 0$ for all $i = 1, \ldots, k$.

  The proof of the result in general is the same as the proof of the general case of
  Theorem \ref{thm.spectral.inv}, that is, using the order reduction operators $Q_0$ and $Q_1$
  obtained from the operators $S_{s_1,\ldots ,s_k}$ of Equation \eqref{D-si}
  as in that proof (we use the same formulas for $Q_0$ and $Q_1$).
  We first observe that the multiplicativity of $\Symb$ and $\In$, the invertibility of
  $Q_0$ and $Q_1$, Theorem
  \ref{thm.complet.indicial}, and Remark \ref{rem.bounded.ADN}(3), imply that
  all of $\Symb(Q_0)$, $\Symb(Q_1)$, $\In(Q_0)$, and $\In(Q_1)$ are invertible
  in their respective spaces (the first two on $S^*\manif$ and the last two as
  bounded operators on $L^2(\manif; E)$ and, respectively, on $L^2(\manif; F)$).
  We also use the relations
  $$\sigma_0(Q_1 T Q_0) = \Symb(Q_1)\Symb(T)\Symb(Q_0)$$
  and $$\In(Q_1 T Q_0) = \In(Q_1)\In(T)\In(Q_0)$$ (again from Theorem
  \ref{thm.complet.indicial} and Remark \ref{rem.bounded.ADN}(3))
  and the bounded case already proved (in the second equivalence)
  to successively obtain
  \begin{align*}
  T : H^{m+[\mathbf{t}]}(\manif; E) \to  H^{m-[\mathbf{s}]}(\manif; F) & \mbox{ is Fredholm }\\
    & \Leftrightarrow Q_1 T Q_0 :L^2(\manif; E) \to L^2(\manif; F) \ \mbox{ is Fredholm }\\
    & \Leftrightarrow \sigma_0(Q_1 T Q_0) \mbox{ and } \ \In(Q_1 T Q_0)\ \mbox{ are invertible}\\
    & \Leftrightarrow \Symb(T) \mbox{ and } \ \In(T)\ \mbox{ are invertible\,.}
  \end{align*}
  This completes the proof because $\Symb(Q_0)$, $\Symb(Q_1)$, $\In(Q_0)$,
  and $\In(Q_1)$ are invertible (because $Q_0$ and $Q_1$ are invertible).
\end{proof}

\def\cprime{$'$}

\end{document}